\documentclass[11pt]{amsart}

\usepackage{amsfonts} 
\usepackage{amsmath} 
\usepackage{amssymb} 
\usepackage{amsthm} 
\usepackage{mathtools}
\usepackage{amscd}
\usepackage{booktabs}
\usepackage{color}
\usepackage[dvipsnames]{xcolor}
\usepackage{import}

\usepackage[margin=3cm]{geometry} 

\usepackage[hidelinks,pagebackref,pdftex]{hyperref} 
\usepackage{makecell} 
\usepackage{multicol}

\usepackage{kbordermatrix} 

\renewcommand*{\backref}[1]{}
\renewcommand*{\backrefalt}[4]{
  \ifcase #1
  [No citations.]
  \or [#2]
  \else [#2]
  \fi }

\usepackage{graphicx} 

\AtBeginDocument{%
   \def\MR#1{}
}

\numberwithin{equation}{section}
\theoremstyle{plain}
\newtheorem{theorem}[equation]{Theorem}
\newtheorem{thm}[equation]{Theorem}
\newtheorem{corollary}[equation]{Corollary}
\newtheorem{lemma}[equation]{Lemma}
\newtheorem{lem}[equation]{Lemma}

\newtheorem{prop}[equation]{Proposition}

\newtheorem*{namedtheorem}{\theoremname}
\newcommand{\theoremname}{testing}

\theoremstyle{definition}

\newtheorem{defn}[equation]{Definition}
\newtheorem{remark}[equation]{Remark}

\newcommand{\To}{\longrightarrow}

\newcommand{\C}{\mathbb{C}}
\newcommand{\CC}{\mathbb{C}}

\newcommand{\CP}{\mathbb{CP}}

\newcommand{\e}{\mathbf{e}}

\newcommand{\f}{\mathbf{f}}
\newcommand{\F}{\mathcal{F}}

\newcommand{\HH}{{\mathbb{H}}}

\newcommand{\half}{\frac{1}{2}}

\newcommand{\mfc}{\mathfrak{c}}
\newcommand{\mfl}{\mathfrak{l}}
\newcommand{\mfm}{\mathfrak{m}}

\newcommand{\Q}{\mathbb{Q}}

\newcommand{\R}{\mathbb{R}}
\newcommand{\RR}{\mathbb{R}}

\newcommand{\T}{\mathbb{T}}
\newcommand{\TT}{\mathcal{T}}

\newcommand{\Z}{\mathbb{Z}}

\newcommand{\calL}{\mathcal{L}}
\newcommand{\calM}{\mathcal{M}}

\DeclareMathOperator{\Span}{Span}

\newcommand{\In}{\mathrm{In}}
\newcommand{\NZ}{\mathrm{NZ}}
\newcommand{\SY}{\mathrm{SY}}
\newcommand{\PSL}{{\mathrm{PSL}}}
\newcommand{\SL}{{\mathrm{SL}}}
\newcommand{\Sp}{{\mathrm{Sp}}}
\newcommand{\bdy}{\partial}

\newcommand{\refsec}[1]{Section~\ref{Sec:#1}}
\newcommand{\refdef}[1]{Definition~\ref{Def:#1}}
\newcommand{\reffig}[1]{Figure~\ref{Fig:#1}}

\newcommand{\refeqn}[1]{\eqref{Eqn:#1}}
\newcommand{\reflem}[1]{Lemma~\ref{Lem:#1}}
\newcommand{\refprop}[1]{Proposition~\ref{Prop:#1}}
\newcommand{\refthm}[1]{Theorem~\ref{Thm:#1}}
\newcommand{\refcor}[1]{Corollary~\ref{Cor:#1}}

\begin{document}

\title{A-polynomials, Ptolemy equations and Dehn filling} 

\author{Joshua A. Howie}
\address{School of Mathematics,
Monash University,
VIC 3800, Australia}
\email{josh.howie@monash.edu}

\author{Daniel V. Mathews} 
\address{School of Mathematics,
Monash University,
VIC 3800, Australia}
\email{Daniel.Mathews@monash.edu}

\author{Jessica S. Purcell}
\address{School of Mathematics,
Monash University,
VIC 3800, Australia}
\email{jessica.purcell@monash.edu}


\begin{abstract}
The A-polynomial encodes hyperbolic geometric information on knots and related manifolds. Historically, it has been difficult to compute, and particularly difficult to determine A-polynomials of infinite families of knots. Here, we compute A-polynomials by starting with a triangulation of a manifold, then using symplectic properties of the Neumann-Zagier matrix encoding the gluings to change the basis of the computation. The result is a simplicifation of the defining equations.  We apply this method to families of manifolds obtained by Dehn filling, and show that the defining equations of their A-polynomials are Ptolemy equations which, up to signs, are equations between cluster variables in the cluster algebra of the cusp torus. 
\end{abstract}

\maketitle

\section{Introduction}

The A-polynomial is a polynomial associated to a knot that encodes a great deal of geometric information. It is closely related to deformations of hyperbolic structures on knots, originally explored by Thurston~\cite{thurston}. Such deformations give rise to a one complex parameter family of representations of the knot group into $\SL(2,\CC)$. All representations form the \emph{representation variety}, which was originally studied in pioneering work of Culler and Shalen~\cite{CullerShalen:Varieties, CullerShalen:Surfaces,CGLS:DehnSurgery}, and remains a very active area of research; see \cite{shalen:survey} for a survey. However representation varieties are difficult to compute, and often have complicated topology. In the 1990s, Cooper, Culler, Gillet, Long and Shalen realised that a representation variety could be projected onto $\CC^2$ using the longitude and meridian of the knot~\cite{CCGLS}, with a simpler image. The image is given by the zero set of a polynomial in two variables, up to scaling. This is the A-polynomial.

Among its geometric properties, the A-polynomial detects many incompressible surfaces, and gives information on cusp shapes and volumes~\cite{CCGLS,CooperLong:Apoly}. It has relations to Mahler measure~\cite{Boyd:Mahler}, and appears in quantum topology through the AJ-conjecture~\cite{garoufalidis:AJ-conj, garoufalidisLe, FrohmanGelcaLofaro}. 
Unfortunately, A-polynomials are also difficult to compute. Unlike other knot polynomials, there are no skein relations to determine them. Originally, they were computed by finding polynomial equations from a matrix presentation of a representation, and then using resultants or Groebner bases to eliminate variables; see, for example~\cite{CooperLong:Apoly}. Unlike other knot polynomials, they are known only for a handful of infinite examples, including twist knots, some double twist knots, and small families of 2-bridge knots~\cite{HosteShanahan04,Mathews:A-poly_twist_knots, Mathews:A-poly_twist_knots_err, HamLee, Petersen:DoubleTwist, Tran}, some pretzel knots~\cite{TamuraYokota,GaroufalidisMattman}, and cabled knots and iterated torus knots~\cite{NiZhang}. Culler has computed A-polynomials for all knots with up to eight crossings, most nine-crossing knots, many ten-crossing knots, and all knots that can be triangulated with up to seven ideal tetrahedra~\cite{Culler}. 

This paper gives a simplified method for determining A-polynomials, especially for infinite families of knots obtained by Dehn filling. Our method is to change the variables in the defining equations. Typically, defining equations for A-polynomials have high degree in the variables to eliminate, making them computationally difficult. Under a change of variables, we show that all such equations can be expressed in degree two in the variables to eliminate. For families of knots obtained by Dehn filling, even more can be said. There will be a finite, fixed number of ``outside equations'', and a sequence of equations determined completely by the slope of the Dehn filling. All such equations exhibit Ptolemy-like properties, with very similar behaviours to cluster algebras. We expect the method to greatly improve our ability to compute families of A-polynomials. Indeed, of all the known examples of infinite families of A-polynomials above, all except the cabled knots and iterated torus knots are obtained by Dehn filling a fixed parent manifold. 

\subsection{Computing the A-polynomial}
Champanerkar introduced a geometric way to compute the A-polynomial based on a triangulation of a knot complement~\cite{Champanerkar:Thesis}. His method is to start with a collection of equations --- one gluing equation for each edge of the triangulation, and two equations for the cusp --- and eliminate variables. 
The coefficients in the gluing and cusp equations are effectively the entries in the Neumann--Zagier matrix~\cite{NeumannZagier}. This matrix has interesting symplectic properties: its rows form part of a standard symplectic basis for a symplectic vector space. Dimofte~\cite{DimofteQRCS,DvdV:Spectral} considered extending this collection of vectors into a standard basis for $\RR^{2n}$, and then changing the basis. This yields a change of variables, and an equivalent set of equations. Eliminating variables again yields (up to technicalities) the A-polynomial; effectively this can be considered a process of symplectic reduction.

There are a few issues with Dimofte's calculations that have made them difficult to use in practice. First, the result appears in physics literature, which makes it somewhat difficult for mathematicians to read. More importantly, to carefully perform the change of basis, in particular to nail down the correct signs in the defining equations, a priori one needs to determine the symplectic dual vectors to the vectors arising from gluing equations. These are not only nontrivial to compute, but also highly non-unique. Only after obtaining such vectors can one invert a large symplectic matrix. 

In this paper, we overcome these issues. Using work of Neumann~\cite{Neumann}, we show that we may ``invert without inverting.'' That is, we show that Dimofte's symplectic reduction can be read off of ingredients already present in the Neumann--Zagier matrix, without having to compute symplectic dual vectors. As a result, we may convert Champanerkar's (possibly complicated) equations into simpler equations that have Ptolemy-like structure.

There are other ways to compute A-polynomials. Zickert introduced one in his work on extended Ptolemy varieties~\cite{Zickert:PtolemyDehnA-poly, GTZ}, inspired by Fock and Goncharov~\cite{FockGoncharov06}. Their work also starts with a triangulation, but in the case of interest assigns six variables per tetrahedron, and relates these by what are called Ptolemy relations and identification relations. After an appropriate transformation, the corresponding variables satisfy gluing equations; see \cite[Section~12]{GTZ}. Zickert notes a ``fundamental duality'' between Ptolemy coordinates and gluing equations in  \cite[Remark~1.13]{Zickert:PtolemyDehnA-poly}. However, it is not clear why the duality arises. The equations we find in this paper are similar to the defining equations of Zickert, but with fewer variables. We expect that the results of this paper may provide a connection to two very different approaches to calculating A-polynomials. While we do not show that the methods of that paper and this one are equivalent, we conjecture that they are, and thus the techniques here may provide a geometric, symplectic explanation for the ``fundamental duality''.

\subsection{Neumann--Zagier matrices and the main theorem}

Let $M$ be a hyperbolic 3-manifold with a triangulation. Then it has an associated Neumann--Zagier matrix, which we will denote by $\NZ$. The properties of $\NZ$ are reviewed in Section~\ref{Sec:symplectic}. 
In short, gluing and cusp equations give a system of the form $\NZ \cdot Z = H+i\pi C$, where $Z$ is a vector of variables related to tetrahedra, and $H$ and $C$ are both vectors of constants. 

Neumann and Zagier showed that if $M$ has one cusp, then the $n$ rows of $\NZ$ corresponding to gluing equations have rank $n-1$. Thus a row can be removed, leaving $n-1$ linearly independent rows. Denote the matrix given by removing such a row of $\NZ$ by $\NZ^{\flat}$, and similarly denote the vector obtained from $C$ by removing the corresponding row by $C^{\flat}$. We will refer to $\NZ^{\flat}$ as the \emph{reduced Neumann--Zagier matrix}. The vector $C^\flat$ is called the \emph{sign vector}.
We will show that, after possibly relabelling the tetrahedra of a triangulation, we may assume one of the entries of $C^{\flat}$ corresponding to a gluing equation is nonzero. 
Neumann has shown that there always exists an integer vector $B$ such that $\NZ^{\flat}\cdot B = C^{\flat}$ ~\cite{Neumann}.

To state the main theorem, we introduce a little more notation.
The last two rows of the
matrix $\NZ^{\flat}$ correspond to cusp equations associated to the meridian and longitude. For ease of notation, we will denote the entries in the row associated to the meridian and longitude, respectively, by
\[ \begin{pmatrix} \mu_1 & \mu_1' & \mu_2 & \mu_2' & \dots
  \end{pmatrix} \mbox{ and }
\begin{pmatrix} \lambda_1 & \lambda_1' & \lambda_2 & \lambda_2' & \dots
\end{pmatrix}.
\] 
Finally, suppose the edges of the tetrahedra are glued into $n$ edges $E_1, \dots, E_n$. Label the ideal vertices of each tetrahedron $0$, $1$, $2$, and $3$, with $1$, $2$, $3$ in anti-clockwise order when viewed from $0$. Then there are six edges, each labeled by a pair of integers $\alpha\beta\in \{01,02,03,12,13,23\}$. For the $j$th tetrahedron, let $j(\alpha\beta)$ denote the index of the edge class to which that edge is identified. That is, if the edge $\alpha\beta$ is glued to $E_k$, then $j(\alpha\beta)=k$. 

\begin{theorem}\label{Thm:MainPtolemy}
Let $M$ be a one-cusped manifold with a hyperbolic triangulation $\TT$, with associated reduced Neumann--Zagier matrix $\NZ^{\flat}$ and sign vector $C^{\flat}$ as above.
Also as above, denote the entries of the last two rows of $\NZ^\flat$ by $\mu_j, \mu_j'$ in the row corresponding to the meridian, and $\lambda_j$, $\lambda_j'$ in the row corresponding to the longitude. Let $B=(B_1, B_1', B_2, B_2', \dots)$ be an integer vector such that $\NZ^{\flat}\cdot B = C^{\flat}$.

Define formal variables $\gamma_1, \dots, \gamma_n$, one associated with each edge of $\TT$. For a tetrahedron $\Delta_j$ of $\TT$, and edge $\alpha\beta\in \{01,02,03,12,13,23\}$, define $\gamma_{j(\alpha\beta)}$ to be the variable $\gamma_k$ such that the edge of $\Delta_j$ between vertices $\alpha$ and $\beta$ is glued to the edge of $\TT$ associated with $\gamma_k$. 

For each tetrahedron $\Delta_j$ of $\TT$, define the \emph{Ptolemy equation} of $\Delta_j$ by
\[
\left( -1 \right)^{B'_j} \ell^{-\mu_j/2} m^{\lambda_j/2} \gamma_{j(01)} \gamma_{j(23)}
+
\left( - 1 \right)^{B_j} \ell^{-\mu'_j/2} m^{\lambda'_j/2} \gamma_{j(02)} \gamma_{j(13)}
-
\gamma_{j(03)} \gamma_{j(12)}
= 0.
\]

When we solve the system of Ptolemy equations of $\TT$ in terms of $m$ and $\ell$, setting $\gamma_n = 1$ and eliminating the variables $\gamma_1, \ldots, \gamma_{n-1}$, we obtain a factor of the $\PSL(2,\C)$ A-polynomial. 
\end{theorem}

In fact, we obtain the same factor as Champanerkar. The precise version of this theorem is contained in \refthm{Ptolemy_Apoly} below.

\begin{remark}
Observe that 
the Ptolemy equations above are always quadratic in the variables $\gamma_j$. Moreover, their form indicates intriguing algebraic structure that is not readily apparent from the gluing equations.
\end{remark}

We find the simplicity and the algebraic structure of the equations of \refthm{MainPtolemy} to be a major feature of this paper. The defining equations of the A-polynomial are quite simple! 
We note that using these equations requires finding the vector $B$ of \refthm{MainPtolemy}. This is a problem in linear Diophantine equations. Because $B$ is guaranteed to exist, it can be found by computing the Smith normal form of the matrix $\NZ$ (see, for example, Chapter~II.21(c) of \cite{Newman:IntegralMatrices}). In practice, we were able to find $B$ for examples with significantly less work.

\begin{remark}\label{Rmk:Projective}
The $\gamma$ variables in \refthm{MainPtolemy} are precisely Dimofte's $\gamma$ variables of \cite{DimofteQRCS}, and these Ptolemy equations are essentially equivalent to those of that paper.

The word ``equivalent" here conceals a projective subtlety. The gluing and cusp equations are a set of $n+2$ equations in $n$ tetrahedron parameters and $\ell,m$, but only $n+1$ of them are independent. The Ptolemy equations are however a set of $n$ independent equations in $n$ edge variables and $\ell,m$. Nonetheless, they are homogeneous, and so $\gamma_1, \ldots, \gamma_n$ can be regarded as varying on $\CP^{n-1}$; alternatively, one can divide through by an appropriate power of one $\gamma_i$ to obtain equations in the $n-1$ variables $\frac{\gamma_1}{\gamma_i}, \ldots, \frac{\gamma_{i-1}}{\gamma_i}, \frac{\gamma_{i+1}}{\gamma_i}, \ldots, \frac{\gamma_n}{\gamma_i}$, which can be eliminated. Effectively, one can simply set one of the variables $\gamma_i$ to $1$.
\end{remark}

A further subtlety arises because our Ptolemy equations are \emph{not} polynomials in $m$ and $\ell$; they are rather polynomials in $m^{1/2}$ and $\ell^{1/2}$. If we set $M = m^{1/2}$ and $L = \ell^{1/2}$ then we obtain \emph{polynomial} Ptolemy equations. Moreover, the variables $L$ and $M$ so defined are essentially those appearing in the $\SL(2,\C)$ A-polynomial: a matrix in $\SL(2,\C)$ with eigenvalues $L,L^{-1}$ yields an element of $\PSL(2,\C)$ corresponding to a hyperbolic isometry with holonomy $L^2 = \ell$. Indeed, the Ptolemy varieties of \cite{Zickert:PtolemyDehnA-poly} are calculated from $\SL(2,\C)$ representations, rather than $\PSL(2,\C)$. We obtain the following.

\begin{corollary}
After setting $M = {\pm}m^{1/2}$ and $L = {\pm}\ell^{1/2}$, eliminating the $\gamma$ variables from the polynomial Ptolemy equations of a one-cusped hyperbolic triangulation yields a polynomial in $M$ and $L$ which contains, as a factor, the factor of the $\SL(2,\C)$ A-polynomial describing hyperbolic structures.
\end{corollary}
The precise version of this corollary is \refcor{Ptolemy_SL2C_Apoly}.

\subsection{Ptolemy equations in Dehn filling}

Our main application of \refthm{MainPtolemy} is to consider the defining equations of A-polynomials under Dehn filling. 

Consider a two-component link in $S^3$ with component knots $K_0, K_1$. Consider Dehn filling $K_0$ along some slope $p/q$; $K_1$ then becomes a knot in a 3-manifold. 
A Dehn filling can be triangulated using \emph{layered solid tori}, originally defined by Jaco and Rubinstein~\cite{JacoRubinstein:LST}; see also Gu\'{e}ritaud--Schleimer \cite{GueritaudSchleimer}.
Building a layered solid torus yields a sequence of triangulations of a once-punctured torus. The combinatorics of the 3-dimensional layered solid torus corresponds closely to the combinatorics of 2-dimensional triangulations of punctured tori.

Triangulations of punctured tori can be endowed with $\lambda$-lengths via work of Penner \cite{Penner87}. 
When one flips a diagonal in a triangulation, the $\lambda$-lengths are related by a Ptolemy equation. This gives the algebra formed by $\lambda$-lengths the structure of a \emph{cluster algebra} \cite{FockGoncharov06, FominShapiroThurston08, GSV05}. Cluster algebras arise in diverse contexts across mathematics (see e.g. \cite{FWZ:ClusterIntro, Williams:ClusterIntroduction}).

We obtain two sets of Ptolemy equations: one for the cluster algebra of the punctured torus coming from $\lambda$-lengths, and one for the tetrahedra in the layered solid torus coming from \refthm{MainPtolemy}. These are identical except for signs. Thus we can regard the algebra generated by our Ptolemy equations as a ``twisted" cluster algebra, where the word ``twisted" indicates some changes of sign.

\begin{thm}
\label{Thm:Ptolemy_eqns_Dehn_filling}
Suppose $M$ has two cusps $\mfc_0, \mfc_1$, and is triangulated such that only two tetrahedra meet $\mfc_1$, and generating curves $\mfm_0, \mfl_0$ on the cusp triangulation of $\mfc_0$ avoid these tetrahedra. Then for any Dehn filling on the cusp $\mfc_1$ obtained by attaching a layered solid torus, the Ptolemy equations satisfy:
\begin{enumerate}
\item There are a finite number of fixed Ptolemy equations, independent of the Dehn filling, coming from tetrahedra outside the Dehn filling. These are obtained as in \refthm{MainPtolemy} using the reduced Neumann--Zagier matrix and $B$ vector for the unfilled manifold. 
\item The Ptolemy equations for the tetrahedra in the solid torus take the form
\[
\pm \gamma_x \gamma_y \pm \gamma_a^2 - \gamma_b^2 = 0,
\]
where $a,b,x,y$ are slopes on the torus boundary and $x,y$ are crossing diagonals. In addition, the variable $\gamma_y$ will appear for the first time in this equation, with $\gamma_x$, $\gamma_a$, and $\gamma_b$ appearing in earlier equations.
\end{enumerate}
\end{thm}
A precise version of this theorem is \refthm{DehnFillPtolemy}.

Theorem~\ref{Thm:Ptolemy_eqns_Dehn_filling} in particular implies that each of the Ptolemy equations for the solid torus can be viewed as giving a recursive definition of the new variable $\gamma_y$. These equations are explicit,  depending on the slope. Since the outside Ptolemy equations are fixed, in practice this gives a recursive definition of the A-polynomial in terms of the slope of the Dehn filling. If we take a sequence of Dehn filling slopes $\{p_i/q_i\}$, then the A-polynomials of the knots $K_i = K_{p_i / q_i}$, are closely related. The Ptolemy equations defining $A_{K_{i+1}}$ are, roughly speaking, obtained from those for $A_{K_i}$ by adding a single extra Ptolemy relation.

We illustrate this theorem by example for twist knots, which are Dehn fillings of the Whitehead link. While A-polynomials of twist knots are known~\cite{HosteShanahan04,Mathews:A-poly_twist_knots, Mathews:A-poly_twist_knots_err}, we still believe this example is useful in showing the simplicity of the Ptolemy equations. In a follow up paper, we apply these tools to a new family of knots whose A-polynomials were previously unknown, namely twisted torus knots obtained by Dehn filling the Whitehead sister~\cite{HMPT:WhiteheadSister}.

\subsection{Structure of this paper}

In \refsec{symplectic}, we recall work of Thurston~\cite{thurston} and Neumann and Zagier~\cite{NeumannZagier}, including gluing and cusp equations, the Neumann--Zagier matrix, and its symplectic properties. We introduce a symplectic change of basis, and show this leads to Ptolemy equations that give the A-polynomial, proving \refthm{MainPtolemy}.

In \refsec{Dehn}, we connect to Dehn fillings. We review the construction of layered solid tori, and triangulations of Dehn filled manifolds, and show how the triangulation adjusts the Neumann--Zagier matrix. Using this, we find Ptolemy equations for any layered solid torus, completing the proof of Theorem~\ref{Thm:Ptolemy_eqns_Dehn_filling}. 

Section~\ref{Sec:Examples} works through the example of knots obtained by Dehn filling the Whitehead link.

\subsection{Acknowledgments}

This work was supported by the Australian Research Council, grants DP160103085 and DP210103136.

\section{From gluing equations to Ptolemy equations via symplectic reduction}
\label{Sec:symplectic}

In this section we discuss Dimofte's symplectic reduction method and refine it to show how gluing and cusp equations are equivalent to Ptolemy equations, proving \refthm{MainPtolemy}.

\subsection{Triangulations, gluing and cusp equations}
\label{Sec:Triangulations_gluing_cusp}

Let $M$ be a 3-manifold that is the interior of a compact manifold $\overline{M}$ with all boundary components tori.
Let the number of boundary tori be $n_\mfc$, so $M$ has $n_\mfc$ cusps. For example, $M$ may be a link complement $S^3-L$, where $L$ is a link of $n_\mfc$ components, and $\overline{M}$ a link exterior $S^3-N(L)$. 

Suppose $M$ has an ideal triangulation. Throughout this paper, unless stated otherwise, \emph{triangulation} means ideal triangulation, and \emph{tetrahedron} means ideal tetrahedron. Throughout, $n$ denotes the number of tetrahedra in a triangulation. 

\begin{figure}
  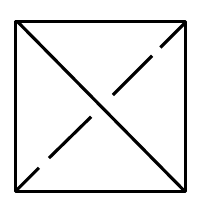
  \caption{A tetrahedron with vertices labeled $0$, $1$, $2$, $3$ and opposite edges labeled $a$, $b$, $c$.}
  \label{Fig:TetrLabels}
\end{figure}

\begin{defn}
\label{Def:OrLabelling}
An \emph{oriented labelling} of a tetrahedron is a labelling of its four ideal vertices with the numbers $0,1,2,3$ as in \reffig{TetrLabels}, up to oriented homeomorphism preserving edges.

In an ideal tetrahedron with an oriented labelling,  we call the opposite pairs of edges $(01,23), (02,13), (03,12)$ respectively the \emph{$a$-edges}, \emph{$b$-edges} and \emph{$c$-edges}.
\end{defn}
In an oriented labelling, around each vertex (as viewed from outside the tetrahedron), the three incident edges are an $a$-, $b$-, and $c$-edge in anticlockwise order.

The number of edges in the triangulation is equal to the number $n$ of tetrahedra, as follows: letting the numbers of edges and faces in the triangulation temporarily be $E$ and $F$, $\partial \overline{M}$ is triangulated with $2E$ vertices, $3F$ edges and $4n$ triangles. As $\partial \overline{M}$ consists of tori, its Euler characteristic $2E-3F+4n$ is zero. Since $2F=4n$, we have $E=n$.

\begin{defn}
\label{Def:Labelled_tri}
A \emph{labelled triangulation} of $M$ is an oriented ideal triangulation of $M$, where
\begin{enumerate}
\item the tetrahedra are labelled $\Delta_1, \ldots, \Delta_n$ in some order, 
\item the edges are labelled $E_1, \ldots, E_n$ in some order, and
\item each tetrahedron is given an oriented labelling.
\end{enumerate}
\end{defn}

As in the introduction, we will need to refer to the edge $E_k$ to which an edge of tetrahedron $\Delta_j$ is glued.
\begin{defn}
\label{Def:jmunu}
For $j \in \{1, \ldots, n\}$ and distinct $\mu, \nu \in \{0,1,2,3\}$, the index of the edge to which the edge $(\mu \nu)$ of $\Delta_j$ is glued is denoted $j(\mu \nu)$.
In other words, the edge $(\mu \nu)$ of $\Delta_j$ is identified to $E_{j(\mu \nu)}$.
\end{defn}

Suppose now that we have a labelled triangulation of $M$.
To each tetrahedron $\Delta_j$ we associate three variables $z_j, z'_j, z''_j$. These variables are associated with the $a$-, $b$- and $c$-edges of $\Delta_j$ and satisfy the equations
\begin{equation}\label{Eqn:Z''Eqn}
  z_j z_j' z_j'' = -1 \quad\quad \mbox{ and}
\end{equation}
\begin{equation}\label{Eqn:NotProd}
  z_j + (z_j')^{-1} -1 = 0.
\end{equation}
If $\Delta_j$ has a hyperbolic structure then these parameters are standard tetrahedron parameters; see \cite{Thurston:3DGT}. Each of $z_j, z'_j, z''_j$ gives the cross ratio of the four ideal points, in some order. The arguments of $z_j, z'_j, z''_j$ respectively give the dihedral angles of $\Delta_j$ at the $a$-, $b$- and $c$-edges. Note that equations~\eqref{Eqn:Z''Eqn} and~\eqref{Eqn:NotProd} imply that none of $z_j, z'_j, z''_j$ can be equal to $0$ or $1$ (i.e.\ tetrahedra are nondegenerate).

\begin{defn}
\label{Def:abc_edges}
In a labelled triangulation of $M$, we denote by
$a_{k,j}, b_{k,j}, c_{k,j}$ respectively the number of $a$-, $b$-, $c$-edges of $\Delta_j$ identified to $E_k$.
\end{defn}

\begin{lemma}\label{Lem:SumToTwo}
  For each fixed $j$, 
\begin{equation}
\label{Eqn:sum_of_abcs}
\sum_{k=1}^n a_{k,j} = 2, \quad
\sum_{k=1}^n b_{k,j} = 2 \quad \text{and} \quad
\sum_{k=1}^n c_{k,j} = 2.
\end{equation}
\end{lemma}

\begin{proof}
Each tetrahedron $\Delta_j$ has two $a$-edges, two $b$-edges and two $c$-edges, so for fixed $j$ the total sum over all $k$ must be $2$.
\end{proof}
The nonzero terms in the first sum are $a_{j(01),j}$ and $a_{j(23),j}$.
Note that $j(01)$ could equal $j(23)$; this occurs when the two $a$-edges of $\Delta_j$ are glued to the same edge.
In that case, $a_{j(01),j}$ and $a_{j(23),j}$ are the same term, equal to $2$. 
If the two $a$-edges are not glued to the same edge, then $E_{j(01)}$ and $E_{j(23)}$ are distinct, each with one $a$-edge of $\Delta_j$ identified to it, and $a_{j(01),j} = a_{j(23),j} = 1$. Similarly, the nonzero terms in the second sum are $b_{j(02),j}, b_{j(13),j}$ and in the third sum $c_{j(03),j}, c_{j(12),j}$.

The numbers $a_{k,j}, b_{k,j}, c_{k,j}$ can be arranged into a matrix.
\begin{defn}
\label{Def:In}
The \emph{incidence matrix} $\In$ of a labelled triangulation $\TT$ is the $n \times 3n$ matrix whose $k$th row is $(a_{k,1}, b_{k,1}, c_{k,1}, \ldots, a_{k,n}, b_{k,n}, c_{k,n})$.
\end{defn}
Thus $\In$ has rows corresponding to the edges $E_1, \ldots, E_n$, and the columns come in triples with the $j$th triple corresponding to the tetrahedron $\Delta_j$.

The \emph{gluing equation} for edge $E_k$ is then
\begin{equation}\label{Eqn:Gluing}
\prod_{j=1}^n z_j^{a_{k,j}}(z_j')^{b_{k,j}}(z_j'')^{c_{k,j}} =1.
\end{equation}
When the ideal triangulation $\TT$ is hyperbolic, the gluing equations express the fact that tetrahedra fit geometrically together around each edge.

Denote the $n_\mfc$ boundary tori of $\overline{M}$ by $\T_1, \ldots, \T_{n_\mfc}$. A triangulation of $M$ by tetrahedra induces a triangulation of each $\T_k$ by triangles. On each $\T_k$ we choose a pair of oriented curves $\mfm_k, \mfl_k$ forming a basis for $H_1 (\T_k)$. By an isotopy if necessary, we may assume each curve is in general position with respect to the triangulation of $\T_k$, without backtracking. Then each curve splits into segments, where each segment lies in a single triangle and runs from one edge to a distinct edge. Each segment of $\mfm_k$ or $\mfl_k$ can thus be regarded as running clockwise or anticlockwise around a unique corner of a triangle; these directions are as viewed from outside the manifold. 
We count anticlockwise motion around a vertex as positive, and clockwise motion as negative.
Each vertex (resp.\ face) of the triangulation of $\T_k$ corresponds to some edge (resp.\ tetrahedron) of the triangulation $\TT$ of $M$; thus each corner of a triangle corresponds to a specific edge of a specific tetrahedron. 
\begin{defn}
\label{Def:abc-incidence}
The \emph{$a$-incidence number} (resp.\ $b$-, $c$-incidence number) of $\mfm_k$ (resp.\ $\mfl_k$) with the tetrahedron $\Delta_j$ is the number of segments of $\mfm_k$ (resp.\ $\mfl_k$) running anticlockwise (i.e.\ positively) through a corner of a triangle corresponding to an $a$-edge (resp.\ $b$-, $c$-edge) of $\Delta_j$, minus the number of segments of $\mfm_k$ (resp.\ $\mfl_k$) running clockwise (i.e.\ negatively) through a corner of a triangle corresponding to an $a$-edge (resp.\ $b$-edge, $c$-edge) of $\Delta_j$.
\begin{enumerate}
\item Denote by $a^\mfm_{k,j}, b^\mfm_{k,j}, c^\mfm_{k,j}$ the $a$-, $b$-, $c$-incidence numbers of $\mfm_k$ with $\Delta_j$.
\item Denote by $a^\mfl_{k,j}, b^\mfl_{k,j}, c^\mfl_{k,j}$ the $a$-, $b$-, $c$-incidence numbers of $\mfl_k$ with $\Delta_j$.
\end{enumerate}
\end{defn}

To each cusp torus $\T_k$ we associate variables $m_k, \ell_k$. The \emph{cusp equations} at $\T_k$ are 
\begin{equation}\label{Eqn:CuspEqn}
m_k = \prod_{j=1}^n z_j^{a^\mfm_{k,j}} \left( z_j' \right)^{b^\mfm_{k,j}} \left( z_j'' \right)^{c^\mfm_{k,j}}, \quad\quad
\ell_k = \prod_{j=1}^n z_j^{a^\mfl_{k,j}} \left( z_j' \right)^{b^\mfl_{k,j}} \left(z_j'' \right)^{c^\mfl_{k,j}}
\end{equation}
When $\TT$ is a \emph{hyperbolic triangulation}, meaning the ideal tetrahedra are all positively oriented and glue to give a smooth, complete hyperbolic structure on the underlying manifold, the cusp equations give $m_k$ and $\ell_k$, the holonomies of the cusp curves $\mfm_k$ and $\mfl_k$, in terms of tetrahedron parameters.

Any hyperbolic triangulation $\TT$ gives tetrahedron parameters $z_j, z'_j, z''_j$ and cusp holonomies $m_k, \ell_k$ satisfying the relationships \eqref{Eqn:Z''Eqn}--\eqref{Eqn:NotProd} between the $z$-variables, the gluing equations \eqref{Eqn:Gluing} and cusp equations \eqref{Eqn:CuspEqn}; moreover, the tetrahedron parameters all have positive imaginary part.
However, in general there may be solutions of these equations which do not correspond to a hyperbolic triangulation, for instance those with $z_j$ with negative imaginary part (which may still give $M$ a hyperbolic structure), or with branching around an edge (which will not). Additionally, not every hyperbolic structure on $M$ may give a solution to the gluing and cusp equations, since the triangulation $\TT$ may not be geometrically realisable.

\subsection{The A-polynomial from gluing and cusp equations}

Suppose now that $n_\mfc=1$, i.e.\ $M$ has one cusp, and moreover, that $M$ is the complement of a knot $K$ in a homology 3-sphere. 

In this case, there is no need for the $k=1$ subscript in notation for the lone cusp, and we may simply write
\begin{gather*}
\mfm = \mfm_1, \quad \mfl = \mfl_1, \quad
m = m_1, \quad \ell = \ell_1, \\
a^\mfm_j = a^\mfm_{1,j}, \quad b^\mfm_j = b^\mfm_{1,j}, \quad c^\mfm_j = c^\mfm_{1,j}, \quad
a^\mfl_j = a^\mfl_{1,j}, \quad b^\mfl_j = b^\mfl_{1,j}, \quad c^\mfl_j = c^\mfl_{1,j}, \quad
\end{gather*}

In this case we can take the boundary curves $(\mfm, \mfl)$ to be a topological longitude and meridian respectively. That is, we may take $\mfl$ to be primitive and nullhomologous in $M$, and $\mfm$ to bound a disc in a neighbourhood of $K$. 

We orient $\mfm$ and $\mfl$ so that the tangent vectors $v_\mfm$ and $v_\mfl$ to $\mfm$ and $\mfl$, respectively, at the point where $\mfm$ intersects $\mfl$ are oriented according to the right hand rule: $v_\mfm\times v_\mfl$ points in the direction of the outward normal.

The equations \eqref{Eqn:Z''Eqn}--\eqref{Eqn:NotProd} relating $z,z',z''$ variables, the gluing equations \eqref{Eqn:Gluing}, and the cusp equations \eqref{Eqn:CuspEqn} are equations in the variables $z_j, z'_j, z''_j$ and $\ell,m$. Solve these equations for $\ell,m$, eliminating the variables $z_j,z'_j,z''_j$ to obtain a relation between $\ell$ and $m$.

Champanerkar~\cite{Champanerkar:Thesis} showed that the above equations can be solved in this sense to give divisors of the $\PSL(2,\C)$ A-polynomial of $M$. Segerman showed that, if one takes a certain extended version of this variety, there exists a triangulation such that all factors of the $\PSL(2,\C)$ A-polynomial are obtained \cite{Segerman:Deformation}. See also \cite{GoernerZickert} for an effective algorithm. 

\begin{thm}[Champanerkar]
\label{Thm:Champ}
When we solve the system of equations \eqref{Eqn:Z''Eqn}--\eqref{Eqn:NotProd}, \eqref{Eqn:Gluing} and \eqref{Eqn:CuspEqn} in terms of $m$ and $\ell$, we obtain a factor of the $\PSL(2,\CC)$ A-polynomial.
\end{thm}

\subsection{Logarithmic equations and Neumann-Zagier matrix}

We now return to the general case where the number $n_\mfc$ of cusps of $M$ is arbitrary. 

Note that equation \eqref{Eqn:Z''Eqn} relating $z_j,z'_j,z''_j$, the gluing equations \eqref{Eqn:Gluing}, and the cusp equations \eqref{Eqn:CuspEqn} are multiplicative. By taking logarithms now we make them additive.

Equation \eqref{Eqn:Z''Eqn} implies that each $z_j, z'_j$ and $z''_j$ is nonzero. Taking (an appropriate branch of)
a logarithm we obtain
\[
\log z_j + \log z'_j + \log z''_j = i\pi
\]
Define $Z_j = \log z_j$ and  $Z_j' = \log z_j'$, using the branch of the logarithm with argument in $(-\pi, \pi]$, and then define $Z''_j$ as
\begin{equation}
\label{Eqn:LogZ''Eqn}
Z_j'' = i\pi - Z_j - Z_j',
\end{equation}
so that indeed $Z''_j$ is a logarithm of $z''_j$.

In a hyperbolic triangulation, each tetrahedron parameter has positive imaginary part. The arguments of $z_j, z'_j, z''_j$ (i.e.\ the imaginary parts of $Z_j, Z'_j, Z''_j$) are the dihedral angles at the $a$-, $b$- and $c$-edges of $\Delta_j$ respectively. They are the angles of a Euclidean triangle, hence they all lie in $(0,\pi)$ and they sum to $\pi$.

The gluing equation \eqref{Eqn:Gluing} expresses the fact that tetrahedra fit together around an edge. Taking a logarithm, we may make the somewhat finer statement that dihedral angles around the edge sum to $2\pi$. Thus we take the logarithmic form of the gluing equations as
\begin{equation}
\label{Eqn:LogGluing}
\sum_{j=1}^n a_{k,j} Z_j + b_{k,j} Z_j' + c_{k,j} Z_j'' = 2\pi i.
\end{equation}
We similarly obtain logarithmic forms of the cusp equations \eqref{Eqn:CuspEqn} as
\begin{equation}
\label{Eqn:LogCuspEqn}
\log m_k = \sum_{j=1}^n a^\mfm_{k,j} Z_j + b^\mfm_{k,j} Z_j' + c^\mfm_{k,j} Z_j'', \quad
\log \ell_k = \sum_{j=1}^n a^\mfl_{k,j} Z_j+ b^\mfl_{k,j} Z_j' + c^\mfl_{k,j} Z_j''.
\end{equation}
We can then observe that any solution of \eqref{Eqn:LogZ''Eqn} and the logarithmic gluing and cusp equations \eqref{Eqn:LogGluing}--\eqref{Eqn:LogCuspEqn} yields, after exponentiation, a solution of \eqref{Eqn:Z''Eqn} and the original gluing \eqref{Eqn:Gluing} and cusp equations \eqref{Eqn:CuspEqn}. Moreover, any solution of \eqref{Eqn:Z''Eqn}, \eqref{Eqn:Gluing} and \eqref{Eqn:CuspEqn} has a logarithm which is a solution of \eqref{Eqn:LogZ''Eqn} and \eqref{Eqn:LogGluing}--\eqref{Eqn:LogCuspEqn}.

Using equation \eqref{Eqn:LogZ''Eqn} we eliminate the variables $Z''_j$ (just as using equation \eqref{Eqn:Z''Eqn} we can eliminate the variables $z''_j$). In doing so, coefficients are combined in a way that persists throughout this paper, and so we define these combinations as follows.
\begin{defn}
\label{Def:NZConsts}
For a given labelled triangulation of $M$, we define
\begin{gather*}
d_{k,j} = a_{k,j} - c_{k,j}, \quad
d'_{k,j} = b_{k,j} - c_{k,j}, \quad
c_k = \sum_{j=1}^n c_{k,j} \quad \text{for $k=1, 2, \ldots, n$,} \\
\mu_{k,j} = a^\mfm_{k,j} - c^\mfm_{k,j}, \quad
\mu'_{k,j} = b^\mfm_{k,j} - c^\mfm_{k,j}, \quad
c^\mfm_k = \sum_{j=1}^n c^\mfm_{k,j} \quad \text{for $k = 1, 2, \ldots, n_\mfc$,} \\
\lambda_{k,j} = a^\mfl_{k,j} - c^\mfl_{k,j}, \quad
\lambda'_{k,j} = b^\mfl_{k,j} - c^\mfl_{k,j}, \quad
c^\mfl_k = \sum_{j=1}^n c^\mfl_{k,j} \quad \text{for $k = 1, 2, \ldots, n_\mfc$.}
\end{gather*}
Note that the index $k$ in the first line steps through the $n$ edges, while the index $k$ in the next two lines steps through the $n_\mfc$ cusps.
\end{defn}

When $n_\mfc = 1$ we can drop the $k$ subscript on cusp terms, so we have
\[
\mu_j = a^\mfm_j - c^\mfm_j, \quad
\mu'_j = b^\mfm_j - c^\mfm_j, \quad
c^\mfm = \sum_{j=1}^n c^\mfm_j, \quad
\lambda_j = a^\mfl_j - c^\mfl_j, \quad
\lambda'_j = b^\mfl_j - c^\mfl_j, \quad
c^\mfl = \sum_{j=1}^n c^\mfl_j.
\]

We thus rewrite the the logarithmic gluing and cusp equations \eqref{Eqn:LogGluing}--\eqref{Eqn:LogCuspEqn} in terms of the variables $Z_j, Z'_j$ and $\ell_k, m_k$ only, as
\begin{align}
  \sum_{j=1}^n d_{k,j} Z_j + d'_{k,j} Z_j' &= i \pi \left( 2 - c_k \right) \label{Eqn:NZGluing} \\
\sum_{j=1}^n \mu_{k,j} Z_j + \mu'_{k,j} Z_j'  &= \log m_k - i\pi c^\mfm_{k} \label{Eqn:NZMEqn} \\
\sum_{j=1}^n \lambda_{k,j} Z_j + \lambda'_{k,j} Z_j' &= \log\ell_k - i\pi c^\mfl_{k}. \label{Eqn:NZEllEqn}
\end{align}

Define the row vectors of coefficients in equations \eqref{Eqn:NZGluing}--\eqref{Eqn:NZEllEqn} as follows:
\[
\begin{array}{rcccccccc}
R^G_k &:=& ( & d_{k,1} & d_{k,1}' & \ldots & d_{k,n} & d_{k,n}' & ) \\
R^\mfm_k &:=& ( & \mu_{k,1} & \mu'_{k,1} & \cdots & \mu_{k,n} & \mu'_{k,n} & ) \\
R^\mfl_k &:= &( & \lambda_{k,1} & \lambda'_{k,1} & \cdots & \lambda_{k,n} & \lambda'_{k,n} & ).
\end{array}
\]
So $R^G_k$ gives the coefficients in the logarithmic gluing equation for the $k$th edge $E_k$, and $R^\mfm_k, R^\mfl_k$ give respectively coefficients in the logarithmic cusp equations for $\mfm_k$ and $\mfl_k$ on the $k$th cusp. 

When $n_\mfc = 1$ we again drop the $k$ subscript on cusp terms and simply write $R^\mfm = R^\mfm_k$ and $R^\mfl = R^\mfl_k$, so that $R^\mfm = (\mu_1, \mu'_1, \ldots, \mu_n, \mu'_n)$ and $R^\mfl = (\lambda_1, \lambda'_1, \ldots, \lambda_n, \lambda'_n)$.

By re-exponentiating we observe natural meanings for the new $d,d',\mu,\mu',\lambda,\lambda',c$ coefficients of \refdef{NZConsts}. The tetrahedron parameters and the holonomies $m_k, \ell_k$ satisfy versions of the gluing and cusp equations without any $z''_j$ appearing, where the $d,d'$ variables appear as exponents in gluing equations, $\mu,\mu',\lambda,\lambda'$ variables appear as exponents in cusp equations, and the $c$ variables determine signs:
\begin{gather*}
\prod_{j=1}^n z_j^{d_{k,j}} \left( z'_j \right)^{d'_{k,j}} = 
\left( -1 \right)^{c_k} \quad \text{for $k=1, \ldots, n$ (indexing edges)} \\
m_k = \left( -1 \right)^{c^\mfm_{k}} \prod_{j=1}^n z_j^{\mu_{k,j}} \left( z'_j \right)^{\mu'_{k,j}}, \quad
\ell_k = \left( -1 \right)^{c^\mfl_{k}} \prod_{j=1}^n z_j^{\lambda_{k,j}} \left( z'_j \right)^{\lambda'_{k,j}} \quad \text{for $k=1, \ldots, n_\mfc$ (cusps)}.
\end{gather*}
When $n_\mfc = 1$, the notation for cusp equations again simplifies so we have \[
m = \left( -1 \right)^{c^\mfm} \prod_{j=1}^n z_j^{\mu_j} \left( z'_j \right)^{\mu'_j}, \quad \text{ and }\quad \ell = \left( -1 \right)^{c^\mfl} \prod_{j=1}^n z_j^{\lambda_j} \left( z'_j \right)^{\lambda'_j}.
\]

The matrix with rows $R_1^G, \dots, R_n^G, R^\mfm_1, R^\mfl_1, \ldots, R^\mfm_{n_\mfc}, R^\mfl_{n_\mfc}$ is called the \emph{Neumann--Zagier matrix}, and we denote it by $\NZ$. The first $n$ rows correspond to the edges $E_1, \ldots, E_n$, and the next rows come in pairs corresponding to the pairs $(\mfm_k, \mfl_k)$ of basis curves for the cusp tori $\T_1, \ldots, \T_{n_\mfc}$. The columns come in pairs corresponding to the tetrahedra $\Delta_1, \ldots, \Delta_n$. Note that the data of a labelled triangulation of  \refdef{Labelled_tri} give us the information to write down the matrix: the edge ordering $E_1, \ldots, E_n$ orders the rows; the tetrahedron ordering $\Delta_1, \ldots, \Delta_n$ orders pairs of columns; and the oriented labelling on each tetrahedron determines each pair of columns.
\begin{equation}\label{Eqn:NZMatrix}
\NZ = \begin{bmatrix} R_1^G \\ \vdots \\ R_n^G \\ R^\mfm_1 \\ R^\mfl_1 \\ \vdots \\ R^\mfm_{n_\mfc} \\ R^\mfl_{n_\mfc} \end{bmatrix}
= 
\kbordermatrix{
 & \Delta_1 & \cdots & \Delta_n \\
E_1 & d_{1,1} \quad d_{1,1}' & \cdots & d_{1,n} \quad d_{1,n}' \\
\vdots & \vdots & \ddots & \vdots \\
E_n & d_{n,1} \quad d_{n,1}' & \cdots & d_{n,n} \quad d_{n,n}' \\
\mfm_1 & \mu_{1,1} \quad \mu_{1,1}' & \cdots & \mu_{1,n} \quad \mu'_{1,n} \\
\mfl_1 & \lambda_{1,1} \quad \lambda'_{1,1} & \cdots & \lambda_{1,n} \quad \lambda'_{1,n} \\
\vdots & \vdots & \ddots & \vdots \\
\mfm_{n_\mfc} & \mu_{n_\mfc,1} \quad \mu'_{n_\mfc,1} & \cdots & \mu_{n_\mfc,n} \quad \mu_{n_\mfc,n} \\
\mfl_{n_\mfc} & \lambda_{n_\mfc,1} \quad \lambda'_{n_\mfc,1} & \cdots & \lambda_{n_\mfc,n} \quad \lambda'_{n_\mfc,n}
}
\end{equation}

The gluing and cusp equations can then be written as a single matrix equation, if we make the following definitions.
\begin{defn} \
\label{Def:ZzHC}
The \emph{$Z$-vector}, \emph{$z$-vector}, \emph{$H$-vector} and \emph{$C$-vector} are defined as
\begin{align*}
Z &:= \left( Z_1, Z'_1, \ldots, Z_n, Z'_n \right)^T, \\
z &:= \left( z_1, z'_1, \ldots, z_n, z'_n \right)^T, \\
H &:= \left( 0, \ldots, 0, \log m_1, \log \ell_1, \ldots, \log m_{n_\mfc}, \log \ell_{n_\mfc} \right)^T, \\
C &:= \left( 2-c_1, \ldots, 2-c_n, -c^\mfm_{1}, -c^\mfl_1, \ldots, -c^\mfm_{n_\mfc}, -c^\mfl_{n_\mfc} \right)^T.
\end{align*}
\end{defn}

The vector $Z$ contains the logarithmic tetrahedral parameters; the vector $H$ contains the cusp holonomies, and the vector $C$ is a vector of constants derived from the gluing data, giving sign terms in exponentiated equations.

We summarise our manipulations of the various equations in the following statement.
\begin{lem}
\label{Lem:LogEqns_NZ}
Let $\TT$ be a labelled triangulation of $M$.
\begin{enumerate}
\item
The logarithmic gluing and cusp equations can be written compactly as
\begin{equation}
\label{Eqn:NZGluingCuspEqn}
\NZ  \cdot Z = H + i \pi C.
\end{equation}
That is, logarithmic gluing and cusp equations \eqref{Eqn:NZGluing}--\eqref{Eqn:NZEllEqn}
are equivalent to \eqref{Eqn:NZGluingCuspEqn}.
\item
After exponentiation, a solution $Z$ of \eqref{Eqn:NZGluingCuspEqn} gives $z$ which, together with $z''_j$ defined by \eqref{Eqn:Z''Eqn}, yields a solution of the gluing equations \eqref{Eqn:Gluing} and cusp equations \eqref{Eqn:CuspEqn}. 
\item
Conversely, any solution $(z_j, z'_j, z''_j)$ of \eqref{Eqn:Z''Eqn}, gluing equations \eqref{Eqn:Gluing} and cusp equations \eqref{Eqn:CuspEqn} yields $z$ with logarithm $Z$ satisfying \eqref{Eqn:NZGluingCuspEqn}.
\item
Any hyperbolic triangulation yields $Z$ and $H$ which satisfy \eqref{Eqn:NZGluingCuspEqn}. \qed
\end{enumerate}
\end{lem}

\subsection{Symplectic and topological properties of the Neumann-Zagier matrix}
\label{Sec:symplectic_top_properties_NZ}

The matrix $\NZ$ has nice symplectic properties, due to 
Neumann--Zagier \cite{NeumannZagier}, which we now recall.

First, we introduce notation for the standard symplectic structure on $\R^{2N}$, for any positive integer $N$.
Denote by $\e_i$ (resp.\ $\f_i$) the vector whose only nonzero entry is a $1$ in the $(2i-1)$th coordinate (resp.\ $2i$th coordinate). Dually, let $x_i$ (resp.\ $y_i$) denote the coordinate function which returns the $(2i-1)$th coordinate (resp.\ $2i$th coordinate). We define the standard symplectic form $\omega$ as 
\begin{equation}\label{Eqn:SymplecticForm}
\omega = dx_1\wedge dy_1 + \cdots+ dx_N\wedge dy_N = \sum_{j=1}^N dx_j \wedge dy_j.
\end{equation}
Thus, given two vectors $V=(V_1, V_1', \dots, V_N, V_N')$ and $W=(W_1, W_1', \dots, W_N, W_N')$ in $\RR^{2N}$,
\[
\omega(V,W) = \sum_{j=1}^N V_j W_j'-V_j' W_j.
\]
Alternatively, $\omega(V,W) = V^T J W = (JV) \cdot W$, where $\cdot$ is the standard dot product, and $J$ is multiplication by $i$ on $\CC^N \cong \RR^{2N}$, i.e.\ $J(\e_i) = \f_i$ and $J(\f_i) = -\e_i$ (hence $J^2=-1$). As a matrix,
\[
J = \begin{bmatrix}
0 & -1 &    &    &				&		&   \\
1 & 0  &    &    & 				&		&\\
  &    &  0 & -1 &  			&		&\\
	&    &  1 & 0  & 				&		&\\
	&    &    &    & \ddots	&		& \\
	&    &    &    &        & 0 & -1 \\
	&    &    &    &        & 1 & 0
	\end{bmatrix}
\]

The ordered basis $( \e_1, \f_1, \ldots, \e_N, \f_N )$ forms a \emph{standard symplectic basis}, satisfying
\[
\omega(\e_i, \f_j) = \delta_{i,j}, \quad
\omega(\e_i, \e_j) = 0, \quad
\omega(\f_i, \f_j) = 0
\]
for all $i,j \in \{1, \ldots, N\}$. Any sequence of $2N$ vectors on which $\omega$ takes the same values on pairs is a \emph{symplectic basis}.

Maps which preserve a symplectic form are called \emph{symplectomorphisms}. We will need to use a few particular linear symplectomorphisms. The proof below is a routine verification.
\begin{lem}
\label{Lem:symplectomorphisms_fixing_fs}
In the standard symplectic vector space $(\R^{2N}, \omega)$ as above, the following linear transformations are symplectomorphisms:
\begin{enumerate}
\item
For $j,k \in \{1, \ldots, N\}$, $j\neq k$, and any $a \in \R$, map $\e_j \mapsto \e_j + a \f_k$, $\e_k \mapsto \e_k + a \f_j$, and leave all other standard basis vectors unchanged.
\item
For $j \in \{1, \ldots, N\}$ and any $a \in \RR$, map $\e_j \mapsto \e_j + a \f_j$, and leave all other standard basis vectors unchanged. \qed
\end{enumerate}
\end{lem}
In fact, it is not difficult to show that the linear symplectomorphisms above generate the group of linear symplectomorphisms which fix all $\f_j$. If we reorder the standard basis $(\e_1, \ldots, \e_n, \f_1, \ldots, \f_n)$, the symplectic matrices fixing the Lagrangian subspace spanned by the $\f_j$ have matrices of the form
\[
\begin{bmatrix} I & 0 \\ A & I \end{bmatrix}
\]
where $I$ is the $n \times n$ identity matrix and $A$ is an $n \times n$ symmetric matrix. These form a group isomorphic to the group of $n \times n$ real symmetric matrices under addition.

Returning to the Neumann-Zagier matrix $\NZ$, observe that its row vectors lie in $\R^{2n}$, where $n$ (as always) is the number of tetrahedra. These vectors behave nicely with respect to $\omega$.
\begin{thm}[Neumann--Zagier \cite{NeumannZagier}]\label{Thm:NeumannZagier}
With $R^G_k, R^\mfm_k, R^\mfl_k$ and $\omega$ as above:
\begin{enumerate}
\item For all $j,k \in \{1, \dots, n\}$, we have $\omega(R^G_j, R^G_k)=0$.
\item For all $j \in \{1, \ldots, n\}$ and $k \in \{1, \ldots, n_\mfc\}$, we have $\omega(R^G_j, R^\mfm_k)=\omega(R^G_j,R^\mfl_k)=0$.
\item For all $j,k \in \{1, \ldots, n_\mfc\}$, we have $\omega(R^\mfm_j,R^\mfl_k)= 2 \delta_{jk}$.
\item The row vectors $R^G_1, \dots, R^G_n$ span a subspace of dimension $n-n_\mfc$.
\item The rank of $\NZ$ is $n+n_\mfc$.
\end{enumerate}
\end{thm}

In light of theorem \ref{Thm:NeumannZagier}(iv), by relabelling edges if necessary, we can assume a labelled triangulation has the property
that the first $n-n_\mfc$ rows of its Neumann--Zagier matrix are linearly independent. We will make this assumption throughout.

According to theorem \ref{Thm:NeumannZagier}, the values of $\omega$ on pairs of vectors taken from the list of $n+n_\mfc$ vectors $\left( R^G_1, \ldots, R^G_{n-n_\mfc}, R^\mfm_1, \half R^\mfl_1, \ldots, R^\mfm_{n_\mfc}, \half R^\mfl_{n_\mfc} \right)$ agree with the value of $\omega$ on corresponding pairs in the list $(\f_1, \ldots, \f_{n-n_\mfc}, \e_{n-n_\mfc+1}, \f_{n-n_\mfc+1}, \ldots, \e_n, \f_n)$. For $R^G_1, \ldots, R^G_{n-n_\mfc}$ linearly independent, there is a linear symplectomorphism sending each vector in the first list to the corresponding vector in the second. 

Accordingly, as observed by Dimofte~\cite{DimofteQRCS} the list of $n+n_\mfc$ vectors
\[ \left( R^G_1, \ldots, R^G_{n-n_\mfc}, R^\mfm_1, \half R^\mfl_1, \ldots, R^\mfm_{n_\mfc}, \half R^\mfl_{n_\mfc} \right)
\]
extends to a symplectic basis for $\R^{2n}$,
\[
\left( R^\Gamma_1, R^G_1, \ldots, R^\Gamma_{n-n_\mfc}, R^G_{n-n_\mfc}, R^\mfm_1, \half R^\mfl_1, \ldots, R^\mfm_{n_\mfc}, \half R^\mfl_{n_\mfc} \right),
\]
with the addition of $n-n_\mfc$ vectors, denoted $R^\Gamma_1, \ldots, R^\Gamma_{n-n_\mfc}$.
Being a symplectic basis means that, in addition to the equations of \refthm{NeumannZagier}(i)--(iii), we also have
\begin{gather*}
\omega(R^\Gamma_j, R^\Gamma_k) = 0
\quad \text{and} \quad
\omega(R^\Gamma_j, R^G_k)  =\delta_{j,k}
\quad \text{for all $j,k\in \{1, \dots, n-n_\mfc\}$, and} \\
\omega(R^\Gamma_j,R^\mfm_k) = \omega(R^\Gamma_j,R^\mfl_k) = 0
\quad \text{for all $j \in \{ 1, \ldots, n-n_\mfc \}$ and $k \in \{1, \ldots, n_\mfc \}$.}
\end{gather*}

Indeed, the $R^\Gamma_j$ may be found by solving the equations above: given $R^G_k, R^\mfm_k, R^\mfl_k$, we may solve successively for $R^\Gamma_1, R^\Gamma_2, \ldots, R^\Gamma_{n-n_\mfc}$. Being solutions of linear equations with rational coefficients, we can find each $R^\Gamma_j \in \Q^{2n}$.

\begin{remark}\label{Rmk:NonUnique}
  Note that the $R^\Gamma_j$ are not unique: there are many solutions to the above equations. Distinct solutions are related precisely by the linear symplectomorphisms of $\R^{2n}$ fixing an $(n+n_\mfc)$-dimensional coisotropic subspace. Following the discussion after \reflem{symplectomorphisms_fixing_fs}, such symplectomorphisms are naturally bijective with $(n-n_\mfc) \times (n-n_\mfc)$ real symmetric matrices. Hence the space of possible $(R^\Gamma_1, \ldots R^\Gamma_{n-n_\mfc})$ has dimension $\half \left( n - n_\mfc \right) \left( n - n_\mfc + 1 \right)$.
\end{remark}

For $k \in \{1, \dots, n-n_\mfc\}$, write
\[ R^\Gamma_k = \left( f_{k,1} \quad f_{k,1}' \quad \dots \quad f_{k,n} \quad f_{k,n}' \right). \]

The symplectic basis $(R^G_1, R^\Gamma_1, \ldots, R^G_{n-n_\mfc}, R^\Gamma_{n-n_\mfc}, R^\mfm_1, \half R^\mfl_1, \ldots, R^\mfm_{n_\mfc}, \half R^\mfl_{n_\mfc})$ forms the  sequence of row vectors of a symplectic matrix, which we call $\SY \in \Sp(2n,\R)$. When $n_\mfc = 1$ we have
\begin{equation}
\label{Eqn:SYDefn}
\SY := 
\begin{bmatrix}
R^\Gamma_1 \\ R^G_1 \\ \vdots \\ R^\Gamma_{n-1} \\ R^G_{n-1} \\ R^\mfm \\ \half R^\mfl
\end{bmatrix}
=
\begin{bmatrix}
f_{1,1} & f_{1,1}' & f_{1,2} & f'_{1,2} & \cdots & f_{1,n} & f'_{1,n}  \\
d_{1,1}& d_{1,1}' & d_{1,2} & d'_{1,2} & \cdots & d_{1,n} & d'_{1,n}  \\
\vdots & \vdots & \vdots & \vdots & \ddots & \vdots & \vdots \\
f_{n-1,1} & f'_{n-1,1} & f_{n-1,2} & f'_{n-1,2} & \cdots & f_{n-1,n} & f'_{n-1,n} \\
d_{n-1,1} & d_{n-1,1}' & d_{n-1,2} & d'_{n-1,2} & \cdots & d_{n-1,n} & d'_{n-1,n} \\
\mu_1 & \mu_1' & \mu_2 & \mu_2' & \cdots & \mu_n & \mu_n' \\
\frac{1}{2} \lambda_1 & \frac{1}{2} \lambda_1' & \frac{1}{2} \lambda_2 & \frac{1}{2} \lambda_2' & \cdots & \frac{1}{2} \lambda_n & \frac{1}{2} \lambda_n'
\end{bmatrix}
\end{equation}

As a symplectic matrix, $\SY$ satisfies $(\SY)^T J (\SY)=J$, and for any vectors $V,W$, $\omega(V,W) = \omega(\SY\cdot V, \SY\cdot W)$.

\subsection{Linear and nonlinear equations and hyperbolic structures}

The symplectic matrix $\SY$ of \eqref{Eqn:SYDefn} shares several rows in common with $\NZ$. We will need to rearrange rows of various matrices, and so we make the following definition.
\begin{defn} 
\label{Def:FlatSharp}
Let $A$ be a matrix with $n+2n_\mfc$ rows, denoted $A_1, \ldots, A_{n+2n_\mfc}$.
\begin{enumerate}
\item
The submatrices $A^I, A^{II}, A^{III}$ consist of the first $n-n_\mfc$ rows, the next $n_\mfc$ rows, and the final $2n_\mfc$ rows. That is,
\[
A^I = \begin{bmatrix} A_1 \\ \vdots \\ A_{n-n_\mfc} \end{bmatrix}, \quad
A^{II} = \begin{bmatrix} A_{n-n_\mfc + 1} \\ \vdots \\ A_n \end{bmatrix}, \quad
A^{III} = \begin{bmatrix} A_{n+1} \\ \vdots \\ A_{n+2n_\mfc} \end{bmatrix}, \quad \mbox{ so }
A = \begin{bmatrix} A^I \\ A^{II} \\ A^{III}
\end{bmatrix}.
\]
\item
The matrix $A^\flat$ consists of the rows of $A^I$ followed by the rows of $A^{III}$. In other words, it is the matrix of $n+n_\mfc$ rows
\[
A^\flat = \begin{bmatrix} A^I \\ A^{III} \end{bmatrix}.
\]
\end{enumerate}
\end{defn}
This matrix $A$ of \refdef{FlatSharp} includes the case of a $(n+2n_\mfc) \times 1$ matrix, i.e.\ a $(n+2n_\mfc)$-dimensional vector.

Observe that \refdef{FlatSharp} applies to the Neumann-Zagier matrix $\NZ$. The matrix $\NZ^I$ has rows $R^G_1, \ldots, R^G_{n-n_\mfc}$, which we may assume are linearly independent. By \refthm{NeumannZagier}(i) and (iv), the rows of $\NZ^I$ form a basis of an isotropic subspace, and the rows of $\NZ^{II}$ also lie in this subspace.
 The matrix $\NZ^{III}$ has rows $R^\mfm_1, R^\mfl_1, \ldots, R^\mfm_{n_\mfc}, R^\mfl_{n_\mfc}$. Theorem~\ref{Thm:NeumannZagier}(iv) and (v) imply that the rows of $\NZ^\flat$ form a basis for the rowspace of $\NZ$.

Similarly for the vector $C$, observe $C^I$ contains the entries $(2-c_1, \ldots, 2-c_{n-n_\mfc})$, and $C^{III}$ contains the entries $(-c^\mfm_1, -c^\mfl_1, \ldots, -c^\mfm_{n_\mfc}, -c^\mfl_{n_\mfc})$. For the holonomy vector $H$, we have $H^I$ and $H^{II}$ are zero vectors, while $H^{III}$ contains cusp holonomies.

The gluing equations \eqref{Eqn:NZGluing} can be written as
\begin{equation}
\label{Eqn:MatrixGluingEqns}
\begin{bmatrix} \NZ^I \\ \NZ^{II} \end{bmatrix}
\cdot Z = i \pi \begin{bmatrix} C^I \\ C^{II} \end{bmatrix}.
\end{equation}
The first $n-n_\mfc$ among these equations are given by
\begin{equation}
\label{Eqn:IndptMatrixGluingEqns}
\NZ^I \cdot Z = i \pi C^I.
\end{equation}
We have seen that the rows of $\NZ^I$ span the rows of $\NZ^{II}$, so  
knowing $\NZ^I \cdot Z$ determines $\NZ^{II} \cdot Z$.
But it is perhaps not so clear whether $\NZ^I \cdot Z = i \pi C^I$ implies that $\NZ^{II} \cdot Z = i \pi C^{II}$. However, as we now show, in a hyperbolic situation this is in fact the case. 

\begin{lem}
\label{Lem:IndptGluingEqns}
Suppose the triangulation $\TT$ has a hyperbolic structure. Then a vector $Z \in \C^{2n}$ satisfies equation \eqref{Eqn:MatrixGluingEqns} if and only if it satisfies equation \eqref{Eqn:IndptMatrixGluingEqns}.
\end{lem}

\begin{proof}
Hyperbolic structures (not necessarily complete) give solutions to the gluing equations $Z = (Z_1, Z'_1, \ldots, Z_n, Z'_n) \in \C^{2n}$;
hence the solution space of \eqref{Eqn:MatrixGluingEqns} is nonempty. Since equations \eqref{Eqn:IndptMatrixGluingEqns} are a subset of those of \eqref{Eqn:MatrixGluingEqns}, the solution space of \eqref{Eqn:IndptMatrixGluingEqns} is also nonempty.

Since both matrices $\begin{bmatrix} \NZ^I \\ \NZ^{II} \end{bmatrix}$ and $\NZ^I$ have rank $n-n_\mfc$, the solution spaces of both \eqref{Eqn:MatrixGluingEqns} and \eqref{Eqn:IndptMatrixGluingEqns} have the same dimension: $2n - (n-n_\mfc) = n+n_\mfc$. 
\end{proof}

Thus, some of the gluing equations of \eqref{Eqn:NZGluing}, or equivalently of \eqref{Eqn:MatrixGluingEqns}, are redundant. The same is true of the larger system \eqref{Eqn:NZGluingCuspEqn}. Then $\NZ^\flat$ is a more efficient version of the Neumann-Zagier matrix, containing only necessary information for computing hyperbolic structures.

As discussed at the end of \refsec{Triangulations_gluing_cusp}, the solution spaces of these equations do not in general coincide with spaces of hyperbolic structures. 
The solution space of \eqref{Eqn:IndptMatrixGluingEqns} contains the space of hyperbolic structures on the triangulation $\TT$, but is strictly larger. These equations treat $Z_j$ and $Z'_j$ as independent variables, but of course they are not. In a hyperbolic structure, $z_j = e^{Z_j}$ and $z'_j = e^{Z'_j}$ are related by the equations \eqref{Eqn:NotProd}.

Indeed, the solution space of the linear equations \eqref{Eqn:IndptMatrixGluingEqns} has dimension $n+n_\mfc$, but there are a further $n$ conditions imposed by the relations $z_j + (z_j')^{-1} -1 = 0$ of \eqref{Eqn:NotProd}. As discussed in the proof of \cite[prop. 2.3]{NeumannZagier}, these $n$ conditions are independent and the result is a variety of dimension $n_\mfc$. However, as we just saw, this variety may contain points that do not correspond to hyperbolic tetrahedra. Moreover, it may not contain all hyperbolic structures, as not every hyperbolic structure may be able to be realised by the triangulation $\TT$.

However, by Thurston's hyperbolic Dehn surgery theorem \cite{Thurston:3DGT}, the space of hyperbolic structures on $M$ is also $n_\mfc$-dimensional. So at a point of the variety defined by the linear equations \eqref{Eqn:IndptMatrixGluingEqns} and the nonlinear equations \eqref{Eqn:NotProd} describing a hyperbolic structure, the variety locally coincides with the space of hyperbolic structures. 

We summarise this section with the following statement.
\begin{lem}
\label{Lem:NZIndptGluingEqns}
Let $\TT$ be a hyperbolic triangulation of $M$, labelled so that its Neumann--Zagier matrix $\NZ$ has rows $R^G_1, \ldots, R^G_{n-n_\mfc}$ linearly independent.
\begin{enumerate}
\item 
The logarithmic gluing equations, expressed equivalently by \eqref{Eqn:NZGluing} or \eqref{Eqn:MatrixGluingEqns}, are equivalent to the smaller independent set of equations \eqref{Eqn:IndptMatrixGluingEqns}.
\item
The variety $V$ defined by the solutions of these linear equations \eqref{Eqn:IndptMatrixGluingEqns}, together with the nonlinear equations \eqref{Eqn:NotProd}, has dimension $n_\mfc$. 
The hyperbolic structures on $\TT$ correspond to a subset of $V$. Near a point of $V$ corresponding to a hyperbolic structure on $\TT$, $V$ parametrises hyperbolic structures on $\TT$.
\item
The logarithmic gluing and cusp equations for $\TT$ are equivalent to 
\begin{equation}
  \label{Eqn:NZIndptGluingCuspEqns}
  \pushQED{\qed}
  \NZ^\flat \cdot Z = H^\flat + i \pi C^\flat. \qedhere
  \popQED
\end{equation} 
\end{enumerate} 
\end{lem}

\subsection{Symplectic change of variables}

Dimofte in \cite{DimofteQRCS} considered using the matrix $\SY$ to \emph{change variables} in the logarithmic gluing and cusp equations. 

If $M$ is hyperbolic, by \reflem{NZIndptGluingEqns} the gluing and cusp equations are equivalent to \eqref{Eqn:NZIndptGluingCuspEqns}. Observe that the rows of $\NZ^\flat$ are (up to a factor of $\half$ in the rows $R_k^\mfl$) a subset of the rows of $\SY$. Indeed, obtain $\SY$ from $\NZ^\flat$ by multiplying $R_k^\mfl$ rows by $\half$, and inserting rows $R^\Gamma_1, \ldots, R^\Gamma_{n-n_\mfc}$.

In the equations of \eqref{Eqn:NZIndptGluingCuspEqns} $Z = (Z_1, Z'_1, \ldots, Z_n, Z'_n)^T$ are regarded as variables, and we now change them using $\SY$. 
\begin{defn}
\label{Def:Gamma_variables}
Given a labelled hyperbolic triangulation $\TT$ and a choice of symplectic matrix $\SY$, define the collection of variables \[ \Gamma = \left(\Gamma_1, G_1, \dots, \Gamma_{n-n_\mfc}, G_{n-n_\mfc}, M_1,\half L_1, \dots, M_{n_\mfc}, \half L_{n_\mfc}\right)^T \] by $\Gamma = \SY \cdot Z$.
\end{defn}

In other words,
\[
\Gamma = 
\SY \; \begin{bmatrix} Z_1 \\ Z_1' \\ \vdots \\ Z_n \\ Z_n' \end{bmatrix}
\quad \Leftrightarrow \quad 
\left\{
\begin{array}{rll}
\Gamma_k &= R^\Gamma_k \cdot Z, \quad &\text{for $k \in \{1, \ldots, n-n_\mfc\}$,} \\
G_k &= R^G_k \cdot Z, \quad &\text{for $k \in \{1, \ldots, n-n_\mfc\}$,} \\
M_k &= R_k^\mfm \cdot Z, \quad &\text{for $k\in \{1, \ldots, n_{\mfc}\}$,  and}\\
\half L_k &= \half R_k^\mfl \cdot Z, \quad &\text{for $k\in \{1, \ldots, n_{\mfc}\}$.}
\end{array}
\right.
\]

\begin{lem}
Let $\TT$ be a labelled hyperbolic triangulation, and $\SY$ a matrix defining the variables $\Gamma$. Then the logarithmic gluing and cusp equations are equivalent to
\begin{equation}
\label{Eqn:VarRelations}
  G_k = i\pi \left(2 - c_k \right), \quad
	M_j = \log m_j - i\pi c_j^\mfm, \quad
	L_j = \log\ell_j - i\pi c_j^\mfl.
\end{equation}
\end{lem}

In the new variables, these equations are simplified. Note that the $\Gamma_k$ variables do not appear in \eqref{Eqn:VarRelations}.

\begin{proof}
The first $n-n_\mfc$ rows of \eqref{Eqn:NZIndptGluingCuspEqns} express the gluing equations as $R^G_k \cdot Z = i \pi (2 - c_k)$, for $k \in \{1, \ldots, n-n_\mfc\}$. 
Remaining rows of \eqref{Eqn:NZIndptGluingCuspEqns} express cusp equations as $R_j^\mfm \cdot Z = \log m_j - i \pi c_j^\mfm$ and $R_j^\mfl \cdot Z = \log \ell_j - c_j^\mfl$.
 \end{proof}

The symplectic change of variables involves writing variables $Z$ in terms of the variables $\Gamma$. That is, we need to invert $\SY$.

As $\SY$ is symplectic, $(\SY)^T J (\SY) = J$, so its inverse is given by $\SY^{-1} = -J (\SY)^T J$, or
\begin{equation}
\label{Eqn:SYInverse}
\begin{bmatrix}
d_{1,1}'  & -f'_{1,1} & \cdots & d'_{n-n_\mfc,1}   & -f'_{n-n_\mfc,1} & \half\lambda'_{1,1} & -\mu'_{1,1} & \cdots & \half\lambda'_{n_\mfc,1} & -\mu'_{n_\mfc,1} \\
-d_{1,1}  & f_{1,1}   & \cdots & -d_{n-n_\mfc,1}   & f_{n-n_\mfc,1}   & -\half\lambda_{1,1} & \mu_{1,1} & \cdots & -\half\lambda_{n_\mfc,1} & \mu_{n_\mfc,1} \\
\vdots & \vdots  & \ddots    & \vdots & \vdots     & \vdots    & \vdots       & \ddots & \vdots & \vdots \\
d_{1,n}' & -f_{1,n}'  & \cdots & d'_{n-n_\mfc,n}   & -f'_{n-n_\mfc,n} & \half\lambda'_{1,n} & -\mu'_{1,n} & \cdots & \half\lambda'_{n_\mfc,n} & -\mu'_{n_\mfc,n} \\
-d_{1,n} & f_{1,n}    & \cdots & -d_{n-n_\mfc,n}   & f_{n-n_\mfc,n}   & -\half\lambda_{1,n} & \mu_{1,n} & \cdots & -\half\lambda_{n_\mfc,n} & \mu_{n_\mfc,n}
\end{bmatrix}
\end{equation}
Thus we explicitly express the $Z_j, Z'_j$ in terms of the variables of $\Gamma$, using $Z = (\SY)^{-1} \Gamma$.
\begin{align}
\label{Eqn:Z_inverted}
Z_j &= \sum_{k=1}^{n-n_\mfc} \left( d'_{k,j} \Gamma_k - f'_{k,j} G_k \right) + \half \sum_{k=1}^{n_\mfc} \left( \lambda'_{k,j} M_k - \mu'_{k,j} L_k \right) \\
\label{Eqn:Z'_inverted}
Z'_j &= \sum_{k=1}^{n-n_\mfc} \left( -d_{k,j} \Gamma_k + f_{k,j} G_k \right) + \half \sum_{k=1}^{n_\mfc} \left( -\lambda_{k,j} M_k + \mu_{k,j} L_k \right)
\end{align}

\subsection{Inverting without inverting}

It is possible to explicitly compute a symplectic matrix $\SY$, then invert it, express the variables $Z$ in terms of the variables $\Gamma$ by \eqref{Eqn:Z_inverted}--\eqref{Eqn:Z'_inverted}, and then solve to obtain the A-polynomial. However, we now show that we can perform this calculation without ever having to find $\SY$ or its inverse $\SY^{-1}$ explicitly --- \emph{provided} that we can find a certain sign term.

To see why this should be the case, note the following preliminary observation. Equations \eqref{Eqn:Z_inverted}--\eqref{Eqn:Z'_inverted} express $Z_j$ and $Z'_j$ in terms of the $\Gamma_k$, $G_k$, $M_i$ and $L_i$. The coefficients of the $\Gamma_k$, $M_i$ and $L_i$ are numbers which appear in the Neumann-Zagier matrix. The only coefficients which do not appear in $\NZ$ are the coefficients of the $G_k$. But the gluing equations \eqref{Eqn:VarRelations} say $G_k = i\pi (2-c_k)$, so upon exponentiation these terms only contribute a sign. In other words, up to sign, all the information we need to write the $Z_j$ in terms of the variables $\Gamma_k, G_k, L_i, M_i$ is already in the Neumann-Zagier matrix.

To implement this, observe that the matrix $-J(\NZ^\flat)^T$ shares many columns with $\SY^{-1}$:
\begin{equation}
\label{Eqn:MinusJNZFlatT}
-J(\NZ^\flat)^T
= \begin{bmatrix}
d_{1,1}' & d_{2,1}' & \cdots & d_{n-n_\mfc,1}' & \mu_{1,1}' & \lambda_{1,1}' & \cdots & \mu'_{n_\mfc,1} & \lambda'_{n_\mfc, 1} \\
-d_{1,1} & -d_{2,1} & \cdots & -d_{n-n_\mfc,1} & - \mu_{1,1} & - \lambda_{1,1} & \cdots & -\mu_{n_\mfc,1} & -\lambda_{n_\mfc,1} \\
\vdots & \vdots & \ddots & \vdots & \vdots & \vdots & \ddots &\vdots &\vdots\\
d_{1,n}' & d_{2,n}' & \cdots & d_{n-n_\mfc,n}' & \mu_{1,n}' & \lambda_{1,n}' & \cdots & \mu_{n_\mfc, n}' & \lambda_{n_\mfc, n}' \\
-d_{1,n} & -d_{2,n} & \cdots & -d_{n-n_\mfc,n} & -\mu_{1,n} & -\lambda_{1,n} & \cdots & -\mu_{n_\mfc,n} & -\lambda_{n_\mfc,n} 
\end{bmatrix}
\end{equation}
In particular, for any quantities $A_1, \ldots, A_{n-n_\mfc}, A_1^\lambda, A_1^\mu, \ldots, A_{n_\mfc}^\lambda, A_{n_\mfc}^\mu$,
\begin{gather*}
  \SY^{-1}
    \left[ A_1 \quad 0 \quad A_2 \quad 0 \quad \cdots \quad A_{n-n_\mfc} \quad 0 \quad A_1^\lambda \quad A_1^\mu  \quad \cdots \quad A_{n_\mfc}^{\lambda} \quad A_{n_\mfc}^\mu \right]^T = \\
-J(\NZ^\flat)^T \left[ A_1 \quad A_2 \quad \cdots \quad A_{n-n_\mfc} \quad -A_1^\mu \quad \half A_1^\lambda \quad \cdots \quad -A_{n_\mfc}^\mu \quad \half A_{n_\mfc}^\lambda
\right]^T
\end{gather*}

Splitting up the $\Gamma_k$ and $G_k$ terms, using \refdef{Gamma_variables} and informed by the gluing and cusp equations \eqref{Eqn:VarRelations}, we obtain
\begin{equation}
  \label{Eqn:GeneralChangeVars}
  Z
= \SY^{-1} \cdot \Gamma
= -J(\NZ^\flat)^T \overline{\Gamma} 
+ \SY^{-1} \overline{G}, 
\end{equation}
where $\overline{\Gamma}$ is the vector
\[
\overline{\Gamma} = \left[ \Gamma_1, \cdots, \Gamma_{n-n_\mfc}, -\half\log\ell_1, \half\log m_1, \cdots, -\half\log\ell_{n_\mfc}, \half\log m_{n_\mfc}\right]^T
\]
and $\overline{G}$ is
\[
\left[ 0, G_1, \cdots, 0, G_{n-n_\mfc}, (M_1-\log m_1), \half(L_1-\log \ell_1), \cdots, (M_{n_\mfc}-\log m_{n_\mfc}), \half (L_{n_\mfc}-\log \ell_{n_\mfc}) \right]^T
\]
The first term $-J(\NZ^{\flat})^T\overline{\Gamma}$ of \eqref{Eqn:GeneralChangeVars} only involves $\NZ$. The final vector $\overline{G}$ consists of the precise quantities which are fixed to be constants by the gluing and completeness equations \eqref{Eqn:VarRelations}. Indeed, \eqref{Eqn:VarRelations} says precisely that the final vector in equation \eqref{Eqn:GeneralChangeVars} is a vector of constants essentially identical in content to $\pi i C^\flat$. We define
\[
C^\# = \left[ 0, 2 - c_1, 0, 2 - c_2, \ldots, 0, 2 - c_{n-n_\mfc}, -c_1^{\mfm}, - \half c_1^{\mfl}, \ldots, -c_{n_\mfc}^{\mfm}, -\half c_{n_\mfc}^\mfl \right]^T,
\]
which is $C^\flat$, with some zeroes inserted, and some factors of one half. So the final vector in \eqref{Eqn:GeneralChangeVars} is set to $\pi iC^\#$, and we obtain the following.
\begin{prop}
\label{Prop:ChangeOfVars}
Given a hyperbolic triangulation, labelled so that its Neumann-Zagier matrix $\NZ$ has rows $R^G_1, \ldots, R^G_{n-n_\mfc}$ linearly independent,
and $\SY$ a
matrix defining the variables $\Gamma$, the logarithmic gluing and cusp equations are equivalent to
\begin{equation}
  \label{Eqn:CuspGluingChangeVars}
  \pushQED{\qed}
  Z = (-J)(\NZ^\flat)^T \overline{\Gamma} +\pi i \;  \SY^{-1} C^\#. \qedhere
  \popQED
\end{equation}
\end{prop}

Once we find a vector $B = \SY^{-1} C^\#$, \refprop{ChangeOfVars} allows us to express the $Z_j$ and $Z'_j$ in terms of the variables $\Gamma_1, \ldots, \Gamma_{n-1}$, and the holonomies $\ell_k, m_k$ of the longitudes and meridians, 
using only information already available in the Neumann-Zagier matrix. There is no need to find the extra vectors $R^\Gamma_k$ of the symplectic basis, or the matrix $\SY$.
If in addition $B$ is an \emph{integer} vector, then when we exponentiate \eqref{Eqn:CuspGluingChangeVars} to obtain the tetrahedron parameters $z_j = e^{Z_j}$ and $z'_j = e^{Z'_j}$, $B$ determines a sign. Hence we refer to this term as a sign term.

The approach outlined above may sound paradoxical: we avoid calculating the symplectic matrix $\SY$, by finding a vector $B=\SY^{-1} C^\#$. This seems to involve the symplectic matrix $\SY$ anyway! However, in the next section we show we can find $B$ by solving a simpler equation, involving only the Neumann-Zagier matrix, and then \emph{choose} $\SY$ so that $B = \SY^{-1} C^\#$. That is, we may use the flexibility in choosing $R^\Gamma_k$ of Remark~\ref{Rmk:NonUnique} to find appropriate $\SY$.

\subsection{The sign term}

We now demonstrate the existence of an $\SY$ and an integer vector $B$ satisfying $\SY \cdot B = C^\#$.

The rows of the matrix equation $\SY \cdot B = C^\#$ are 
\begin{gather}
\label{Eqn:RkGammaDotB}
R_k^\Gamma \cdot B = 0, \quad \text{for $k = 1, \ldots, n-n_\mfc$,} \\
\label{Eqn:GSignEqns}
R_k^G \cdot B = 2- c_k, \quad \text{for $k = 1, \ldots, n-n_\mfc$,} \\
\label{Eqn:CuspSignEqns}
R_k^\mfm \cdot B = - c_k^\mfm, \quad 
R_k^\mfl\cdot B = - c_k^\mfl, \quad \text{for $k=1, \ldots, n_{\mfc}$.}
\end{gather}
Equations \eqref{Eqn:GSignEqns}--\eqref{Eqn:CuspSignEqns} are exactly the equations in the rows of a matrix equation with $\NZ^\flat$:
\begin{equation}\label{Eqn:SignEquation}
  \NZ^\flat \cdot B = C^\flat.
\end{equation}

This equation has been studied by Neumann; it is known to always have an integer solution.
\begin{thm}[Neumann \cite{Neumann}, Theorem 2.4] \
\label{Thm:NeumannSolution}
\begin{enumerate}
\item
There exists an integer vector $B$ satisfying
$\NZ \cdot B = C$.
\item
Given an integer vector $B_0$ such that $\NZ \cdot B_0 = C$, the set of integer solutions to $\NZ \cdot B = C$ includes
\[
B_0 + \Span_\Z \left( J R^G_1, \ldots, J R^G_n \right)
= \left\{ B_0 + \sum_{k=1}^n a_k J R^G_k \; \mid \; a_1, \ldots, a_n \in \Z \right\}.
\]
\end{enumerate}
\end{thm}
Neumann's result is more precise, incorporating a parity condition on $B$ not needed here. Additionally, we will not need part (ii) of the theorem until later, but we state it now.
Note that, by taking a subset of the rows, or equations, $\NZ \cdot B = C$ implies $\NZ^\flat \cdot B = C^\flat$.

In order to solve $\SY \cdot B = C^\#$, it remains to satisfy the equations \eqref{Eqn:RkGammaDotB}. As discussed above, we do this not by adjusting $B$, but by judicious choice of the vectors $R^\Gamma_k$, and hence the matrix $\SY$. Recall from \refsec{symplectic_top_properties_NZ} that there is substantial freedom in choosing the vectors $R^\Gamma_k$.
But first we deal with a technical condition on the triangulation, which we need for the argument.
Recall $c_k = \sum_{j=1}^n c_{k,j}$ (\refdef{NZConsts}), where $c_{k,j}$ is the number of $c$-edges of the tetrahedron $\Delta_j$ identified to edge $E_k$ (\refdef{abc_edges}). So $c_k$ is just the number of $c$-edges of tetrahedra identified to $E_k$.

\begin{lem}\label{Lem:GoodLabelling}
Any triangulation of $M$ has a labelling such that 
\begin{enumerate}
\item its Neumann-Zagier matrix $\NZ$ has rows $R^G_1, \ldots, R^G_{n-n_\mfc}$ linearly independent, and 
\item
there exists $k\in\{1, \ldots, n-n_\mfc\}$ with $c_k \neq 2$.
\end{enumerate}
\end{lem}

In other words, the conclusion of the lemma requires that some edge be incident to a number of $c$-edges other than $2$.
In fact, we will see that one can start from any labelled triangulation, and it suffices to relabel the vertices of at most one tetrahedron, and possibly reorder some edges. Moreover, we can choose any edge $E_k$ with nonzero $R^G_k$, and adjust so that this particular edge is incident to $c_k \neq 2$ $c$-edges.

The proof of \reflem{GoodLabelling} requires that $n>n_\mfc$. In fact, Adams and Sherman proved that $n\geq 2n_\mfc$ for any finite volume orientable hyperbolic 3-manifold with $n_\mfc$ cusps~\cite{AdamsSherman}.

\begin{proof}[Proof of \reflem{GoodLabelling}]
Take a labelled triangulation $\TT$ of $M$. 
Choose some $k \in \{1, \ldots, n \}$ such that $R^G_k$ is nonzero. (Such $k$ certainly exists since the $R^G_k$ span a space of rank $n - n_\mfc \geq 1$.)
We claim that if $c_k = 2$, then $\TT$ can be relabelled so that $c_k \neq 2$.

Let $\Delta_t$ be a tetrahedron of $\TT$. The relabellings of $\Delta_t$ have the effect of cyclically permuting the $a$-, $b$- and $c$-edges, 
hence cyclically permuting the triple $(a_{k,t}, b_{k,t}, c_{k,t})$; however other terms $c_{k,j}$ in the sum for $c_k$ are unchanged. Hence, if one of $a_{k,t}$ or $b_{k,t}$ is not equal to $c_{k,t}$, then a relabelling of $\Delta_t$ will change $c_k$ to a distinct value, not $2$, as desired. Otherwise, all relabellings of $\Delta_t$ leave $c_k = 2$, and we have $a_{k,t} = b_{k,t} = c_{k,t}$, so $d_{k,t} = d'_{k,t} = 0$ (\refdef{NZConsts}). 

The above argument applies to any tetrahedron $\Delta_t$ of $\TT$. Thus, if every relabelling of any single tetrahedron leaves $c_k = 2$, then the numbers $d_{k,t} = d'_{k,t} = 0$ for all $t \in \{1, \ldots, n\}$. But these are precisely the entries in the vector $R^G_k$ forming a row of $\NZ^\flat$, so $R^G_k = 0$, contradicting $R^G_k \neq 0$ above. This contradiction proves the claim. Moreover, after relabelling the tetrahedron, there still exists $t \in \{1, \ldots, n\}$ such that $a_{k,t}, b_{k,t}, c_{k,t}$ are not all equal, and hence $R^G_k$ is not zero.

Thus, there exists a relabelling of a single tetrahedron that makes $c_k \neq 2$, and $R^G_k$ remains nonzero. Call the resulting labelled triangulation $\TT'$ and Neumann-Zagier matrix $\NZ'$. 
Now by \refthm{NeumannZagier}(iv), the first $n$ row vectors of $\NZ'$ span an $(n-n_{\mfc})$-dimensional space. Hence we may relabel the edges so that the edges labelled $1, \ldots, n-n_\mfc$ have linearly independent row vectors, and our chosen edge is among them. This relabelling satisfies the lemma. 
\end{proof}

For a triangulation as in \reflem{GoodLabelling}, the nonzero entry of $C^\flat$ provides the leverage to make a choice of vectors $R^\Gamma_k$ so that they satisfy \eqref{Eqn:RkGammaDotB}.
\begin{lem}\label{Lem:AnyNeumannVectorWorks}
Suppose that $\TT$ is labelled to satisfy \reflem{GoodLabelling}. Let $B \in \Z^{2n}$ be a vector satisfying $\NZ^\flat \cdot B = C^\flat$. Then there exist vectors $R^\Gamma_1, \dots, R^\Gamma_{n-n_\mfc}$ in $\Q^{2n}$ such that 
\begin{enumerate}
\item $(R^\Gamma_1, R^G_1, \ldots, R^\Gamma_{n-n_\mfc}, R^G_{n-n_\mfc}, R_1^\mfm, \half R_1^\mfl, \cdots, R_{n_\mfc}^\mfm, \half R_{n_\mfc}^\mfl)$ forms a symplectic basis, and
\item for all $j \in \{1, \ldots, n-n_\mfc\}$ we have $R^\Gamma_j \cdot B = 0$.
\end{enumerate}
\end{lem}

\begin{proof}
We start from arbitrary choices of the $R^\Gamma_k \in \Q^{2n}$ such that
\[
\left( R^\Gamma_1, R^G_1, \ldots, R^\Gamma_{n-n_\mfc}, R^G_{n-n_\mfc}, R_1^\mfm, \half R_1^\mfl, \ldots, R_{n_\mfc}^\mfm, \half R_{n_\mfc}^\mfl \right)
\]
is a symplectic basis.

\reflem{symplectomorphisms_fixing_fs} allows us to adjust the $R^\Gamma_k$, without changing any $R^G_k$, $R_j^\mfm$ or $R_j^\mfl$, so that we still have a symplectic basis. In particular, we may make the following modifications:
\begin{enumerate}
\item
For $j\neq k \in \{1, \ldots, n-n_\mfc\}$, and $a \in \R$, map $R^\Gamma_j \mapsto R^\Gamma_j + a R^G_k$, $R^\Gamma_k \mapsto R^\Gamma_k + a R^G_j$.
\item
Take $j \in \{1, \ldots, n-n_\mfc \}$ and $a \in \R$, and map $R^\Gamma_j \mapsto R^\Gamma_j + a R^G_j$.
\end{enumerate}

Let $R_j^\Gamma \cdot B = a_j$. We will adjust the $R_j^\Gamma$ until all $a_j = 0$.

We claim there exists a $k \in \{1, \ldots, n-n_\mfc \}$ such that $R^G_k \cdot B \neq 0$. Indeed, as $\TT$ satisfies \reflem{GoodLabelling}, there exists a $k \in \{1, \ldots, n-n_\mfc \}$ such that $c_k \neq 2$. Then the $k$th row of the equation $\NZ^\flat \cdot B = C^\flat$ says that $\alpha := R^G_k \cdot B = 2 - c_k$, which is nonzero as claimed. 

First, modify $R^\Gamma_k$ by (ii), replacing $R_k^\Gamma$ with $(R_k^\Gamma)' = R_k^\Gamma - \frac{a_k}{\alpha} R_k^G$. Then $(R_k^\Gamma)' \cdot B = R_k^\Gamma \cdot B - \frac{a_k}{\alpha} R_k^G \cdot B = 0$. Thus the modification makes $a_k = 0$; the other $a_j$ are unchanged. 

Now consider $j \neq k$. If $R_j^G \cdot B \neq 0$, modify $R^\Gamma_j$ by (ii) to set $a_j = 0$. Otherwise, $R_j^G \cdot B = 0$ and modify $R_j^\Gamma$ and $R_k^\Gamma$ by (i), replacing them with $(R_j^\Gamma)' = R_j^\Gamma - \frac{a_j}{\alpha} R_k^G$ and $(R_k^\Gamma)' = R_k^\Gamma - \frac{a_j}{\alpha} R_j^G$ respectively. Then $(R_j^\Gamma)' \cdot B = R_j^\Gamma \cdot B - \frac{a_j}{\alpha} R_k^G \cdot B = 0$ and $(R_k^\Gamma)' \cdot B = R_k^\Gamma \cdot B - \frac{a_j}{\alpha} R_j^G \cdot B = a_k = 0$. Again the effect is to set $a_j = 0$ and leave the other $a_i$ unchanged. 

Modifying $R^\Gamma_j$ in this way for each $j \neq k$, we obtain the desired vectors.
\end{proof}

We summarise the result of this section in the following proposition.
\begin{prop}
\label{Prop:ZConvert}
Let $\TT$ be a hyperbolic triangulation labelled to satisfy \reflem{GoodLabelling}. Let $B$ be an integer vector such that $\NZ^\flat \cdot B = C^\flat$ (such a vector exists by \refthm{NeumannSolution}).
Then there exists a symplectic matrix $\SY$ defining variables $\Gamma$, such that the logarithmic gluing and cusp equations are equivalent to the equation
\begin{equation}
  \label{Eqn:CuspGluingNZVars}
  \pushQED{\qed}
  Z = (-J)(\NZ^\flat)^T \overline{\Gamma} + \pi iB. \qedhere
  \popQED
\end{equation}
\end{prop}

We have now realised our claim of ``inverting without inverting". Proposition~\ref{Prop:ZConvert} allows us to convert the variables $Z_i, Z'_i$ into the variables $\Gamma_i$, together with the cusp holonomies $\ell_i, m_i$, without having to actually calculate the vectors $R_i^\Gamma$ or the matrix $\SY$! The only information we need is the Neumann-Zagier matrix $\NZ$, and the integer vector $B$ such that $\NZ^\flat \cdot B = C^\flat$.

\subsection{The A-polynomial from gluing equations and from Ptolemy equations}

Suppose that $n_\mfc = 1$, we have a labelled triangulation $\TT$ satisfying \reflem{GoodLabelling}, and a vector $B = (B_1, B'_1, \ldots, B_n, B'_n)^T$ such that $\NZ^\flat \cdot B = C^\flat$. 

Proposition~\ref{Prop:ZConvert} converts the logarithmic gluing and cusp equations --- linear equations --- into the variables $\Gamma_1, \ldots, \Gamma_{n-1}$, together with the cusp holonomies $m,\ell$. We now convert the nonlinear equations \eqref{Eqn:NotProd} into these variables.

We first convert to the exponentiated variables $z_j$. Let $\gamma_j = e^{\Gamma_j}$. 
Using equation \eqref{Eqn:CuspGluingNZVars}, and the known form of $(-J)(\NZ^\flat)^T$ from \eqref{Eqn:MinusJNZFlatT}, we obtain
\begin{equation}
\label{Eqn:zjInTermsOfGammalm}
z_j = 
\left( -1 \right)^{B_j} \ell^{-\mu_j'/2} m^{\lambda_j'/2} \prod_{k=1}^{n-1} \gamma_k^{d_{k,j}'}, \quad\quad
z_j' = 
\left( -1 \right)^{B_j'} \ell^{\mu_j/2} m^{-\lambda_j/2} \prod_{k=1}^{n-1} \gamma_k^{-d_{k,j}}.
\end{equation}

Then the nonlinear equation \eqref{Eqn:NotProd} for the tetrahedron $\Delta_j$ becomes
\[
\left( -1 \right)^{B_j} \ell^{-\mu_j'/2} m^{\lambda_j'/2} \prod_{k=1}^{n-1} \gamma_k^{d_{k,j}'}
+
\left( -1 \right)^{B_j'} \ell^{-\mu_j/2} m^{\lambda_j/2} \prod_{k=1}^{n-1} \gamma_k^{d_{k,j}}
- 1
= 0.
\]
Since $d_{k,j} = a_{k,j} - c_{k,j}$ and $d'_{k,j} = b_{k,j} - c_{k,j}$ (\refdef{NZConsts}), we may multiply through by $\gamma^{c_{k,j}}$; then the exponents become the incidence numbers $a_{k,j}, b_{k,j}, c_{k,j}$ of the various types of edges of tetrahedra with edges of the triangulation (\refdef{abc_edges}).
\begin{equation}
\label{Eqn:NotProd_in_gammas}
\left( -1 \right)^{B_j} \ell^{-\mu_j'/2} m^{\lambda_j'/2} \prod_{k=1}^{n-1} \gamma_k^{b_{k,j}}
+
\left( -1 \right)^{B_j'} \ell^{-\mu_j/2} m^{\lambda_j/2} \prod_{k=1}^{n-1} \gamma_k^{a_{k,j}}
- 
\prod_{k=1}^{n-1} \gamma_k^{c_{k,j}}
= 0.
\end{equation}
Each product in the above expression is simpler than it looks: it is a polynomial of total degree at most 2 in the $\gamma_k$, by \reflem{SumToTwo}! The product $\prod_{k=1}^{n-1} \gamma_k^{a_{k,j}}$ has $j$ fixed, referring to the tetrahedron $\Delta_j$. The product is over the various edges $E_k$ of the triangulation; the exponent $a_{k,j}$ is the incidence number of the $a$-edges of $\Delta_j$ with the edge $E_k$. But $\Delta_j$ only has two $a$-edges, so at most two $a_{k,j}$ are nonzero, and the $a_{k,j}$ sum to 2 as in \eqref{Eqn:sum_of_abcs}.

Recall the notation $j(\mu \nu)$ of \refdef{jmunu}. For fixed $j$, the only nonzero $a_{k,j}$ are $a_{j(01),j}$ and $a_{j(23),j}$ (which may be the same term). Thus the product $\prod_{k=1}^{n-1} \gamma_k^{a_{k,j}}$ is equal to the product of $\gamma_{j(01)}$ and $\gamma_{j(23)}$, with the caveat that $\gamma_n$ does not appear in the product. Indeed, in \refdef{Gamma_variables} we only define $\Gamma_1, \ldots, \Gamma_{n-1}$, so only $\gamma_1, \ldots, \gamma_{n-1}$ are defined. However, it is worthwhile to introduce $\gamma_n$ as a formal variable.
\begin{defn}
\label{Def:PtolemyEqn}
Let $\TT$ be a labelled triangulation of a 3-manifold with one cusp, and let $B$ be an integer vector such that $\NZ^\flat \cdot B = C^\flat$. The \emph{Ptolemy equation} of the tetrahedron $\Delta_j$ is
\[
\left( -1 \right)^{B'_j} \ell^{-\mu_j/2} m^{\lambda_j/2} \gamma_{j(01)} \gamma_{j(23)}
+
\left( - 1 \right)^{B_j} \ell^{-\mu'_j/2} m^{\lambda'_j/2} \gamma_{j(02)} \gamma_{j(13)}
-
\gamma_{j(03)} \gamma_{j(12)}
= 0.
\]
The \emph{Ptolemy equations} of $\TT$ consist of Ptolemy equations for each tetrahedron of $\TT$.
\end{defn}
Equation \eqref{Eqn:NotProd_in_gammas} is the Ptolemy equation for $\Delta_j$, with the formal variable $\gamma_n$ set to $1$.

Let us now put the work of this section together.

\begin{thm}
\label{Thm:Ptolemy_Apoly}
Let $\TT$ be a hyperbolic triangulation of a one-cusped $M$, labelled to satisfy \reflem{GoodLabelling}. When we solve the system of Ptolemy equations of $\TT$ in terms of $m$ and $\ell$, setting $\gamma_n = 1$ and eliminating the variables $\gamma_1, \ldots, \gamma_{n-1}$, we obtain a factor of the $\PSL(2,\C)$ A-polynomial, which is also the polynomial of \refthm{Champ}.
\end{thm}
(Note that the polynomial described here, arising by eliminating variables from a system of equations, is only defined up to multiplication by units, and the equality of polynomials here should be interpreted accordingly.)

\begin{proof}[Proof of \refthm{Ptolemy_Apoly}]
Theorem~\ref{Thm:Champ} tells us that solving equations
\eqref{Eqn:Z''Eqn}--\eqref{Eqn:NotProd}, \eqref{Eqn:Gluing} and \eqref{Eqn:CuspEqn} for $m$ and $\ell$, eliminating the variables $z_j, z'_j, z''_j$, yields a factor of the $\PSL(2,\C)$ A-polynomial. By \reflem{LogEqns_NZ}, a solution of the logarithmic gluing and cusp equations, after exponentiation, gives a solution of \eqref{Eqn:Z''Eqn}, \eqref{Eqn:Gluing} and \eqref{Eqn:CuspEqn}; and conversely any solution of \eqref{Eqn:Z''Eqn}, \eqref{Eqn:Gluing} and \eqref{Eqn:CuspEqn} has a logarithm solving the logarithmic gluing and cusp equations.

By \refprop{ZConvert}, after introducing appropriate $B$ and $\SY$ and variables $\Gamma$, which all exist, the logarithmic gluing and cusp equations are equivalent to \eqref{Eqn:CuspGluingNZVars}. Exponentiating gives us that the equations \eqref{Eqn:zjInTermsOfGammalm} imply \eqref{Eqn:Z''Eqn}, \eqref{Eqn:Gluing} and \eqref{Eqn:CuspEqn}. Combining these with \eqref{Eqn:NotProd} yields the equations \eqref{Eqn:NotProd_in_gammas}, one for each tetrahedron. Therefore, any solution of the equations \eqref{Eqn:NotProd_in_gammas} for $\gamma_1, \ldots, \gamma_{n-1}, m, \ell$ yields a solution  of \eqref{Eqn:Z''Eqn}--\eqref{Eqn:NotProd}, \eqref{Eqn:Gluing} and \eqref{Eqn:CuspEqn}.  
Conversely, any solution of \eqref{Eqn:Z''Eqn}--\eqref{Eqn:NotProd}, \eqref{Eqn:Gluing} and \eqref{Eqn:CuspEqn} has a logarithm satisfying the logarithmic gluing and cusp equations, hence yields solutions of \eqref{Eqn:NotProd_in_gammas}.

Thus the pairs $(\ell, m)$ arising in solutions of (\eqref{Eqn:Z''Eqn}--\eqref{Eqn:NotProd}, \eqref{Eqn:Gluing} and \eqref{Eqn:CuspEqn}) are those arising in solutions of \eqref{Eqn:NotProd_in_gammas}. The latter equations  are the Ptolemy equations of $\TT$ with $\gamma_n$ set to $1$.
Thus, the $(\ell, m)$ satisfying the polynomial obtained by solving the Ptolemy equations with $\gamma_n = 1$ are also those satisfying the polynomial of \refthm{Champ}.
\end{proof}

\begin{corollary}
\label{Cor:Ptolemy_SL2C_Apoly}
With $\TT$ and $M$ as above, let $A_0(\calL,\calM)$ denote the factor of the $\SL(2,\C)$ A-polynomial describing hyperbolic structures on $\TT$.
Letting $\mathcal{L} = \ell^{1/2}$ and $\mathcal{M} = m^{1/2}$ and solving the Ptolemy equations with $\gamma_n = 1$ as above, we obtain a polynomial in $\mathcal{M}$ and $\mathcal{L}$ which contains a factor either $A_0(\calL,\calM)$ or $A_0(-\calL, \calM)$. 
\end{corollary}

\begin{proof}
Suppose $(\mathcal{L},\mathcal{M})$ lies in the zero set of the factor of the $\SL(2,\C)$ A-polynomial describing hyperbolic structures on $\TT$. Then there is a representation $\pi_1(M) \To SL(2,\C)$ sending the longitude to a matrix with eigenvalues $\mathcal{L},\mathcal{L}^{-1}$ and the meridian to a matrix with eigenvalues $\mathcal{M},\mathcal{M}^{-1}$. Projecting to $\PSL(2,\C)$ we have the holonomy of a hyperbolic structure on $\TT$ whose cusp holonomies are given by $\mathcal{L}^2 = \ell$ and $\mathcal{M}^2 = m$ respectively. Hence $(\ell, m)$ and the tetrahedron parameters of the hyperbolic structure solve the gluing and cusp equations $\TT$, hence satisfy the polynomial of \refthm{Ptolemy_Apoly}.
\end{proof}

\section{Dehn fillings and triangulations}\label{Sec:Dehn}

\subsection{Layered solid tori}\label{Sec:LST}

Suppose we have a triangulation where a cusp $\mfc_1$ meets exactly two tetrahedra $\Delta^\mfc_1$ and $\Delta^\mfc_2$ in exactly one ideal vertex per tetrahedron. (We show in Appendix~A, \refprop{NiceTriangulation}, that such a triangulation can be constructed for quite general manifolds with two or more cusps.)
These two tetrahedra together give a triangulation of a manifold homeomorphic to $T^2\times [0,\infty)$ with a single point removed from $T^2\times\{0\}$. The boundary component $T^2\times\{0\}$ of $\Delta^\mfc_1\cup\Delta^\mfc_2$ is a punctured torus, triangulated by the two ideal triangles of $\bdy \Delta^\mfc_1$ and $\bdy\Delta^\mfc_2$ that do not meet the cusp $\mfc_1$.
We will remove $\Delta^\mfc_1 \cup \Delta^\mfc_2$ from our triangulated manifold, and obtain a space with boundary a punctured torus, triangulated by the same two ideal triangles. We will then replace $\Delta^\mfc_1\cup\Delta^\mfc_2$ by a solid torus with a triangulation such that the boundary is a triangulated once-punctured torus. This will give a triangulation of the Dehn filling.

A \emph{layered solid torus} is a triangulation of a solid torus, first described by Jaco and Rubinstein~\cite{JacoRubinstein:LST}; see also \cite{GueritaudSchleimer}. When working with ideal triangulations, as in our situation, the boundary of a layered solid torus consists of two ideal triangles whose union is a triangulation of a punctured torus. The space of all two-triangle triangulations of punctured tori is described by the Farey graph. 
A layered solid torus can be built using the combinatorics of the Farey graph.

Recall first the construction of the Farey triangulation of $\HH^2$. We view $\HH^2$ in the disc model, with antipodal points $1/0$ and $0/1$ in $\bdy\HH^2$ lying on a horizontal line through the centre of the disc, and $1/1$ at the north pole, $-1/1$ at the south pole. Two points $a/b$ and $c/d$ in $\Q \cup \{\infty\}\subset\bdy\HH^2$ have distance measured by
\[ \iota (a/b, c/d) = |ad - bc|. \]
Here $\iota (\cdot, \cdot)$ denotes geometric intersection number of slopes on a punctured torus. 
We draw an ideal geodesic between each pair $a/b$, $c/d$ with $|ad-bc|=1$. This gives the \emph{Farey triangulation}. The dual graph of the Farey triangulation is an infinite trivalent tree, which we denote by $\F$.

Any triangulation of a once-punctured torus consists of three slopes on the boundary of the torus, with each pair of slopes having geometric intersection number $1$. Denote the slopes by $f$, $g$, $h$. This triple determines a triangle in the Farey triangulation. Moving across an edge $(f,g)$ of the Farey triangulation, we arrive at another triangle whose vertices include $f$ and $g$; but the slope $h$ is replaced with some other slope $h'$. This corresponds to changing the triangulation on the punctured torus, replacing lines of slope $h$ with lines of slope $h'$.

When we wish to perform a Dehn filling by attaching a solid torus to a triangulated once-punctured torus, there are four important slopes involved. Three of the slopes are the slopes of the initial triangulation of the once-punctured solid torus. For example, these might be $0/1$, $1/0$, and $1/1$. We will typically denote the slopes by $(f,g,h)$. These determine an initial triangle $T_0$ in the Farey graph. The other important slope is $r$, the slope of the Dehn filling.

Now consider the geodesic in $\HH^2$ from the centre of 
$T_0$
to the slope $r \subset \bdy \HH^2$. This geodesic passes through a sequence of distinct triangles in the Farey graph, which we denote $T_0, T_1, \ldots, T_{N+1}$. Each $T_{j+1}$ is adjacent to $T_j$. We regard this as a walk or voyage through the triangulation; more precisely, we can regard $T_0, \ldots, T_N$ as forming an oriented path in the dual tree $\F$ without backtracking.
The slope $r$ appears as a vertex of the final triangle $T_{N+1}$, but not in any earlier triangle.

We build the layered solid torus by stacking tetrahedra $\Delta_0, \Delta_1, \ldots$ onto the punctured torus, replacing one set of slopes $T_0$ with another $T_1$, then another $T_2$, and so on. That is, two consecutive punctured tori always have two slopes in common and two that differ by a diagonal exchange. The diagonal exchange is obtained in three-dimensions by layering a tetrahedron onto a given punctured torus such that the diagonal on one side matches the diagonal to be replaced. See \reffig{LST}.

\begin{figure}
\begin{center}
  \includegraphics{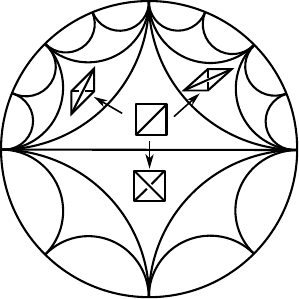}
  \caption{Constructing a layered solid torus}
  \label{Fig:LST}
\end{center}
\end{figure}

For each edge crossed in the path from $T_0$ to $T_N$, layer on a tetrahedron, obtaining a collection of tetrahedra homotopy equivalent to $T^2\times I$. After gluing $k$ tetrahedra $\Delta_0, \ldots, \Delta_{k-1}$, the side $T^2\times\{0\}$ has the triangulation whose slopes are given by $T_0$, and the side $T^2\times\{1\}$ has slopes given by $T_k$.
Two of the faces of $\Delta_{k-1}$ are glued to triangles of the previous layer, with slopes given by $T_{k-1}$, and the other two faces form a triangulation of the ``top" boundary $T^2 \times \{1\}$; this triangulation has slopes given by $T_k$.
Continue until $k=N$, obtaining a triangulated complex consisting of $N$ tetrahedra $\Delta_0, \ldots, \Delta_{N-1}$, with boundary consisting of two once-punctured tori, one triangulated by $T_0$ and the other by $T_N$.

Recall we are trying to obtain a triangulation of a solid torus for which the slope $r$ is homotopically trivial. Note that $r$ is a diagonal of the triangulation $T_{N}$. That is, a single diagonal exchange replaces the triangulation $T_N$ with $T_{N+1}$; and $T_{N+1}$ is a triangulation consisting of two slopes $s$ and $t$ in common with $T_N$, together with the slope $r$, which cuts across a slope $r'$ of $T_N$. To homotopically kill the slope $r$, fold the two triangles of $T_N$ across the diagonal slope $r'$, as in \reffig{FoldTriangles}. Gluing the two triangles on one boundary component of $T^2\times I$ in this manner gives a quotient that is homeomorphic to a solid torus, with boundary still triangulated by $T_0$. Inside, the slopes $s$ and $t$ are identified. The slope $r$ has been folded onto itself, meaning it is now homotopically trivial. Note that $N$ is the number of ideal tetrahedra in the layered solid torus. 

\begin{figure}
  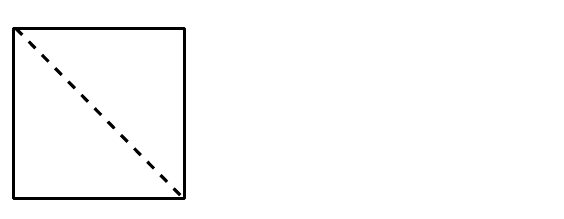
  \caption{Folding makes the diagonal slope $r$ homotopically trivial.}
  \label{Fig:FoldTriangles}
\end{figure}

There are two exceptional cases. If $N=0$ then no tetrahedra are layered to form a layered solid torus. Instead, we fold across existing faces to homotopically ``kill'' the slope $r$ that lies in one of the three Farey triangles adjacent to $(f,g,h)$. This can be considered as attaching a degenerate layered solid torus, consisting of a single face, folded into a M\"obius band. 

There is one other \emph{extra-exceptional} case. In this case, the slope $r$ is one of $f, g, h$. We can triangulate the Dehn filling: for example we can attach a tetrahedron covering the edge corresponding to $r$, performing a diagonal exchange on the once-punctured torus triangulation, then immediately fold the two new faces across the diagonal, creating an edge with valence one. This case will be ignored in the arguments below.

\subsection{Notation for a voyage in the Farey triangulation}
\label{Sec:FareyVoyage}

We now give notation to keep closer track of the slopes obtained at each stage of the construction of a layered solid torus. 

As we have seen, each tetrahedron $\Delta_{k-1}$ replaces one set of slopes with another; the set of slopes corresponding to the triangle $T_{k-1}$ in the Farey triangulation is replaced with the set of slopes with the triangle $T_{k}$. Thus, we associate to $\Delta_{k-1}$ an oriented edge of the dual tree $\F$ of the Farey triangulation, from $T_{k-1}$ to $T_{k}$.

As $\F$ is an infinite trivalent tree, at each stage of a path in $\F$ without backtracking, after we begin and before we stop, there are two choices: turning left or right. As is standard, we denote L and R for these choices. Note that the choice of L or R is not well-defined when moving from $T_0$ to $T_1$, but thereafter the choice of L or R is well-defined. Thus, to the path $T_0, T_1, \ldots, T_{N+1}$ in $\F$, there is a word of length $N$ in the letters $\{\mbox{L,R}\}$. We call this word $W$. The $j$th letter of $W$ corresponds to the choice of L or R when moving from $T_j$ to $T_{j+1}$, which also corresponds to adding tetrahedron $\Delta_j$. 

As we voyage at each stage from $T_{k}$ to $T_{k+1}$, we pass through an edge $e_k$ of the Farey triangulation (dual to the corresponding edge of $\F$), which has one endpoint to our left (port) and one to our right (starboard).\footnote{As ``left" and ``right" are used in the context or the previous paragraph, we use the nautical terminology here.} We leave behind an old slope, one of the slopes of $T_{k}$, namely the one not occurring in $T_{k+1}$. And we head towards a new slope, namely the slope of $T_{k+1}$ which is not one of $T_k$.
\begin{defn} 
As we pass from $T_{k}$ to $T_{k+1}$, across the edge $e_k$, the slope corresponding to 
\begin{enumerate}
\item the endpoint of $e_k$ to our left is denoted $p_k$ (for port);
\item the endpoint of $e_k$ to our right is denoted $s_k$ (for starboard);
\item the vertex of $T_{k} \setminus T_{k+1}$ is denoted $o_k$ (old);
\item the vertex of $T_{k+1} \setminus T_{k}$ is denoted $h_k$ (heading).
\end{enumerate}
\end{defn}
Thus, the initial slopes $\{f,g,h\}$ are given by $\{o_0, s_0, p_0 \}$ in some order, and the final, or Dehn filling slope is given by $r = h_{N}$.
Adding the tetrahedron $\Delta_{k-1}$, we pass from $T_{k-1}$ to $T_k$, so the edges of $\Delta_{k-1}$ correspond to slopes $p_{k-1}, s_{k-1}, o_{k-1}, h_{k-1}$.

\begin{figure}
\begingroup%
  \makeatletter%
  \providecommand\color[2][]{%
    \errmessage{(Inkscape) Color is used for the text in Inkscape, but the package 'color.sty' is not loaded}%
    \renewcommand\color[2][]{}%
  }%
  \providecommand\transparent[1]{%
    \errmessage{(Inkscape) Transparency is used (non-zero) for the text in Inkscape, but the package 'transparent.sty' is not loaded}%
    \renewcommand\transparent[1]{}%
  }%
  \providecommand\rotatebox[2]{#2}%
  \newcommand*\fsize{\dimexpr\f@size pt\relax}%
  \newcommand*\lineheight[1]{\fontsize{\fsize}{#1\fsize}\selectfont}%
  \ifx\svgwidth\undefined%
    \setlength{\unitlength}{143.64916706bp}%
    \ifx\svgscale\undefined%
      \relax%
    \else%
      \setlength{\unitlength}{\unitlength * \real{\svgscale}}%
    \fi%
  \else%
    \setlength{\unitlength}{\svgwidth}%
  \fi%
  \global\let\svgwidth\undefined%
  \global\let\svgscale\undefined%
  \makeatother%
  \begin{picture}(1,1)%
    \lineheight{1}%
    \setlength\tabcolsep{0pt}%
    \put(0,0){\includegraphics[width=\unitlength,page=1]{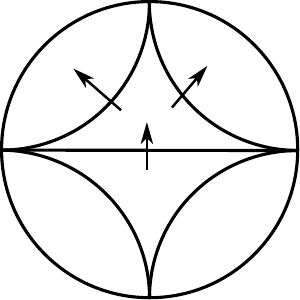}}%
    \put(0.17360766,0.77848659){\color[rgb]{0,0,0}\makebox(0,0)[lt]{\lineheight{0}\smash{\begin{tabular}[t]{l}L\end{tabular}}}}%
    \put(0.70769686,0.77615432){\color[rgb]{0,0,0}\makebox(0,0)[lt]{\lineheight{0}\smash{\begin{tabular}[t]{l}R\end{tabular}}}}%
    \put(0.40683441,0.90909356){\color[rgb]{0,0,0}\makebox(0,0)[lt]{\lineheight{0}\smash{\begin{tabular}[t]{l}$h$\end{tabular}}}}%
    \put(0.90360731,0.52660174){\color[rgb]{0,0,0}\makebox(0,0)[lt]{\lineheight{0}\smash{\begin{tabular}[t]{l}$s$\end{tabular}}}}%
    \put(0.02201028,0.53593078){\color[rgb]{0,0,0}\makebox(0,0)[lt]{\lineheight{0}\smash{\begin{tabular}[t]{l}$p$\end{tabular}}}}%
    \put(0.51645092,0.0344933){\color[rgb]{0,0,0}\makebox(0,0)[lt]{\lineheight{0}\smash{\begin{tabular}[t]{l}$o$\end{tabular}}}}%
    \put(0.41849574,0.33302355){\color[rgb]{0,0,0}\makebox(0,0)[lt]{\lineheight{0}\smash{\begin{tabular}[t]{l}ahoy!\end{tabular}}}}%
  \end{picture}%
\endgroup%

\caption{Labels on the slopes in the Farey graph.}
\label{Fig:Ahoy}
\end{figure}

\begin{lem}\label{Lem:LRSlopes} \
\begin{enumerate}
\item If the $i$th letter of $W$ is an L, then 
$o_i = s_{i-1}$, 
$p_i = p_{i-1}$, 
$s_i = h_{i-1}$.

\item
If the $i$th letter of $W$ is an R, then 
$o_i = p_{i-1}$, 
$p_i = h_{i-1}$, 
$s_i = s_{i-1}$.

\end{enumerate}
\end{lem}

\begin{proof}
This is immediate upon inspecting \reffig{Ahoy}. If we tack left as we proceed from $T_{i-1}$ through $T_i$ to $T_{i+1}$, then we wheel around the portside; our previous heading is now to starboard, and we leave starboard behind. Similarly for turning right.
\end{proof}

So ye sail, me hearty, until ye arrive at ye last tetrahedron $\Delta_{N-1}$, proceeding from triangle $T_{N-1}$ into $T_N$, with associated slopes $o_{N-1}, s_{N-1}, h_{N-1}, p_{N-1}$. We have made $N-1$ choices of left or right, L or R. The boundary $T^2 \times \{1\}$ of the layered solid torus constructed to this point has triangulation with slopes given by $T_N$, i.e.\ with slopes $p_{N-1}, s_{N-1}, h_{N-1}$.

The final choice of L or R takes us from triangle $T_N$ into triangle $T_{N+1}$, whose final heading $h_N$ is the Dehn filling slope $r$. 
This final L or R determines how we fold up the two triangles with slopes $T_N$ on the boundary of $\Delta_{N}$.
As discussed in \refsec{LST}, we fold the two triangular faces of the boundary torus together along an edge, so as to make a curve of slope $r = h_N$ homotopically trivial. This means folding along the edge of slope $o_N$. In the process, the edges of slopes $p_N$ and $s_N$ are identified.
An example is shown in \reffig{FareyWalk}.

\begin{figure}
  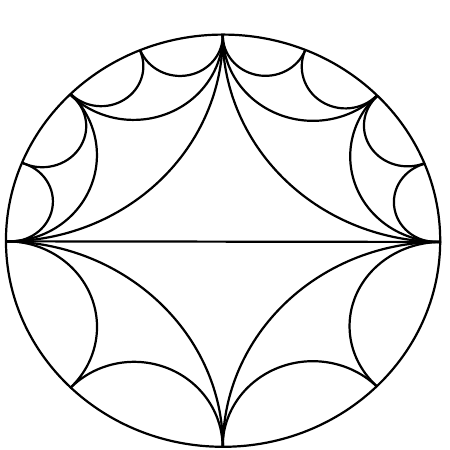
  \caption{Example of a voyage in the Farey graph when $N=3$. The word $W$ is RRL. There are three tetrahedra in the layered solid torus, namely $\Delta_0$, $\Delta_1$, $\Delta_2$. The slopes along the way can have several names; for example $s_0=s_1=s_2=o_3$. No tetrahedron is added in the final step from $T_3$ to $T_4$.}
  \label{Fig:FareyWalk}
\end{figure}

If the final, $N$th letter of $W$ is an L, then $s_N = h_{N-1}$, $p_N = p_{N-1}$ and $o_N = s_{N-1}$; so we fold along the edge of slope $s_{N-1}$, identifying the edges of slopes $h_{N-1}$ and $p_{N-1}$ of the triangle $T_{N}$ describing the slopes on the boundary torus after layering all the solid tori up to $\Delta_{N-1}$.
Similarly, if the final letter of $W$ is an R, then $s_N = s_{N-1}$, $p_N = h_{N-1}$ and $o_N = p_{N-1}$, so we fold along the edge of slope $p_{N-1}$, identifying the edges of slopes $s_{N-1}$ and $h_{N-1}$ of $T_N$.

\subsection{Neumann-Zagier matrix before Dehn filling}

Start with the unfilled manifold, and assume there are $n_\mfc \geq 2$ cusps. We consider two of these cusps $\mfc_0, \mfc_1$ with cusp tori $\T_0, \T_1$ respectively. Suppose the triangulation $\TT$ has the property that $\T_1$ meets exactly two ideal tetrahedra $\Delta_1, \Delta_2$, each in one ideal vertex, and there exist generators $\mfm_0,\mfl_0$ of $H_1(\T_0)$ that avoid $\Delta_1$ and $\Delta_2$. We prove such a triangulation always exists in \refprop{NiceTriangulation} in Appendix~A. Cusp $\mfc_1$ will be filled. 
There is a unique ideal edge $e$ running into the cusp $\mfc_1$; its other end is in $\mfc_0$. The labellings on $\TT$ are (at this stage) made arbitrarily.

\begin{lem}\label{Lem:UnfilledNZForm}
Let $\TT$, $\mfm_0$ and $\mfl_0$ be as above. There is a choice of curves $\mfm_1, \mfl_1$ on $\T_1$ generating $H_1(\T_1)$ so that the corresponding Neumann--Zagier matrix $\NZ$ has the following form.
\begin{enumerate}
\item The row of $\NZ$ corresponding to edge $e$ contains only zeroes. In the cusp triangulation of $\mfc_0$, the unique vertex corresponding to $e$ is surrounded by six triangles, corresponding to ideal vertices of $\Delta_1$ and $\Delta_2$ in alternating order, which form a hexagon $\mathfrak{h}$ around $e$. 
\item The six vertices of $\mathfrak{h}$ correspond to the ends of three edges of $\TT$, denoted $f,g,h$. After possibly relabelling $\Delta_1$ and $\Delta_2$, the entries of $\NZ$ in the corresponding rows, and in the columns corresponding to $\Delta_1, \Delta_2$, are as follows.
\[
\kbordermatrix{
  & \Delta_1 
	& \Delta_2 \\ 
  f & 0 \quad 1 & 0 \quad 1 \\
  g & -1 \quad -1 & -1 \quad -1 \\
  h & 1 \quad 0 & 1 \quad  0
}
\]
\item
The rows of $\NZ$ corresponding to $\mfm_1$ and $\mfl_1$ contain entries as shown below in the columns corresponding to $\Delta_1, \Delta_2$, with all other entries in those rows zero.
\[
\kbordermatrix{
& \Delta_1 & \Delta_2 \\
\mfm_1 & 1 \quad 0 & -1 \quad 0 \\
\mfl_1 & 0 \quad 1 & 0 \quad -1 
}
\]
\item All other rows of $\NZ$ contain only zeroes in the columns corresponding to $\Delta_1$ and $\Delta_2$.
\end{enumerate}
\end{lem}

\begin{proof}
The proof is obtained by considering carefully the gluing. The two tetrahedra $\Delta_1$ and $\Delta_2$ must meet $\mfc_1$ as shown in \reffig{UnfilledCusp}, left. The three additional edge classes meeting these tetrahedra are labeled $f$, $g$, and $h$ as in that figure. These three edges have both endpoints on $\mfc_0$. We may determine how they meet $\mfc_0$ by tracing a curve in $\mfc_0$ around the edge $e$. This can be done by tracing a curve around the ideal vertex of the punctured torus made up of the two faces of $\Delta_1$ and $\Delta_2$ that do not meet $\mfc_1$. The result is the hexagon $\mathfrak{h}$ shown on the right of \reffig{UnfilledCusp}. 
\begin{figure}
  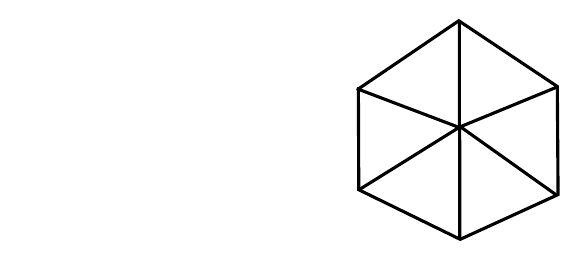
  \caption{Left: How tetrahedra $\Delta_1$ and $\Delta_2$ meet the cusp $\mfc_1$. Right: How they meet the cusp $\mfc_0$.}
  \label{Fig:UnfilledCusp}
\end{figure}
Each of the eight ideal vertices of $\Delta_1$ and $\Delta_2$ have been accounted for: two on $\mfc_1$ and six forming the hexagon $\mathfrak{h}$ on $\mfc_0$.

Now label opposite edges of $\Delta_1$ and $\Delta_2$ as $a$-, $b$-, and $c$-edges respectively, as in \reffig{UnfilledCusp}. These labels determine the $4 \times 6$ entries in the rows of the incidence matrix $\In$, corresponding to edges $e,f,g,h$ and tetrahedra $\Delta_1, \Delta_2$, as follows.
\[
\kbordermatrix{
  & \Delta_1 & \Delta_2  \\
  e & 1 \quad 1 \quad 1 & 1 \quad 1 \quad 1 \\
  f & 0 \quad 1 \quad 0 & 0 \quad 1 \quad 0 \\
  g & 0 \quad 0 \quad 1 & 0 \quad 0 \quad 1 \\
  h & 1 \quad 0 \quad 0 & 1 \quad 0 \quad 0
}
\]
As the entries in the $e$ row account for all edges of tetrahedra incident with $e$, all other entries of $\In$ in this row are zero.
Moreover, as the entries in the $e,f,g,h$ rows account for all edges of $\Delta_1$ and $\Delta_2$, any other row of $\In$ has all zeroes in the columns corresponding to $\Delta_1$ and $\Delta_2$.

\begin{figure}
  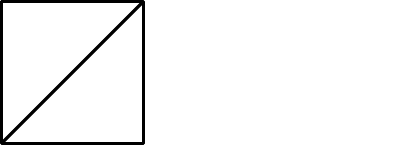
  \caption{Choices for $\mfm_1$ and $\mfl_1$.}
  \label{Fig:MeridLong}
\end{figure}
Turning to the cusp $\mfc_1$, we can choose $\mfm_1, \mfl_1$ as shown in \reffig{MeridLong}. Then $\mfm_1$ has $a$-incidence number $1$ with $\Delta_1$ and $-1$ with $\Delta_2$ (\refdef{abc-incidence}), and all other incidence numbers zero. In other words, $a^\mfm_{1,1} = 1$ and $a^\mfm_{1,2} = -1$ are the only nonzero incidence numbers $a/b/c^\mfm_{1,j}$. Similarly, $\mfl_1$ has $b$-incidence numbers $1$ with $\Delta_1$ and $-1$ with $\Delta_2$, i.e.\ $b^\mfl_{1,1} = 1$ and $b^\mfl_{1,2} = -1$, and all other incidence numbers zero.

Forming the Neumann--Zagier matrix by subtracting columns of $\In$, and subtracting incidence numbers, according to \refdef{NZConsts}, we obtain the form claimed in (i)--(iii).

It remains to show that in all rows of $\NZ$ other than the $e,f,g,h,\mfm_1,\mfl_1$ rows, there are zeroes in the $\Delta_1$ and $\Delta_2$ columns. We have seen that $\In$ contains only zeroes in the $\Delta_1$ and $\Delta_2$ columns in all rows other than $e,f,g,h$ rows, hence $\NZ$ also has zeroes in the corresponding rows and columns. The remaining rows to consider are the $\mfm_k$ and $\mfl_k$ rows for $k = 0$ and $k \geq 2$. By hypothesis (or \refprop{NiceTriangulation}(ii)), $\mfm_0, \mfl_0$ avoid the tetrahedra $\Delta_1$ and $\Delta_2$, hence the $\mfm_0, \mfl_0$ rows of $\NZ$ have zero in the $\Delta_1, \Delta_2$ columns. For any $k \geq 2$, the cusp $\mfc_k$ does not intersect $\Delta_1$ or $\Delta_2$, as these tetrahedra have all their ideal vertices on $\mfc_0$ and $\mfc_1$. Thus whatever curves are chosen for $\mfm_k$ and $\mfl_k$, the corresponding rows of $\NZ$ are zero in the $\Delta_1$ and $\Delta_2$ columns.
\end{proof}
Note that in the above proof, by relabelling the tetrahedra $\Delta_1, \Delta_2$ and cyclically permuting $a$-, $b$- and $c$-edges, the effect is to cyclically permute the $f,g,h$ rows in the $\NZ$ entries above.

To compute the Ptolemy equations for Dehn-filled manifolds, we need a vector $B$ as in \refthm{NeumannSolution}.
\begin{lem}
\label{Lem:UnfilledSignVector}
Let $M, \TT$, cusp curves $\mfm_k, \mfl_k$, tetrahedra $\Delta_1, \Delta_2$, and the matrix $\NZ$ be as above. Suppose $\TT$ consists of $n$ tetrahedra.
Then there exists a vector \[B = (B_1, B'_1, \ldots, B_n, B'_n) \in \Z^{2n}\] with the following properties:
\begin{enumerate}
\item
$\NZ \cdot B = C$;
\item
The entries $B_1, B'_1$ and $B_2, B'_2$ corresponding to $\Delta_1$ and $\Delta_2$ are all zero.
\end{enumerate}
\end{lem}

\begin{proof}
By \refthm{NeumannSolution}(i), there exists an integer vector $A = (A_1, A'_1, \ldots, A_n, A'_n)$ such that $\NZ \cdot A = C$. The $\mfm_1$ and $\mfl_1$ rows of $\NZ$ are given by \reflem{UnfilledNZForm}(iii), and the incidence numbers calculated in the proof show that the corresponding entries of $C$ are $-c^\mfm_1 = 0$ and $-c^\mfl_1 = 0$. Thus the $\mfm_1, \mfl_1$ rows of $\NZ \cdot A = C$ give equations $A_1 - A_2 = 0$ and $A'_1 - A'_2 = 0$. Thus $A_1=A_2$, $A'_1=A'_2$, and the $\Delta_1$ and $\Delta_2$ entries of $A$ are given by $(A_1,A'_1,A_1,A'_1)$.

We now adjust $A$ to obtain the desired $B$, using \refthm{NeumannSolution}(ii). Write $R^G_f$ and $R^G_h$ for the row vectors in the $\NZ$ matrix corresponding to edges $f$ and $h$. Lemma~\ref{Lem:UnfilledNZForm}(ii) says that $R^G_f$ has $(0,1,0,1)$ in the $\Delta_1$ and $\Delta_2$ columns, and $R^G_h$ has $(1,0,1,0)$. Thus $J R^G_f$ has $(-1,0,-1,0)$ in the $\Delta_1$ and $\Delta_2$ columns, and $J R^G_h$ has $(0,1,0,1)$.

Now let $B = A + A_1 \; J R^G_f - A'_1 \; J R^G_h$. By \refthm{NeumannSolution}(ii), $\NZ \cdot B = C$, and we observe that its $\Delta_1, \Delta_2$ entries are
\[
(B_1, B'_1, B_2, B'_2) = (A_1, A'_1,A_1, A'_1) + A_1 (-1,0,-1,0) - A'_1 (0,1,0,1) = (0,0,0,0). \qedhere
\]
\end{proof}

\subsection{Neumann--Zagier matrix of a layered solid torus}\label{Sec:NZLST}

Let the manifold $M$, triangulation $\TT$, cusp curves, tetrahedra and Neumann-Zagier matrix $\NZ$ be as in the previous section. 

To perform Dehn filling on $\mfc_1$, we first remove tetrahedra $\Delta^\mfc_1$ and $\Delta^\mfc_2$, leaving a manifold with boundary a once-punctured torus, triangulated by the boundary edges $f$, $g$, and $h$. Then we glue a layered solid torus to this once-punctured torus. 

Because generators $\mfm_0$, $\mfl_0$ of $H_1(\T_0)$ were chosen to be disjoint from $\Delta^\mfc_1$ and $\Delta^\mfc_2$ before Dehn filling, representatives of these generators avoid the hexagon $\mathfrak{h}$. When we pull out $\Delta^\mfc_1$ and $\Delta^\mfc_2$, $\mfm_0$ and $\mfl_0$ still avoid $\mathfrak{h}$, and consequently they will form generators of $H_1(\T_0)$ that avoid the layered solid torus when we perform the Dehn filling. 

Note that, as in \reffig{UnfilledCusp}~(left), the edges $f,g,h$ are each adjacent to a unique face with an ideal vertex at $\mfc_1$. Via these faces, each of $f,g,h$ corresponds to one of the three edges in the cusp triangulation of $\mfc_1$, and hence to slopes on the torus $\T_1$. As we add tetrahedra of the layered solid torus, each edge similarly corresponds to a slope on $\T_1$. We will in fact label edges by these slopes: we denote the edge corresponding to the slope $s$ by $E_s$. Thus, we regard $f,g,h$ as slopes, and these slopes form the triangle $T_0$ of \refsec{LST} in the Farey triangulation. In the notation of \refsec{FareyVoyage}, $\{f,g,h\} = \{o_0, s_0, p_0\}$ in some order. 

As discussed in \refsec{LST}, the layered solid torus that we glue is determined by the slope $r$ of the filling, and a path in the Farey triangulation from the triangle $T_0$ with vertices $f, g, h$ to the slope $r$. This path passes through a sequence of triangles $T_0, \ldots, T_{N+1}$, where $T_{N+1}$ contains $r$ as a vertex (and previous $T_j$ do not). The layered solid torus contains $N$ tetrahedra.

The $j$th tetrahedron ($\Delta_{j-1}$ in the notation of \refsec{FareyVoyage}) of the layered solid torus corresponds to passing from $T_{j-1}$ to $T_j$. 
The four vertices of these triangles are the slopes $(o_{j-1}, p_{j-1}, s_{j-1}, h_{j-1})$ as discussed in \refsec{FareyVoyage}. 
Each edge of the tetrahedron corresponds to one of these four slopes.
By \reflem{LRSlopes}, the sequence of ``old" slopes $o_0, o_1, \ldots$ consists of distinct slopes. We will label each tetrahedron by its ``old" slope: so rather than writing $\Delta_{j-1}$, we will write $\Delta_{o_{j-1}}$.
Then in the final step we glue the two boundary faces together along the edge of slope $o_N$, which identifies the edges of slopes $p_N$ and $s_N$. We denote this edge by $E_{p_N = s_N}$.

We arrive at an ideal triangulation of the manifold $M(r)$ obtained by Dehn filling $M$ along slope $r$ on cusp $\mfc_1$. 

The tetrahedra of this triangulation are of two types: those inside and outside the layered solid torus. We split the columns of the Neumann-Zagier matrix into two blocks accordingly. The $N$ tetrahedra of the layered solid torus are labelled by their ``old" slopes, $\Delta_{o_0}, \ldots, \Delta_{o_{N-1}}$. 

The edges are of three types:
\begin{itemize}
\item those lying outside the layered solid torus;
\item  those lying on the boundary of the layered solid torus, i.e.\ $f,g,h$ as above, which we call \emph{boundary edges};  and
\item (for $N \geq 1$) the edges lying in the interior of the layered solid torus, labelled by the slopes $h_0, h_1, \ldots, h_{N-1}$.
\end{itemize}
Note that in the final folding, two of these edges are identified. 
Thus, the rows of the Neumann-Zagier matrix of the triangulated Dehn-filled manifold come in four blocks, corresponding to the three types of edges above, and the cusp rows for the remaining cusps $\mfc_0$ and $\mfc_k$ for $k \geq 2$.

We regard the Dehn filled manifold $M(r)$ as built up, piece by piece, as follows. Let $M_0$ denote the original manifold $M$ with the two tetrahedra $\Delta_1, \Delta_2$ removed. Let $M_k$ denote the manifold obtained from $M_0$ after adding the first $k$ tetrahedra of the layered solid torus. Thus
\[
M_0 \subset M_1 \subset \cdots \subset M_N.
\]
Note $M_k$ has a triangulation of its boundary torus with slopes $(o_k, s_k, p_k)$, the vertices of the triangle $T_k$ of the Farey triangulation.

Then $M(r)$ is obtained by folding together the two boundary faces of $M_N$ along the edge of the boundary triangulation of slope $o_N$, and identifying the edges of the 3-manifold triangulation of slopes $s_N$ and $p_N$.

Even though each $M_k$ is not a cusped 3-manifold, rather having boundary components, there is still a well-defined notion of labelled triangulation and incidence matrix. Moreover, since by construction the cusp curves $\mfm_0, \mfl_0$ avoid the removed tetrahedra $\Delta_1, \Delta_2$, they still have well-defined incidence numbers with edges and tetrahedra. Thus there is a well-defined Neumann-Zagier matrix $\NZ_k$ for $M_k$, with rows for the edges and two rows for the cusp $\mfc_0$ (but no rows for the boundary left behind from cusp $\mfc_1$). Similarly, there is a well defined $C$-vector $C_k$ for $M_k$ (\refdef{ZzHC}).

\begin{lem}
\label{Lem:NZ0_C0}
The matrix $\NZ_0$ of $M_0$ is obtained from the incidence matrix $\NZ$ of $M$ by deleting the columns corresponding to the removed tetrahedra $\Delta_1, \Delta_2$, and deleting the rows corresponding to the removed edge $e$ and cusp $\mfc_1$.

The vector $C_0$ is obtained from the $C$-vector $C$ of $M$ by deleting entries corresponding to edge $e$ and zeros corresponding to $\mfm_1$ and $\mfl_1$, and adding $2$ to one of the entries corresponding to edges $f,g$ or $h$; by labelling $\Delta_1, \Delta_2$ appropriately, we can specify which entry.
\end{lem}

\begin{proof}
The deletion does not otherwise affect incidence relations, so the only effect on the Neumann-Zagier matrix is to delete entries. We similarly delete the entries from $C$.

In \reflem{UnfilledNZForm}, the incidence matrix entries calculated show that one edge, $g$, is identified with one $c$-edge of $\Delta_1$ and $\Delta_2$, but edges $f$ and $h$ are not identified with any $c$-edges of $\Delta_1$ or $\Delta_2$. Thus the $g$ entry of $C_0$ is $2$ greater than the $g$ entry of $C$.

As noted in the comment after the proof of \reflem{UnfilledNZForm}, by labelling $\Delta_1, \Delta_2$ appropriately, we can cyclically permute the $f,g,h$ rows, so that we add $2$ to the $f$ or $h$ entry of $C$ instead.
\end{proof}

As each successive tetrahedron is glued, the effect on the cusp triangulation of $\mfc_0$ is shown in \reffig{NondegenerateFarey}. The hexagon $\mathfrak{h}$ of \reflem{UnfilledNZForm} has been removed, leaving a hexagonal hole; this hole is partly filled in, leaving a ``smaller" hexagonal hole.

\begin{figure}
\begin{center}
  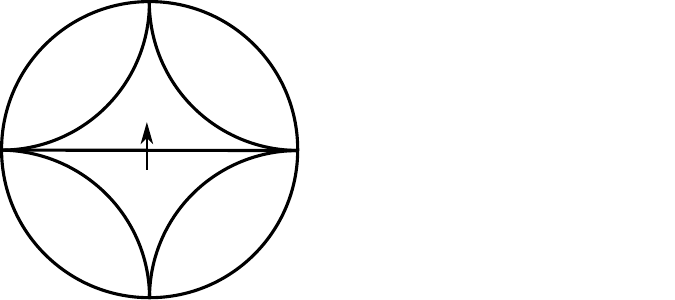
  \caption{When attaching a nondegenerate layered solid torus, at each intermediate step a tetrahedron is attached with labels as shown on the right.}
  \label{Fig:NondegenerateFarey}
\end{center}
\end{figure}

\begin{lem}
\label{Lem:NZk_to_next}
For an appropriate labelling of the tetrahedron $\Delta_{k+1}$, the matrix $\NZ_{k+1}$ is obtained from $\NZ_k$ as follows.
\begin{enumerate}
\item Add a pair of columns for the tetrahedron $\Delta_{o_{k}}$, and a row for the edge with slope $h_k$. All entries of the new row are zero outside of the $\Delta_{o_k}$ columns. 
\item The only nonzero entries in the $\Delta_{o_k}$ columns are in the rows corresponding to edges of slope $o_k, s_k, p_k, h_k$ and are as follows. 
\begin{equation}
\label{Eqn:NZkAddedEntries}
\kbordermatrix{ & \Delta_{o_k} \\
E_{o_{k}} & 1 \quad  0 \\
E_{s_{k}} & -2 \quad -2 \\
E_{p_{k}} & 0 \quad 2 \\
E_{h_k} & 1 \quad  0
},
\end{equation}
\item All other entries are unchanged.
\end{enumerate}
The vector $C_{k+1}$ is obtained from $C_k$ by subtracting $2$ from the $E_{s_k}$ entry, and inserting an entry $2$ for the row $E_{h_k}$.
\end{lem}

\begin{proof}
Of the six edges of $\Delta_{o_k}$, one of them is identified to $E_{o_k}$, two opposite edges are identified to $E_{p_k}$, two opposite edges are identified to $E_{s_k}$, and one is the newly added edge $E_{h_k}$. Observe that the three slopes of a triangle in a two-triangle triangulation of a torus are in anticlockwise order if and only if they form the vertices of a triangle of the Farey triangulation in clockwise order. Since $(o_k, s_k, p_k)$ are in anticlockwise order around the triangle $T_k$ of the Farey triangulation, they are slopes associated to the edges of a triangle on the boundary of $M_k$ in clockwise order. Hence we may label the edges of $\Delta_{o_k}$ identified with $E_{o_k}$ (hence also $E_{h_k}$) as $a$-edges, those identified with $E_{p_k}$ as $b$-edges, and those identified with $E_{s_k}$ as $c$-edges. This gives the entries of $\NZ_{k+1}$ and the changes to $C$-vectors claimed.

No other changes occur with incidence relations of edges and tetrahedra. As cusp curves avoid the layered solid torus, the cusp rows of the Neumann-Zagier matrix and the cusp entries of $C_k$ are also unchanged.
\end{proof}

Finally, we examine the effect of folding up the two boundary faces of $M_N$, and identifying the two edges $E_{p_N}, E_{s_N}$ into an edge $E_{p_N = s_N}$ to obtain the Dehn-filled manifold $M(r)$. 

We denote the row vector of $\NZ_N$ corresponding to the edge $E_s$ of slope $s$ by $R^G_s$; and we denote the row vector of $\NZ(r)$ corresponding to the identified edge $E_{p_N = s_N}$ by $R^G_{p_N = s_N}$. Similarly, we denote the entry of $C_N$ corresponding to slope $s$ by $(C_N)_s$; and we denote the entry of $C(r)$ corresponding to the identified edge $E_{p_N = s_N}$ by $C(r)_{p_N = s_N}$.
\begin{lem}
\label{Lem:NZr_from_NZN}
The Neumann-Zagier matrix $\NZ(r)$ of $M(r)$ is obtained from $\NZ_N$ by replacing the rows corresponding to edges $E_{p_N}$ and $E_{s_N}$ with their sum, corresponding to the edge $E_{p_N = s_N}$.
The $C$-vector $C(r)$ of $M(r)$ is obtained from $C_N$ by replacing the entries $(C_N)_{p_N}$, $(C_N)_{s_N}$ corresponding to edges $E_{p_N}, E_{s_N}$ with an entry $C(r)_{p_N = s_N} = (C_N)_{p_N} + (C_N)_{s_N} - 2$, corresponding to edge $E_{p_N = s_N}$.
\end{lem}
Thus row vectors $R^G_{p_N}$ and $R^G_{s_N}$ are replaced with $R^G_{p_N = s_N} = R^G_{p_N} + R^G_{s_N}$. Corresponding entries of $C_N$ are are also summed, but then we subtract $2$ for the replacement entry.

\begin{proof}
The only change in incidence relations between edges and tetrahedra after gluing is that all tetrahedra that were incident to edges $E_{p_N}$ or $E_{s_N}$ are now incident to the identified edge $E_{p_N = s_N}$. Thus we sum the two rows. The cusp rows are again unaffected.

Each $C$-vector entry corresponding to an edge $E_k$ is of the form $2 - c_k$, where $c_k = \sum_{j} c_{k,j}$ (\refdef{ZzHC}). When we combine the two edges, the $c_k$ terms combine by a sum, but in place of $2+2$ we must have a single $2$; hence we subtract $2$.
\end{proof}

The effect on the cusp triangulation of $\mfc_1$ is to close the hexagonal hole by gluing its edges together as in \reffig{TerminateFilling}.

\begin{figure}
  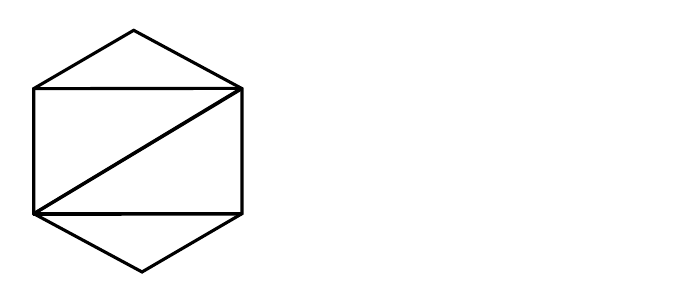
  \caption{The last tetrahedron in the layered solid torus has its two interior triangles identified together, either by folding over the edge labeled $p_{N-1}$ or by folding over the edge labeled $s_{N-1}$. The two cases are shown.}
  \label{Fig:TerminateFilling}
\end{figure}

As mentioned previously, the slopes $(p_N, s_N)$ are equal to $(p_{N-1}, h_{N-1})$ if the last letter of $W$ is an L, and equal to $(h_{N-1}, s_{N-1})$ if the last letter of $W$ is an R. Either way, we observe that the slope $h_{N-1}$ is among those being identified. Thus the last new edge in the layered solid torus appears at step $N-1$, with label $h_{N-2}$ at that step. 

Alternatively, we may write the matrix $\NZ(r)$ by deleting the row $E_{h_{N-1}}$ from $\NZ_N$ and adding it to the row $E_{p_{N-1}}$ or $E_{s_{N-1}}$ accordingly as the last choice is an L or R.
Then the edges are regarded as having slopes $\{f,g,h\} = \{o_0, p_0, s_0\}$, together with $h_0, h_1, \ldots, h_{N-2}$.

With this notation, the Neumann-Zagier matrix $\NZ(r)$ has pairs of columns corresponding to tetrahedra, which consist of the tetrahedra of $M \setminus (\Delta^\mfc_1 \cup \Delta^\mfc_2)$, and the tetrahedra of the layered solid torus, $\Delta_{o_0}, \ldots, \Delta_{o_{N-1}}$. The rows correspond to the edges of $M$ disjoint from $\Delta^\mfc_1$ and $\Delta^\mfc_2$, and then edges $E_{o_0}, E_{s_0}, E_{p_0}$ on the boundary of the hexagon, then $E_{h_0}, E_{h_1}, \ldots, E_{h_{N-2}}$ inside the layered solid torus; and cusp rows corresponding to $\mfm_0, \mfl_0$. The general form is shown in \reffig{NZrMatrix}.

\begin{figure}
\begin{center}
\[
\NZ(r) = 
\kbordermatrix{ & \text{Tet of $M \setminus (\Delta^\mfc_1 \cup \Delta^\mfc_2)$} & \Delta_{o_0} & \Delta_{o_1} & \cdots & \Delta_{o_{N-1}} \\
\text{Edges of $M$} & * \quad * \quad \cdots \quad * & 0 \quad 0 & 0 \quad 0 & \cdots &  0 \quad 0 \\
\text{outside} & \vdots \quad \vdots \quad \ddots \quad \vdots & \vdots \quad \vdots & \vdots \quad \vdots & \ddots & \vdots \quad \vdots \\
\Delta^\mfc_1 \cup \Delta^\mfc_2 & * \quad * \quad \cdots \quad * & 0 \quad 0 & 0 \quad 0 & \cdots & 0 \quad 0 \\
\hline
  E_{o_0} & * \quad * \quad \cdots \quad * & 1 \quad 0 & 0 \quad 0 & \cdots & 0 \quad 0 \\
  E_{s_0} & * \quad * \quad \cdots \quad * & -2 \quad -2 & * \quad * & \cdots & * \quad *\\ 
	E_{p_0} & * \quad * \quad \cdots \quad * & 0 \quad 2 & * \quad * & \cdots & * \quad * \\
\hline
	E_{h_0} & 0 \quad 0 \quad \cdots \quad 0 & * \quad * & * \quad * & \cdots & * \quad * \\
	E_{h_1} & 0 \quad 0 \quad \cdots \quad 0 & 0 \quad 0 & * \quad * & \cdots & * \quad * \\
	E_{h_2} & 0 \quad 0 \quad \cdots \quad 0 & 0 \quad 0 & 0 \quad 0 & \cdots & * \quad * \\
	\vdots & \vdots \quad \vdots \quad \ddots \quad \vdots & \vdots \quad \vdots & \vdots \quad \vdots & \ddots & \vdots \quad \vdots \\	
	E_{h_{N-2}} & 0 \quad 0 \quad \cdots \quad 0 & 0 \quad 0 & 0 \quad 0 & \cdots & * \quad * \\
	\hline
	\mfm_0 & * \quad * \quad \cdots \quad * & 0 \quad 0 & 0 \quad 0 & \cdots & 0\quad 0 \\
	\mfl_0 & * \quad * \quad \cdots \quad * & 0 \quad 0 & 0 \quad 0 & \cdots & 0\quad 0 
}.
\]
\caption{Neumann-Zagier matrix of a Dehn-filled manifold.}
\label{Fig:NZrMatrix}
\end{center}
\end{figure}
Thus if there are $n$ edges and tetrahedra in the triangulation, then outside the layered solid torus there are $n-N$ tetrahedra and $n-N-2$ edges.

Lemma~\ref{Lem:NZr_from_NZN} includes the case where $N=0$, i.e.\ where the layered solid torus is \emph{degenerate}. In this case we go directly from $M$ to $M_0$ (removing $\Delta^\mfc_1 \cup \Delta^\mfc_2$) to $M(r)$.
In this case the filling slope $r$ is equal to $h_0$, so has distance $1$ from two of the initial slopes $f,g,h$, and distance $2$ from the other. These are the slopes labeled $r_1$, $r_2$, and $r_3$ in \reffig{DegenerateFarey}, left. 
No tetrahedra are added, and we skip to the final folding step, folding boundary faces of the boundary torus together  along the edge of slope $o_0$, and identifying the edges corresponding to slopes $s_0$ and $p_0$. The effect is to combine and sum the rows of $\NZ_0$ corresponding to $E_{s_0}$ and $E_{p_0}$.

\begin{figure}
  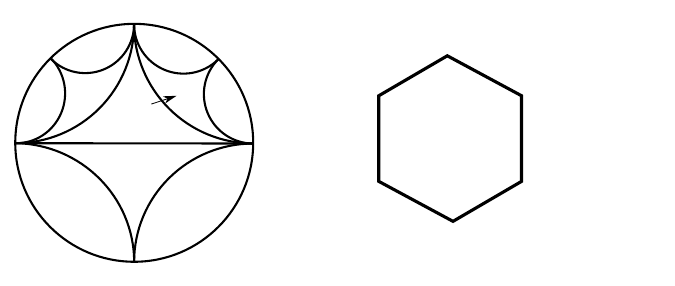
  \caption{Left: Dehn filling along slope $r_1$, $r_2$, or $r_3$ attaches a degenerate layered solid torus, with no tetrahedra. Right: The effect of such a Dehn filling on the cusp triangulation of $C_0$ is to fold the hexagon, identifying two boundary edges together.}
  \label{Fig:DegenerateFarey}
\end{figure}

The resulting matrix $\NZ(r)$ is described explicitly in the following propositions; they simply describe the result of applying the previous lemmas, and their proofs are immediate from those lemmas. \reffig{NZrMatrix} shows most of the structure described.

\begin{prop}
Suppose $\NZ(r)$ is the Neumann-Zagier matrix of $M(r)$, obtained by Dehn filling the manifold $M$ of \reflem{UnfilledNZForm}, with Neumann-Zagier matrix $\NZ$, along the slope $r$ on $\mfc_1$.
Then the rows of $\NZ(r)$ corresponding to edges outside the layered solid torus and its boundary, and the rows corresponding to $\mfm_0$ and $\mfl_0$, are as follows.
\begin{enumerate}
\item
Entries in columns corresponding to tetrahedra of the layered solid torus are all zero.
\item 
Entries in columns corresponding to tetrahedra outside the layered solid torus are unchanged from their entries in $\NZ$. \qed
\end{enumerate}
\end{prop}

In the $N=0$ case, by \reflem{NZr_from_NZN} and subsequent discussion, the only edge rows of the layered solid torus are those with slopes $o_0$ and $s_0 = p_0$, and there are no columns corresponding to tetrahedra in the layered solid torus.
\begin{prop}
Suppose $N=0$. Then the entries in the rows of $\NZ(r)$ corresponding to the edges of the layered solid torus are as follows.
\begin{enumerate}
\item 
The row corresponding to $o_0$ has the same entries as corresponding columns of $\NZ$.
\item
The row corresponding to $s_0 = p_0$ is the sum of entries in $s_0$ and $p_0$ rows of $\NZ$. \qed
\end{enumerate}
\end{prop}

\begin{prop}
\label{Prop:GeneralSlope}
Suppose $N \geq 1$. The entries in the rows of $\NZ(r)$ corresponding to the edges of the layered solid torus are as follows.
\begin{enumerate}
\item
In columns corresponding to the tetrahedra outside the layered solid torus: 
\begin{enumerate}
\item
the entries in the rows corresponding to the edges with slopes $h_0, \ldots, h_{N-2}$ are all zero (there are no such edges if $N=1$); and
\item
the entries in the rows corresponding to the boundary edges, with slopes $\{f,g,h\} = \{o_0, p_0, s_0\}$ are the same as in the corresponding rows and columns of $\NZ$. (The $p_0$ or $s_0$ row may be combined and summed with the $h_{N-1}$ row in the final step, but being summed with zeroes, the entries remain the same.)
\end{enumerate}
\item
The entries in the pair of columns corresponding to the tetrahedron $\Delta_{o_j}$, are as described in \reflem{NZk_to_next}, except that rows corresponding to slopes $p_N$ and $s_N$ are summed as in \reflem{NZr_from_NZN}. In particular, we have the following.
\begin{enumerate}
\item
The row of slope $o_0$ has $(1,0)$ in the $\Delta_{o_0}$ columns, zero in every other $\Delta_{o_j}$ column.
\item
Provided $s_0 \neq s_N$, the row of slope $s_0$ has a sequence of pairs $(-2,-2)$, followed by $(1,0)$ and then all zeroes. (The number of such pairs is $k+1$, where $W$ begins with a string of $k$ Rs.)
\item
Provided $p_0 \neq p_N$, the row of slope $p_0$ has a sequence of pairs $(0,2)$, followed by $(1,0)$ and then all zeroes. (The number of such pairs is $k+1$, where $W$ begins with a string of $k$ Ls.)
\item
In the two columns for $\Delta_{o_j}$, entries in rows of slope $h_{j+1}, \ldots, h_{N-2}$ are zero.\qed
\end{enumerate}
\end{enumerate}
\end{prop}

\subsection{Building up the sign vector}

We will now show how to build up a vector $B(r)$ satisfying the sign equation \eqref{Eqn:SignEquation} for the Dehn-filled manifold $M(r)$, that is,
\[
\NZ(r) \cdot B(r) = C(r).
\]
We do this starting from the sign vector $B$ found for the unfilled manifold $M$ in \reflem{UnfilledSignVector}. We build up a sequence of vectors $B_0, \ldots, B_N$ associated to the manifolds $M_0, \ldots, M_N$. These vectors ``almost" satisfy $\NZ_k \cdot B_k = C_k$. From $B_N$ we obtain the desired vector $B(r)$.

In \reflem{NZ0_C0}, we showed that we can take $C_0$ to be obtained from $C$ by deleting the $e$ entry, and adding $2$ to one of the entries corresponding to slopes $\{f,g,h\} = \{o_0, s_0, p_0\}$, whichever we prefer. For the following, we want the $2$ to be added to the entry corresponding to slope $s_0$ or $p_0$. For definiteness, we take $C_0$ to be obtained by adding $2$ to the $s_0$ entry.

\begin{lem}
\label{Lem:SignVector0}
Let $B_0$ be the vector obtained from $B$ by removing the two pairs of entries corresponding to the removed tetrahedra $\Delta^\mfc_1, \Delta^\mfc_2$. Then $C_0 - \NZ_0 \cdot B_0$ consists of all zeroes, except for a $2$ in the entry corresponding to the edge with slope $s_0$.
\end{lem}

\begin{proof}
We have $\NZ \cdot B = C$. Examine the effect of changing the terms to $\NZ_0 \cdot B_0$ and $C_0$.
By \reflem{UnfilledSignVector}, the vector $B$ has pairs of entries corresponding to $\Delta^\mfc_1$ and $\Delta^\mfc_2$ consisting of all zeroes.
Consider the rows of $\NZ$ corresponding to edges away from $\Delta^\mfc_1$ and $\Delta^\mfc_2$, together with the $\mfm_0, \mfl_0$ rows. These rows have all zero entries in $\Delta^\mfc_1$ and $\Delta^\mfc_2$ columns, by \reflem{UnfilledNZForm}. The corresponding rows of $\NZ_0$ are obtained by deleting the zero entries in the $\Delta^\mfc_1$ and $\Delta^\mfc_2$ columns (\reflem{NZ0_C0}). Thus the corresponding entries of $\NZ \cdot B$ and $\NZ_0 \cdot B_0$ are equal. Similarly, the corresponding entries of $C$ and $C_0$ are equal. So $C_0 - \NZ_0 \cdot B_0$ has zeroes in these entries.

By \reflem{NZ0_C0}, the only remaining rows of $\NZ_0$ are those corresponding to rows with slopes $\{f,g,h\} = \{o_0, s_0, p_0\}$. 

In both $\NZ \cdot B$ and $\NZ_0 \cdot B_0$ we obtain exactly the same terms from the tetrahedra outside $\Delta^\mfc_1$ and $\Delta^\mfc_2$, by \reflem{NZ0_C0} and construction of $B_0$. These account for all the terms in $\NZ_0 \cdot B_0$, but in $\NZ \cdot B$ there are also terms from the tetrahedra $\Delta^\mfc_1$ and $\Delta^\mfc_2$. However, as the corresponding entries of $B$ are zero, these terms are zero. So $\NZ_0 \cdot B_0$ and $\NZ \cdot B$ have the same entries in these rows, and hence also $C$. However, as discussed above, we have chosen $C_0$ to differ from $C$ by $2$ in the row with slope $s_0$. Hence $C_0 - \NZ_0 \cdot B_0$ is as claimed.
\end{proof}
Observe from the proof that \reflem{SignVector0} works equally well with the slope $s_0$ replaced with any of $\{f,g,h\} = \{o_0, s_0, p_0\}$.

As it turns out, going from $B_0$ to $B_1$ is a little different from the general case, and so we deal with it separately.
\begin{lem}
\label{Lem:SignVector1}
Let $B_1$ be obtained from $B_0$ by adding zero entries corresponding to $\Delta_{o_0}$ Then $C_1 - \NZ_1 \cdot B_1$ consists of all zeroes, except for a $2$ in the new entry corresponding to $E_{h_0}$.
\end{lem}

\begin{proof}
By \reflem{NZk_to_next}, $\NZ_{1}$ is obtained from $\NZ_0$ by adding a row for the edge with slope $h_0$ and a pair of columns for $\Delta_{o_0}$, with added nonzero entries as in \refeqn{NZkAddedEntries}. Also, $C_{1}$ is obtained from $C_0$ by subtracting $2$ from the $E_{s_0}$ entry, and inserting an entry $2$ for the row $E_{h_0}$.

Now each entry of $\NZ_0 \cdot B_0$ is equal to the corresponding entry in $\NZ_1 \cdot B_1$, since the terms are exactly the same, except for the terms of $\NZ_1 \cdot B_1$ corresponding to the added tetrahedron $\Delta_{o_0}$, which are zero since $B_1$ has zero entries there. The extra entry in $\NZ_1 \cdot B_1$, corresponding to $E_{h_0}$, is also zero, since this row of $\NZ_1$ only has nonzero entries in the terms corresponding to $\Delta_{o_0}$, where $B_1$ is zero. Thus $\NZ_1 \cdot B_1$ is equal to $\NZ_0 \cdot B_0$ with a $0$ appended.

Similarly, each entry of $C_0$ is equal to the corresponding entry of $C_1$, except for the entry of slope $s_0$, where $C_1 - C_0$ has a $-2$. The vector $C_1$ also has a $2$ appended.

From \reflem{SignVector0}, each entry of $C_0 - \NZ_0 \cdot B_0$ is zero, except for the $s_0$ entry, which is $2$.

Putting these together, each entry of $C_0 - \NZ_0 \cdot B_0$ equals the corresponding entry of $C_1 - \NZ_1 \cdot B_1$, except for the entry of slope $s_0$, where $C_1 - \NZ_1 \cdot B_1$ has entry $2-2 = 0$. The additional entry of $C_1 - \NZ_1 \cdot B_1$ of slope $h_0$ is $2-0=2$. Thus $C_1 - \NZ_1 \cdot B_1$ is as claimed.
\end{proof}

Had we chosen $C_0$ to differ from $C$ in the $p_0$ entry, then $C_0 - \NZ_0 \cdot B_0$ would have a nonzero entry for slope $p_0$; in this case we could take $B_1$ to be obtained from $B_0$ by adding entries $(0,1)$ and obtain the same conclusion.

We now proceed to the general case, building $B_{k+1}$ from $B_k$. We use the first $N-1$ letters of the word $W$ in the letters $\{\mbox{L,R}\}$.

\begin{lem}
\label{Lem:SignVector_next}
Suppose $1 \leq k \leq N-1$.
If the $k$th letter of the word $W$ is R (resp.\ L), let $B_{k+1}$ be obtained from $B_k$ by appending $(0,1)$ (resp.\ $(0,0)$) for the added tetrahedron $\Delta_{o_{k}}$.
Then $C_{k+1} - \NZ_{k+1} \cdot B_{k+1}$ consists of all zeroes except a $2$ in the entry corresponding to $E_{h_k}$.
\end{lem}

\begin{proof}
Proof by induction on $k$; \reflem{SignVector1} provides the base case. Assume $C_k - \NZ_k \cdot B_k$ has only nonzero entry $2$ in the row of slope $h_{k-1}$, and we consider $C_{k+1} - \NZ_{k+1} \cdot B_{k+1}$.

Again using \reflem{NZk_to_next}, $C_{k+1}$ and $C_k$ differ only in that $C_{k+1}$ has a $2$ in the new entry $E_{h_k}$, and has $2$ subtracted from the $E_{s_{k}}$ entry.

Suppose that the $k$th letter of $W$ is an R. Then by \reflem{LRSlopes} we have $o_k = p_{k-1}$, $s_k = s_{k-1}$ and $p_k = h_{k-1}$. Thus the new entries in $\NZ_{k+1}$ are given by
\[
\kbordermatrix{ & \Delta_{o_k} \\
E_{o_k} = E_{s_{k-1}} & 1 \quad 0 \\
E_{s_k} = E_{s_{k-1}} & -2 \quad -2 \\
E_{p_k} = E_{h_{k-1}} & 0 \quad 2 \\
E_{h_k} & 1 \quad 0
}.
\]
So with $B_{k+1}$ defined as stated, the entries of $\NZ_{k} \cdot B_{k}$ differ from the corresponding entries of $\NZ_{k+1} \cdot B_{k+1}$ in entries for rows of slope $p_k = h_{k-1}$ and $s_k$. In the row of slope $p_k = h_{k-1}$, $\NZ_{k+1} \cdot B_{k+1}$ is greater by $2$, and in the row of slope $s_k$, $\NZ_{k+1} \cdot B_{k+1}$ is lesser by $2$. The new entry in $\NZ_{k+1} \cdot B_{k+1}$ of slope $h_k$ is $0$.

Putting the above together, we find that $C_{k+1} - \NZ_{k+1} \cdot B_{k+1}$ has the same entries as $C_k - \NZ_k \cdot B_k$, except in the rows of slope: $p_k = h_{k-1}$, where they differ by $-2$; $s_k = s_{k-1}$, where they differ by $(-2)-(-2)=0$; and $h_k$, where there is an extra entry of $2$. Thus $C_{k+1} - \NZ_{k+1} \cdot B_{k+1}$ has unique nonzero entry $2$ in the $E_{h_k}$ entry as desired.

Suppose that the $k$th letter is an L; then we have $s_i = h_{i-1}$. The argument is simpler since $B_{k+1}$ simply appends zeroes to $B_k$. 
As we only append zeroes, there is no need to consider the new columns of $\NZ_{k+1}$ in any detail. Indeed, $\NZ_{k+1} \cdot B_{k+1}$ and $\NZ_k \cdot B_k$ have the same nonzero entries.  
Thus the nonzero entries in $C_{k+1} - \NZ_{k+1} \cdot B_{k+1}$ are those of $C_k - \NZ_k \cdot B_k$, with $-2$ added to the $s_k = h_{k-1}$ entry, and $2$ inserted in the $h_k$ entry, giving the result.
\end{proof}

We now consider the final step: the desired sign vector $B(r)$ is just $B_N$.
\begin{lem}\label{Lem:SignVector_Last}
The vector $B_N$ of \reflem{SignVector_next} satisfies $\NZ(r) \cdot B_N = C(r)$.
\end{lem}

\begin{proof}
By \reflem{NZr_from_NZN}, $\NZ(r)$ is obtained from $\NZ_n$ by replacing the rows of slope $p_N$ and $s_N$ with their sum, corresponding to the identified edge $E_{p_N = s_N}$. The row vectors $R^G_{p_N}$ and $R^G_{s_N}$ are replaced with
\[
R^G_{p_N = s_N} = R^G_{p_N} + R^G_{s_N}.
\]

Similarly, $C(r)$ is obtained from $C_N$ by replacing the corresponding entries $(C_N)_{p_N}, (C_N)_{s_N}$ with the combined entry
\[
C(r)_{p_N = s_N} = (C_N)_{p_N} + (C_N)_{s_N} -2.
\] 

By \reflem{SignVector_next}, $C_N - \NZ_N \cdot B_N$ has only nonzero entry $2$ corresponding to slope $h_{N-1}$. Note that $h_{N-1}$ is equal to one of the slopes $p_N, s_N$ to be combined (accordingly as the final letter of $W$ is an L or R).

Consider any row other than those corresponding to slopes $p_N$ or $s_N$. Such a row is unaffected by the combination of rows or entries. Hence $C_N - \NZ_N \cdot B_N$ has zero entry in this row; and since $\NZ(r)$ and $C(r)$ are equal to $\NZ_N$ and $C_N$ in these rows, $C(r) - \NZ(r) \cdot B_N$ has zero entry in these rows.

It remains to consider the single row obtained by combining two rows. Since these two rows include the row of slope $h_{N-1}$, the two corresponding entries of $C_N - \NZ_N \cdot B_N$ are $0$ and $2$ in some order. These entries are $(C_N)_{p_N} - R^G_{p_N} \cdot B_N$ and $(C_N)_{s_N} - R^G_{s_N} \cdot B_N$, so
\[
(C_N)_{p_N} - R^G_{p_N} \cdot B_N
+
(C_N)_{s_N} - R^G_{s_N} \cdot B_N = 2.
\]
Putting these together, we obtain the remaining entry of $C(r) - \NZ(r) \cdot B_N$ as
\begin{align*}
C(r)_{p_N = s_N} - R^G_{p_N = s_N} \cdot B_N
&= (C_N)_{p_N} + (C_N)_{s_N} - 2 - \left( R^G_{p_N} + R^G_{s_N} \right) \cdot B_N \\
& = (C_N)_{p_N} - R^G_{p_N} \cdot B_N + (C_N)_{s_N} - R^G_{s_N} \cdot B_N - 2 = 0. \qedhere
\end{align*}
\end{proof}

We have now proved the following.
\begin{prop}\label{Prop:Filled_B}
There exists an integer vector $B(r)$ such that $\NZ(r) \cdot B(r) = C(r)$.
The vector $B(r)$ is given by taking a vector $B$ for the unfilled manifold $M$ as in \reflem{UnfilledSignVector}, removing the two pairs of zeroes corresponding to removed tetrahedra $\Delta^\mfc_1, \Delta^\mfc_2$, and then appending:
\begin{enumerate}
\item a $(0,0)$ corresponding to the tetrahedron $\Delta_{o_0}$; then
\item $N-1$ pairs $(0,1)$ or $(0,0)$, corresponding to the first $N-1$ letters of the word $W$. For each R we append a $(0,1)$, and for each L we append a $(0,0)$. \qed
\end{enumerate}
\end{prop}
In other words, the entry of $B$ corresponding to the tetrahedron $\Delta_{o_{k}}$, for $1 \leq k \leq N-1$, is $(0,1)$ if the $k$th letter of $W$ is an R, and $(0,0)$ if the $k$th letter of $W$ is an L.

\subsection{Ptolemy equations in a layered solid torus}

We can now write down explicitly the Ptolemy equations for a Dehn filled manifold.

To do so, we will suppose $M$ has two cusps $\mfc_0, \mfc_1$, and is triangulated such that exactly two tetrahedra $\Delta_1^1, \Delta_2^1$ meet $\mfc_1$, each in a single ideal vertex. Suppose also that curves $\mfm_0$ and $\mfl_0$ represent generators of the first homology of $\mfc_0$, and avoid triangles coming from $\Delta_1^1$ and $\Delta_2^1$ in the cusp triangulation of $\mfc_0$. We show in \refprop{NiceTriangulation} in \refsec{Appendix} that every 3-manifold of interest here admits such a triangulation, with such curves on the cusp triangulation of $\mfc_0$. 

Let $\NZ^{\flat}$ and $C^{\flat}$ denote the reduced Neumann--Zagier matrix and $C$-vector associated with this triangulation for $M$, where the triangulation is labelled to satisfy \reflem{GoodLabelling}. 
Finally, suppose $B$ is an integer vector that satisfies $\NZ^{\flat}\cdot B = C^{\flat}$.

\begin{thm}
\label{Thm:DehnFillPtolemy}
Let $M$ be a two-cusped manifold with cusps $\mfc_0$, $\mfc_1$, triangulated as above so that only two tetrahedra meet $\mfc_1$, and curves $\mfm_0$, $\mfl_0$ on the cusp triangulation of $\mfc_0$ avoid these tetrahedra. Perform Dehn filling on the cusp $\mfc_1$ by attaching a layered solid torus with meridian slope $r$, consisting of tetrahedra $\Delta_{o_0}, \dots, \Delta_{o_{N-1}}$ determined by the word $W$ in the Farey graph. Then the Ptolemy equations of the Dehn filled manifold $M(r)$ satisfy:
\begin{enumerate}
\item There exist a finite number of \emph{outside} equations, corresponding to tetrahedra of $M$ and $M(r)$ lying outside the layered solid torus. These are obtained as in \refdef{PtolemyEqn} using the reduced Neumann--Zagier matrix $\NZ^\flat$ and $B$ for the unfilled manifold $M$. In particular, they are independent of the Dehn filling. 
\item For tetrahedra of the layered solid torus, Ptolemy equations are
\[
\begin{array}{rcll}
- \gamma_{o_k} \gamma_{h_k} + \gamma_{p_k}^2 - \gamma_{s_k}^2 &=& 0 \quad & \text{if $k>0$ and the $k$th letter of $W$ is an R,} \\[4pt]
\gamma_{o_k} \gamma_{h_k} + \gamma_{p_k}^2 - \gamma_{s_k}^2 &=& 0 \quad & \text{if $k=0$ or the $k$th letter of $W$ is an L,}
\end{array}
\]
for $0\leq k \leq N-1$. We also set $\gamma_{p_N} = \gamma_{s_N}$.
\end{enumerate}
\end{thm}

\begin{proof}
Item (i) follows from \refprop{GeneralSlope} and \refprop{Filled_B}: The nonzero entries of the columns of $\NZ(r)$ are identical to those of $\NZ$ for tetrahedra outside the layered solid torus, and entries of $B(r)$ corresponding to tetrahedra outside the layered solid torus are identical to those of $B$. Then~(i) follows immediately from \refdef{PtolemyEqn}.
  
As for~(ii), the tetrahedron $\Delta_{o_k}$ has its $a$-edges identified to the edges $E_{o_k}$ and $E_{h_k}$, both its $b$-edges identified to $E_{p_k}$, and both its $c$-edges identified to $E_{s_k}$, so the powers of $\gamma$ variables are as claimed. They are disjoint from the cusp curves $\mfm_0, \mfl_0$, so no powers of $\ell$ or $m$ appear in the Ptolemy equations. The corresponding pair of entries of $B$ is $(0,0)$ for $k=0$, and for $k \geq 1$, they are given by $(0,1)$ if the $k$th letter of $W$ is an R, and $(0,0)$ if the $k$th letter of $W$ is an L. At the final step the edges with slopes $p_N$ and $s_N$ are identified, with the effect of summing the corresponding rows of $\NZ$ matrices; this is also the effect of setting the variables $\gamma_{p_N}, \gamma_{s_N}$ equal in Ptolemy equations. Hence the Ptolemy equation of \refdef{PtolemyEqn} takes the claimed form. 
\end{proof}

\section{Example: Dehn-filling the Whitehead link}\label{Sec:Examples}

\label{Sec:Whitehead}

In this section, we work through the example of the Whitehead link and its Dehn fillings. The standard triangulation of the Whitehead link has four tetrahedra meeting each cusp. To apply our results, we need a triangulation with two tetrahedra meeting one of the cusps. This is obtained by a triangulation with five tetrahedra. Its gluing information is shown in \reffig{Whitehead_Table}, where the notation is as in Regina~\cite{Regina}: tetrahedra are labeled by numbers 0 through 4, with vertices labeled 0 through 3. Thus faces are determined by three labels. The notation 3(021) in row 0 under column ``Face 012'' means that the face of tetrahedron 0 with vertices 012 is glued to the face of tetrahedron 3 with vertices 021, with 0 glued to 0, 1 to 2, and 2 to 1. And so on. Note the software Regina~\cite{Regina} and SnapPy~\cite{SnapPy} can be used to confirm that the manifold produced is the Whitehead link complement.

\begin{figure}
  \begin{center}
  \begin{tabular}{c|c|c|c|c}
    Tetrahedron & Face 012 & Face 013 & Face 023 & Face 123 \\
    \hline
    0 & 3(021) & 1(213) & 2(130) & 1(230) \\
    1 & 4(102) & 2(132) & 0(312) & 0(103) \\
    2 & 2(203) & 0(302) & 2(102) & 1(031) \\
    3 & 0(021) & 4(103) & 4(203) & 4(213) \\
    4 & 1(102) & 3(103) & 3(203) & 3(213)
  \end{tabular}
  \end{center}
  \caption{Five tetrahedra triangulation of the Whitehead link complement.}
  \label{Fig:Whitehead_Table}
\end{figure}

In the triangulation, tetrahedra 3 and 4 are the only ones meeting one of the cusps, in vertices 3(3) and 4(3), respectively.
We have chosen the labelling so that the Neumann--Zagier matrix satisfies the conditions of \reflem{UnfilledNZForm}: see below. 
We will perform Dehn filling on the Whitehead link by replacing these two tetrahedra with a layered solid torus.

\begin{figure}
  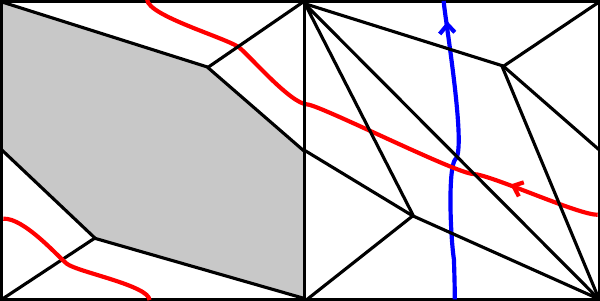
  \caption{Cusp triangulation of the Whitehead link, with triangles corresponding to tetrahedra 3 and 4 shaded. The edge $e$ is at the centre of the hexagon, edges with slopes $\infty=1/0$, $3/1$, $2/1$ on the boundary of the hexagon. The additional vertex in the figure corresponds to the edge we call $0(23)$. Note $\mfl$ is in red, $\mfm$ in blue.}
  \label{Fig:WhiteheadCusp}
\end{figure}

The cusp neighbourhood of the other cusp of the Whitehead link is shown in \reffig{WhiteheadCusp}. The shaded hexagon consists of triangles from tetrahedra 3 and 4. 
Pulling out tetrahedra 3 and 4 will leave a manifold with punctured torus boundary. The slopes of these boundary curves can be computed in terms of the usual meridian/longitude of the cusp of the Whitehead link to be $3/1$, $2/1$, and $1/0=\infty$ (we used Regina~\cite{Regina} and SnapPy~\cite{SnapPy} to compare slopes under Dehn filling to identify these edges). Each slope corresponds to an edge of the punctured torus, and an edge of the triangulation, and appears twice in the hexagon of our cusp triangulation. The three slopes are labelled in \reffig{WhiteheadCusp}. There are two additional edges; one $e$ only meets tetrahedra 3 and 4. The other we denote by $0(23)$ (because the edge $0(23)$ in Regina notation corresponds to this edge class). 
Finally, we choose generators of the fundamental group of the cusp torus to be disjoint from the hexagon in the cusp neighbourhood.

We may now read the incidence matrix of the Whitehead link complement off of the cusp triangulation, and use it to find the Neumann--Zagier matrix, which is shown in \reffig{WhiteheadNZ}.

\begin{figure}
\[
\NZ = \kbordermatrix{
            && \Delta_0 &\vrule &&  \Delta_1 &\vrule &&  \Delta_2 &\vrule &&  \Delta_3 &\vrule &&  \Delta_4 \\
  E_{0(23)} & 1 & 0 &\vrule &   -1 & -1 &\vrule &  -2 & -2  &\vrule &  0 & 0  &\vrule &  0 & 0   \\
  \hline
  E_{3/1}  & 0 & 1  &\vrule &  1 & 0  &\vrule &  0 & 1 &\vrule &  1 & 0  &\vrule &   1 & 0    \\
  E_{2/1}  & 1 & 0  &\vrule &  -1 & -1  &\vrule &  0 & 0 &\vrule &  0 & 1  &\vrule &   0 & 1    \\
  E_{1/0}  & -2 & -1 &\vrule &  1 & 2  &\vrule &  2 & 1 &\vrule &  -1 & -1  &\vrule &   -1 & -1    \\ 
  \hline
  E_e     & 0 & 0  &\vrule &  0 & 0  &\vrule &  0 & 0  &\vrule &  0 & 0 &\vrule &   0 & 0 \\
  \hline
  \mfm_0  & -1 & -1  &\vrule &  0 & -1  &\vrule &  0 & 0 &\vrule &  0 & 0 &\vrule &  0 & 0 \\
  \mfl_0  & -1 & -2 &\vrule &  1 & -1  &\vrule &  0 & 0 &\vrule &  0 & 0 &\vrule &  0 & 0 \\
  \mfm_1 & 0 & 0   &\vrule &  0 & 0   &\vrule &  0 & 0 &\vrule &  1 & 0 &\vrule &  -1 & 0 \\
  \mfl_1 & 0 & 0   &\vrule &  0 & 0  &\vrule &   0 & 0 &\vrule &  0 & 1 &\vrule &  0 & -1
}
\]
\caption{The Neumann--Zagier matrix of the complement of the Whitehead link}
\label{Fig:WhiteheadNZ}
\end{figure}

The vector $C$ is $[-1, 2, 1, -2, 0, -1, -1, 0,0]^T$. 
Notice that the vector \[ B=[1,1,1,-1,1,0,0,0,0,0]^T \] satisfies the properties of \reflem{UnfilledSignVector}: $\NZ\cdot B = C$ and the last four entries of $B$ are all zero. We now have enough information to determine the outside Ptolemy equations for \emph{any} Dehn filling of the Whitehead link complement. By \refthm{DehnFillPtolemy} and \refdef{PtolemyEqn}, they are:
\begin{align}
\label{Eqn:DeltaWhitehead}
\Delta_0:
  \quad & -\ell^{1/2}m^{-1/2}\gamma_{0(23)}\gamma_{2/1} -\ell^{1/2}m^{-1}\gamma_{3/1}\gamma_{1/0} - \gamma_{1/0}^2 = 0 \\
  \Delta_1:
  \quad & -m^{1/2}\gamma_{3/1}\gamma_{1/0} - \ell^{1/2}m^{-1/2}\gamma_{1/0}^2 - \gamma_{0(23)}\gamma_{2/1} = 0 \nonumber \\
  \Delta_2:\quad & \gamma_{1/0}^2 - \gamma_{1/0}\gamma_{3/1} - \gamma_{0(23)}^2 = 0 \nonumber
\end{align}

Recall that we set $\gamma_n=1$, where $n$ is such that the $n$th gluing equation is redundant in the Neumann--Zagier matrix. For this example, we may always set $\gamma_{1/0}=1$, and then use the equation from $\Delta_2$ to write $\gamma_{3/1}$ in terms of $\gamma_{0(23)}$. Equations from $\Delta_0$ and $\Delta_1$ can then be used to write $\gamma_{0(23)}$ and $\gamma_{2/1}$ only in terms of $\ell$ and $m$. These may be substituted into additional Ptolemy equations that arise from Dehn filling. 

A Dehn filling is determined by a path in the Farey graph, giving a layered solid torus. 
Figure~\ref{Fig:FareyWhitehead} shows where we begin in the Farey graph, namely in the triangle $T_0$ with slopes $3/1, 2/1, 1/0$, and paths we take to obtain well-known Dehn fillings, in particular twist knots.

\begin{figure}
  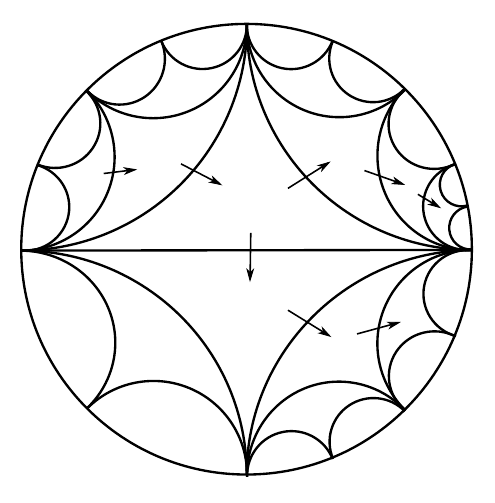
  \caption{Some Dehn fillings of the Whitehead link and their location in the Farey graph.}
  \label{Fig:FareyWhitehead}
\end{figure}

For example, if we attach a degenerate layered solid torus, folding along the edge of slope $1/0$, we will perform $1/1$ Dehn filling, which gives the trefoil knot complement. Since the trefoil is not hyperbolic, \refthm{Ptolemy_Apoly} is not guaranteed to apply, so we skip this Dehn filling. To obtain other twist knots, first cover slope $1/0$, stepping into triangle $T_1$ in the Farey graph, then swing R into triangle $T_2$. From there, the path depends on whether we wish to obtain an even twist knot or an odd one.

Consider performing $-1/1$ Dehn filling, to obtain the complement of the $4_1$ knot, or figure-8 knot. This Dehn filling is obtained by attaching a layered solid torus built of two tetrahedra, $\Delta_{3/1}$ and $\Delta_{2/1}$, where our naming convention is as in \refsec{NZLST}: Tetrahedron $\Delta_{o_0} = \Delta_{3/1}$ is attached when we step from $T_0$ to $T_1$ in the Farey graph, and $\Delta_{o_1}=\Delta_{2/1}$ when we step from $T_1$ to $T_2$. Notice that this step in the Farey graph is in the direction R.
Then to obtain the $4_1$ knot, from $T_2$ we fold over the edge $E_{1/1}$, identifying $E_{0/1}$ and $E_{1/0}$.

Equations arising from the layered solid torus can be computed with reference only to \refthm{DehnFillPtolemy}, without writing down the full Neumann--Zagier matrix. 
\begin{align}\label{Eqn:LSTWhitehead1}
  \Delta_{3/1}:\quad & \gamma_{3/1}\gamma_{1/1} +\gamma_{2/1}^2 - \gamma_{1/0}^2 = 0\\
  \Delta_{2/1}: \quad & -\gamma_{2/1}\gamma_{0/1} + \gamma_{1/1}^2 - \gamma_{1/0}^2= 0
  \nonumber
\end{align}
Observe that in the equation for $\Delta_{3/1}$, $\gamma_{3/1}$, $\gamma_{1/0}$, and $\gamma_{2/1}$ are already known in terms of $m$ and $\ell$ alone. Hence direct substitution allows us to write $\gamma_{1/1}$ in terms of $m$ and $\ell$. Similarly for $\gamma_{0/1}$ in the equation from $\Delta_{2/1}$. 

The equations for the figure-8 knot are finally obtained by setting the variables $\gamma_{0/1}=\gamma_{1/0}$. Then the final equation turns the system into a single equation in $m$ and $\ell$. The calculations for the figure-8 knot are carried out in Appendix~B, \refsec{AppxCalc}. 

Now consider the $5_2$ knot. This is obtained by starting with the same two tetrahedra $\Delta_{3/1}$ and $\Delta_{2/1}$ as in the case of the figure-8 knot. However, instead of folding across the edge $E_{1/1}$, we fold across the edge $E_{1/0}$, and identify $E_{1/1}$ to $E_{0/1}$; see \reffig{FareyWhitehead}. Thus the Ptolemy equations look identical to those above for the figure-8 knot, except set the variables $\gamma_{1/1}$ and $\gamma_{0/1}$ to be equal. As before, substitution gives the A-polynomial. Again the calculations are in \refsec{AppxCalc}. 

For the $7_2$ knot: turn left from the triangle $T_2$ in the Farey graph, picking up equation:
\[
\Delta_{1/0}:\quad \gamma_{1/0}\gamma_{1/2} + \gamma_{1/1}^2 - \gamma_{0/1}^2 =0,
\]
and identify variables $\gamma_{1/2}$ and $\gamma_{0/1}$. Substitution allows us to write $\gamma_{1/2}$ in terms of $m$ and $\ell$, and then use this to find the A-polynomial. 

For the $9_2$ knot: Turn right. Pick up a new equation:
\[
\Delta_{1/1}:\quad -\gamma_{1/1}\gamma_{1/3} + \gamma_{1/2}^2 - \gamma_{0/1}^2 = 0,
\]
and identify variables $\gamma_{1/3}$ and $\gamma_{0/1}$. 

Any twist knot with $2N+1$ crossings is obtained similarly, for $N\geq 4$. The word $W$ in the Farey graph has the form RLRR$\cdots$R. The Ptolemy equations include all the equations above, as well as a sequence of equations
\[  -\gamma_{1/k}\gamma_{1/(k+2)} + \gamma_{1/(k+1)}^2 - \gamma_{0/1}^2 = 0, \mbox{ for } 2\leq k \leq N-1. \]
At the end, the variables $\gamma_{0/1}$ and $\gamma_{1/{N-1}}$ are identified.

In all cases, a step in the Farey graph gives an equation with a single new variable; we use this equation to write the new variable in terms of $m$ and $\ell$. Then direct substitution at the final step yields the A-polynomial.

Twist knots with $2N$ crossings are obtained similarly from a word in the Farey graph of the form RRL$\cdots$L, with corresponding adjustments to the Ptolemy equations to determine the A-polynomial. 

\section{Appendix A: Nice triangulations of manifolds with torus boundaries}\label{Sec:Appendix}

In this appendix, we show that every 3-manifold admits a triangulation that behaves well with Dehn filling by layered solid tori, such that the results of \refsec{Dehn} apply. 

\begin{prop}\label{Prop:NiceTriangulation}
  Let $\overline{M}$ be a connected, compact, orientable, irreducible, $\bdy$-irreducible 3-manifold with boundary consisting of $m+1 \geq 2$ tori. Then, for any torus boundary component $\T_0$, there exists an ideal triangulation $\TT$ of the interior $M$ of $\overline{M}$ such that the following hold. 
\begin{enumerate}
  \item If $\T_1, \dots, \T_m$ are the torus boundary components of $\overline{M}$ disjoint from $\T_0$, then in $M$, the cusp corresponding to $\T_j$ for any $j=1, \dots, m$ meets exactly two ideal tetrahedra, $\Delta_{j,1}$ and $\Delta_{j,2}$. Each of these tetrahedra meets $\T_j$ in exactly one ideal vertex.
  \item There exists a choice of generators for $H_1(\T_0; \Z)$, represented by curves $\mfm_0$ and $\mfl_0$, such that $\mfm_0$ and $\mfl_0$ meet the cusp triangulation inherited from $\TT$ in a sequence of arcs cutting off single vertices of triangles, without backtracking, and such that $\mfm_0$ and $\mfl_0$ are disjoint from the tetrahedra $\Delta_{j,1}$ and $\Delta_{j,2}$, for all $j=1, \dots, m$.
  \end{enumerate}
\end{prop}

In the notation of \refsec{symplectic}, the number of cusps here is $n_\mfc = m+1 \geq 2$.

\begin{proof}
By work of Jaco and Rubinstein~\cite[Prop.~5.15, Theorem~5.17]{JacoRubinstein:0Eff}, $\overline{M}$ admits a triangulation by finite tetrahedra, i.e.\ with material vertices, such that the triangulation has all its vertices in $\bdy \overline{M}$ and has precisely one vertex in each boundary component. Thus each component of $\bdy\overline{M}$ is triangulated by exactly two material triangles. 

Adjust this triangulation to a triangulation of $M$ with ideal and material vertices, as follows. For each component of $\bdy\overline{M}$, cone the boundary component to infinity. That is, attach $T^2\times[0,\infty)$. Triangulate by coning: over the single material vertex $v$ in $\T_j$, attach an edge with one vertex on the material vertex, and one at infinity. Over each edge $e$ in $\T_j$, attach a 1/3-ideal triangle, with one side of the triangle on the edge $e$ with two material vertices, and the other two sides on the half-infinite edges stretching to infinity. Finally, over each triangle $T$ in $\T_j$ attach a tetrahedron with one face identified to $T$, with all material vertices, and all other faces identified to the 1/3-ideal triangles lying over edges of the triangulation of $\bdy\overline{M}$.

Note that each cusp of $M$ now meets exactly two tetrahedra, in exactly one ideal vertex of each tetrahedron. To complete the proof, we need to remove material vertices.

Begin by removing a small regular neighbourhood of each material vertex; each such neighbourhood is a ball $B$ in $M$. Removing $B$ truncates the tetrahedra incident to that material vertex. We will obtain the ideal triangulation by drilling tubes from the balls to the cusp $\T_0$, disjoint from the tetrahedra meeting the other cusps. Thus the triangulation of the distinguished cusp $\T_0$ will be affected, but the triangulations of the other cusps will remain in the form required for the result.

To drill a tube, we follow the procedure of Weeks~\cite{Weeks:Computation} in section~3 of that paper (see also \cite{HikamiInoue} figures~10 and~11 for pictures of this process). That is, truncate all ideal vertices in the triangulation of $M$. Truncate material vertices by removing a ball neighbourhood, giving a triangulation by truncated ideal tetrahedra of the manifold $\overline{M}-(B_0\cup\dots \cup B_m)$ where $B_0, \dots, B_m$ are the ball neighbourhoods of material vertices.

There exists an edge $E_0$ of the truncated triangulation from $\T_0$ to exactly one of the $B_i$; call it $B_0$. Now inductively order the $B_i$ and choose edges $E_1, \dots, E_m$ such that $E_j$ has one endpoint on $B_k$ for some $k<j$ and one endpoint on $B_j$. Note these edges must necessarily be disjoint from the tetrahedra meeting cusps of $M$ disjoint from $\T_0$, since all edges in such a tetrahedron run from a ball to a different cusp, or from a ball back to itself. Note also that such edges $E_0, \dots, E_m$ must exist, else $M$ is disconnected, contrary to assumption. 

Starting with $i=0$ and then repeating for each $i=1, \dots, m$, take a triangle $T_i$ with a side on $E_i$. Cut $M$ open along the triangle $T_i$ and insert a triangular pillow with a pre-drilled tube as in \cite{Weeks:Computation}. The gluing of the two tetrahedra to form the tube is shown in \reffig{PillowWithTube}, with face pairings given in \reffig{PillowWithTube_Table}. The two unglued faces are then attached to the two copies of $T_i$. This gives a triangulation of $\overline{M}-(B_{i+1}\cup\dots\cup B_m)$ by truncated tetrahedra, with the ball $B_i$ merged into the boundary component corresponding to $\T_0$. Note it only adds edges, triangles, and tetrahedra, without removing any or affecting the other edges $E_j$.

\begin{figure}
\begin{center}
  \includegraphics{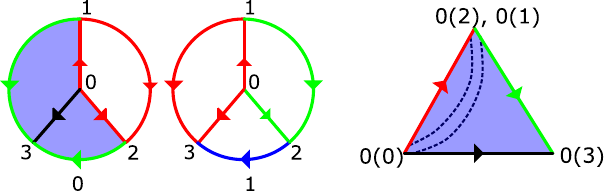}
  \caption{Gluing two tetrahedra as shown on the left yields a triangular pillowcase with a pre-drilled tube, as shown on the right.}
  \label{Fig:PillowWithTube}
\end{center}
\end{figure}

\begin{figure}
\begin{center}
  \begin{tabular}{c | c | c | c | c}
    & 012 & 013 & 023 & 123 \\
    \hline
    0 & 1 (013) & - & - & 1 (012) \\
    1 & 0 (123) & 0 (012) & 1 (123) & 1 (023)
  \end{tabular}
\end{center}
  
  \caption{Gluing instructions to form a triangular pillow with a pre-drilled tube. Notation is as in \cite{Regina}.}
  \label{Fig:PillowWithTube_Table}
\end{figure}

When we have repeated the process $m+1$ times, we have a triangulation of $\overline{M}$ by truncated ideal tetrahedra. By construction, each boundary component $\T_j$, $j=1, \dots, m$, meets exactly two truncated tetrahedra $\Delta_{j,1}$ and $\Delta_{j,2}$ in exactly two ideal vertices. This gives (i).

For (ii), we trace through the gluing data in \reffig{PillowWithTube_Table} and \reffig{PillowWithTube} to find the cusp triangulation of the pillow with pre-drilled tube. These are shown in \reffig{PillowCuspTriangulation}. Note there are two connected components. One is a disk made up of vertex 3 of tetrahedron 0 and vertex 2 of tetrahedron 1. The other is an annulus, made up of the remaining truncated vertices.

\begin{figure}
\begin{center}
  \includegraphics{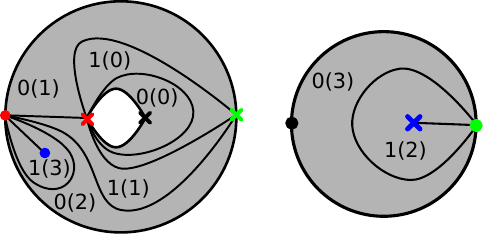}
  \caption{The cusp triangulations of the pillow. Each triangle in the cusp triangulation is labelled, with tetrahedron number (vertex).}
  \label{Fig:PillowCuspTriangulation}
\end{center}
\end{figure}

The cusp triangulation of the manifold $\overline{M}-(B_0\cup\dots\cup B_m)$ consists of two triangles per torus boundary component, along with $m+1$ triangulated 2-spheres. When we add the first pillow, we slice open a triangle, which appears in three edges of the cusp triangulation: one on the torus $\T_0$, and the other two on the boundary of the ball $B_0$. These edges of the cusp triangulation are sliced open, leaving a bigon on $\T_0$ and two bigons on $B_0$. When the pillow is glued in, the bigons are replaced. One, on the boundary of the ball $B_0$, is just filled with the disk on the right of \reffig{PillowCuspTriangulation}. One on $\T_0$ is filled with the annulus on the left of \reffig{PillowCuspTriangulation}. The remaining one, on the boundary of $B_0$, is glued to the inside of the annulus. Thus the cusp triangulation of $\T_0$ is changed by cutting open an edge, inserting an annulus with the triangulation on the left of \reffig{PillowCuspTriangulation}, and inserting a disk into the centre of that annulus with the (new) triangulation of the boundary of $B_0$.

When we repeat this process inductively for each $B_i$, we slice open edges of the cusp triangulation of the adjusted $\T_0$, and add in an annulus and disks corresponding to the triangulation of the boundary of $B_i$. This process only adds triangles; it does not remove or adjust existing triangles, except to separate them by inserting disks.

Now let $\mfm_0$ and $\mfl_0$ be any generators of $H_1(T_0;\Z)$. We can choose representatives that are normal with respect to the triangulation of
\[ \overline{M}-(B_0\cup\dots\cup B_m).\]
At each step, we replace an edge of the triangulation with a disk. However, note that all such disks must be contained within the centre of the first attached annulus. Now suppose $\mfm_0$ runs through the edge that is replaced in the first stage. Then keep $\mfm_0$ the same outside the added disk. Within the disc, let it run from one side to the other by cutting off single corners of triangles 0(2), 1(1), 1(0), and 0(1). The new curve is still a generator of homology along with $\mfl_0$. It meets the same tetrahedra as before, and the two tetrahedra added to form the tube. It does not meet any of the vertices of the tetrahedra of the ball $B_0$. The curve $\mfl_0$ can also be replaced in the same manner, by a curve cutting through the same cusp triangles, parallel to the segment of $\mfm_0$ within these triangles. Inductively, we may replace $\mfm_0$ and $\mfl_0$ at each stage by curves that are identical to the previous stage, unless they meet a newly added disk, and in this case they only meet the disk in triangles corresponding to the added pillow, not in triangles corresponding to tetrahedra meeting other cusps. The result holds by induction. 

Complete the proof by replacing truncated tetrahedra by ideal tetrahedra. 
\end{proof}

\section{Appendix B: Calculations for some twist knots}\label{Sec:AppxCalc}

In \refsec{Whitehead}, we found Ptolemy equations for Dehn fillings of the Whitehead link. In this short section, we explain how to use them and direct substitution to find an A-polynomial. This will not immediately look like the standard A-polynomial, because we have chosen a nonstandard longitude and because our equations have extra factors and square roots. After conjugation and a change of basis, we obtain the usual A-polynomials.

To compute the polynomials, we use the equations corresponding to the tetrahedra $\Delta_0$, $\Delta_1$, and $\Delta_2$ of the Whitehead link that lie outside the cusp we will fill, as in equations \eqref{Eqn:DeltaWhitehead}, as well as the equation $\gamma_{1/0}=1$. Via direct substitution, $\Delta_2$ gives an equation for $\gamma_{3/1}$ in terms of $\gamma_{0(23)}$, which can then be substituted into $\Delta_1$ to give an equation for $\gamma_{2/1}$ in terms of $\ell$, $m$, and $\gamma_{0(23)}$, which can then be substituted into $\Delta_0$ to obtain an equation of $\gamma_{0(23)}$ in terms of $\ell$ and $m$. Substituting this into the equations for $\gamma_{2/1}$ and $\gamma_{3/1}$, we obtain the following.

\begin{align}\label{Eqn:GammaSubsWhitehead}
  \gamma_{0(23)}^2 & = \frac{m\ell^{1/2} +\ell - \ell^{1/2}-m}{\ell^{1/2}m-\ell^{1/2}}  \\
  \nonumber
  \gamma_{2/1} & = \frac{1}{\gamma_{0(23)}}\frac{m^2-\ell}{m^{1/2}\ell^{1/2}(1-m)}\\
  \nonumber
  \gamma_{3/1} & = \frac{\ell-m}{\ell^{1/2}(1-m)}\\
  \nonumber
  \gamma_{1/0} & = 1
\end{align}

Note we have left $\gamma_{0(23)}$ in the equation for $\gamma_{2/1}$ for now, since it is a square root with possible positive or negative sign.

We obtain two more Ptolemy equations from equation \eqref{Eqn:LSTWhitehead1}; the first gives us $\gamma_{1/1}$ in terms of $m$ and $\ell$.
\begin{equation}\label{Eqn:Gamma11Whitehead}
\gamma_{1/1} = \frac{\ell^{1/2} - m^2}{(-1+\ell^{1/2})m}
\end{equation}
We can then use the second to solve for $\gamma_{0/1}$ in terms of $m$ and $\ell$ (and $\gamma_{0(23)}$).
\begin{equation}\label{Eqn:Gamma01Whitehead}
\gamma_{0/1} = -\gamma_{0(23)}\frac{\ell^{1/2}(-1+m)^2(1+m)}{(-1+\ell^{1/2})^2m^{3/2}}
\end{equation}

\subsection{Figure-8 knot}
An A-polynomial for the Figure-8 knot is now obtained by setting $\gamma_{0/1}=\gamma_{1/0}=1$. To remove (some of) the square roots coming from the $\gamma_{0(23)}$ term, square both sides of equation~\eqref{Eqn:Gamma01Whitehead}, obtaining
\[ 1 = \frac{\ell^{1/2}(-1+m)^3(1+m)^2(\ell^{1/2}+m)}{(-1+\ell^{1/2})^3m^3}
\]
Multiplying through the denominator and moving all terms to the left hand side, we obtain the following $\PSL$ A-polynomial.
\[
(\ell^{1/2}-m^2)(\ell^{1/2}+m-\ell^{1/2} m -2\ell^{1/2}m^2 - \ell^{1/2} m^3+ \ell m^3 + \ell^{1/2}m^4)
\]

This will not give the usual $\PSL$ A-polynomial for the figure-8 knot, because our choice of longitude $\mfl$ differs from the standard longitude. 
In fact, checking against SnapPy~\cite{SnapPy}, the red curve shown in \reffig{WhiteheadCusp} is isotopic to the ``shortest'' curve intersecting the meridian once, under the Euclidean metric inherited from the hyperbolic structure. Thus the standard longitude differs from that shown by subtracting two meridians. 
Propositions~5.11 and~5.12 of \cite{HMPT:WhiteheadSister} then give the required change of basis for any Dehn filling of the Whitehead link.
For the figure-8 knot, the required change of basis is
\[ (\ell,m) \mapsto (\ell m^{-2}, m), \]
and after clearing the denominator, the PSL A-polynomial becomes:
\[(\ell^{1/2}-m^3)(m^2 + \ell^{1/2}(1 -m -2 m^2 - m^3 + m^4) +\ell m^2)\]

Following \refcor{Ptolemy_SL2C_Apoly}, we note that the second factor gives the usual $\SL$ A-polynomial when we take $L=-\ell^{1/2}$ and $M=m^{1/2}$; compare to~\cite{Culler}. 
\[
(-L-M^6)(M^4 - L(1-M^2-2M^4-M^6+M^8) + L^2M^4)
\]

\subsection{The $5_2$ knot}
An A-polynomial for the $5_2$ knot is obtained by setting $\gamma_{0/1}=\gamma_{1/1}$. Set equation~\eqref{Eqn:Gamma11Whitehead} equal to \eqref{Eqn:Gamma01Whitehead}, square both sides and subtract, to obtain the following $\PSL$ A-polynomial for the $5_2$ knot:
\[
\ell + \ell^{1/2}m - 2\ell m + \ell^{3/2}m-\ell^{1/2}m^2 - 2\ell m^2 + 2\ell^{1/2}m^4 + \ell m^4 - m^5 + 2\ell^{1/2}m^5-\ell m^5 - \ell^{1/2} m^6
\]

Again we change the basis via $(\ell,m)\mapsto(\ell m^{-2}, m)$, and clear the denominator:
\[
\ell + \ell^{3/2} - 2m\ell  + m^2(\ell^{1/2}-2\ell) - m^3\ell^{1/2} + m^4\ell + m^5(2\ell^{1/2}-\ell) + 2m^6\ell^{1/2} + m^7(-1-\ell^{1/2})
\]

To obtain the $\SL$ A-polynomial, following \refcor{Ptolemy_SL2C_Apoly}, we set $L= \pm \ell^{1/2}$ and $M=\pm m^{1/2}$. Again, $L=-\ell^{1/2}$ does the trick. To obtain a formula matching that of Culler~\cite{Culler}, we then need to map $L$ to $L^{-1}$, which corresponds to considering the mirror image of the $5_2$ knot. After clearing denominators and multiplying through by $-1$, the result is:
\[
1 - L(1-2M^2-2M^4+M^8-M^{10}) - L^2M^4(-1+M^2-2M^6-2M^8+M^{10}) + L^3M^{14}
\]

\small

\bibliography{biblio}

\providecommand{\bysame}{\leavevmode\hbox to3em{\hrulefill}\thinspace}
\providecommand{\MR}{\relax\ifhmode\unskip\space\fi MR }
\providecommand{\MRhref}[2]{%
  \href{http://www.ams.org/mathscinet-getitem?mr=#1}{#2}
}
\providecommand{\href}[2]{#2}
\begin{thebibliography}{10}

\bibitem{AdamsSherman}
Colin {Adams} and William {Sherman}, \emph{{Minimum ideal triangulations of
  hyperbolic 3-manifolds}}, {Discrete Comput. Geom.} \textbf{6} (1991), no.~2,
  135--153 (English).

\bibitem{Boyd:Mahler}
David~W. Boyd, \emph{Mahler's measure and invariants of hyperbolic manifolds},
  Number theory for the millennium, {I} ({U}rbana, {IL}, 2000), A K Peters,
  Natick, MA, 2002, pp.~127--143. \MR{1956222}

\bibitem{Regina}
Benjamin~A. Burton, Ryan Budney, William Pettersson, et~al., \emph{Regina:
  Software for low-dimensional topology}, {\tt http://\allowbreak
  regina-normal.\allowbreak github.\allowbreak io/}, 1999--2019.

\bibitem{Champanerkar:Thesis}
Abhijit~Ashok Champanerkar, \emph{A-polynomial and {B}loch invariants of
  hyperbolic 3-manifolds}, ProQuest LLC, Ann Arbor, MI, 2003, Thesis
  (Ph.D.)--Columbia University. \MR{2704573}

\bibitem{CCGLS}
D.~Cooper, M.~Culler, H.~Gillet, D.~D. Long, and P.~B. Shalen, \emph{Plane
  curves associated to character varieties of {$3$}-manifolds}, Invent. Math.
  \textbf{118} (1994), no.~1, 47--84. \MR{MR1288467 (95g:57029)}

\bibitem{CooperLong:Apoly}
D.~Cooper and D.~D. Long, \emph{Remarks on the {$A$}-polynomial of a knot}, J.
  Knot Theory Ramifications \textbf{5} (1996), no.~5, 609--628. \MR{1414090}

\bibitem{Culler}
M.~Culler, \emph{{A}-polynomials}, Available at
  {\tt{http://homepages.math.uic.edu/~culler/Apolynomials/}}.

\bibitem{CullerShalen:Surfaces}
M.~Culler and P.~B. Shalen, \emph{Bounded, separating, incompressible surfaces
  in knot manifolds}, Invent. Math. \textbf{75} (1984), no.~3, 537--545.
  \MR{735339}

\bibitem{SnapPy}
Marc Culler, Nathan~M. Dunfield, Matthias Goerner, and Jeffrey~R. Weeks,
  \emph{Snap{P}y, a computer program for studying the geometry and topology of
  $3$-manifolds}, Available at {\tt{http://\allowbreak snappy.\allowbreak
  computop.\allowbreak org}}, 2016.

\bibitem{CGLS:DehnSurgery}
Marc Culler, C.~McA. Gordon, J.~Luecke, and Peter~B. Shalen, \emph{Dehn surgery
  on knots}, Ann. of Math. (2) \textbf{125} (1987), no.~2, 237--300.
  \MR{881270}

\bibitem{CullerShalen:Varieties}
Marc Culler and Peter~B. Shalen, \emph{Varieties of group representations and
  splittings of {$3$}-manifolds}, Ann. of Math. (2) \textbf{117} (1983), no.~1,
  109--146. \MR{683804}

\bibitem{DimofteQRCS}
Tudor Dimofte, \emph{Quantum {R}iemann surfaces in {C}hern-{S}imons theory},
  Adv. Theor. Math. Phys. \textbf{17} (2013), no.~3, 479--599. \MR{3250765}

\bibitem{DvdV:Spectral}
Tudor Dimofte and Roland van~der Veen, \emph{A {S}pectral {P}erspective on
  {N}eumann-{Z}agier}, arxiv:1403.5215, 2014.

\bibitem{FockGoncharov06}
Vladimir Fock and Alexander Goncharov, \emph{Moduli spaces of local systems and
  higher {T}eichm\"{u}ller theory}, Publ. Math. Inst. Hautes \'{E}tudes Sci.
  (2006), no.~103, 1--211. \MR{2233852}

\bibitem{FominShapiroThurston08}
Sergey Fomin, Michael Shapiro, and Dylan Thurston, \emph{Cluster algebras and
  triangulated surfaces. {I}. {C}luster complexes}, Acta Math. \textbf{201}
  (2008), no.~1, 83--146. \MR{2448067}

\bibitem{FWZ:ClusterIntro}
Sergey Fomin, Lauren Williams, and Andrei Zelevinsky, \emph{Introduction to
  {C}luster {A}lgebras. {C}hapters 1--3}, arxiv:1608.05735, 2016.

\bibitem{FrohmanGelcaLofaro}
C.~Frohman, R.~Gelca, and W.~Lofaro, \emph{The {A}-polynomial from the
  noncommutative viewpoint}, Trans. Amer. Math. Soc. \textbf{354} (2002),
  no.~2, 735--747. \MR{1862565}

\bibitem{garoufalidis:AJ-conj}
S.~Garoufalidis, \emph{On the characteristic and deformation varieties of a
  knot}, Proceedings of the {C}asson {F}est, Geom. Topol. Monogr., vol.~7,
  Geom. Topol. Publ., Coventry, 2004, pp.~291--309 (electronic). \MR{2172488
  (2006j:57028)}

\bibitem{garoufalidisLe}
S.~Garoufalidis and T.~T.~Q. L{\^e}, \emph{The colored {J}ones function is
  {$q$}-holonomic}, Geom. Topol. \textbf{9} (2005), 1253--1293. \MR{MR2174266
  (2006j:57029)}

\bibitem{GaroufalidisMattman}
Stavros Garoufalidis and Thomas~W. Mattman, \emph{The {$A$}-polynomial of the
  {$(-2,3,3+2n)$} pretzel knots}, New York J. Math. \textbf{17} (2011),
  269--279. \MR{2811064}

\bibitem{GTZ}
Stavros Garoufalidis, Dylan~P. Thurston, and Christian~K. Zickert, \emph{The
  complex volume of {${\rm SL}(n,\mathbb{C})$}-representations of 3-manifolds},
  Duke Math. J. \textbf{164} (2015), no.~11, 2099--2160. \MR{3385130}

\bibitem{GSV05}
Michael Gekhtman, Michael Shapiro, and Alek Vainshtein, \emph{Cluster algebras
  and {W}eil-{P}etersson forms}, Duke Math. J. \textbf{127} (2005), no.~2,
  291--311. \MR{2130414}

\bibitem{GoernerZickert}
Matthias Goerner and Christian~K. Zickert, \emph{Triangulation independent
  {P}tolemy varieties}, Math. Z. \textbf{289} (2018), no.~1-2, 663--693.
  \MR{3803807}

\bibitem{GueritaudSchleimer}
Fran\c{c}ois Gu\'eritaud and Saul Schleimer, \emph{Canonical triangulations of
  {D}ehn fillings}, Geom. Topol. \textbf{14} (2010), no.~1, 193--242.
  \MR{2578304}

\bibitem{HamLee}
Ji-Young Ham and Joongul Lee, \emph{An explicit formula for the
  {$A$}-polynomial of the knot with {C}onway's notation {$C(2n,3)$}}, J. Knot
  Theory Ramifications \textbf{25} (2016), no.~10, 1650057, 9. \MR{3548475}

\bibitem{HikamiInoue}
Kazuhiro Hikami and Rei Inoue, \emph{Braids, complex volume and cluster
  algebras}, Algebr. Geom. Topol. \textbf{15} (2015), no.~4, 2175--2194.
  \MR{3402338}

\bibitem{HosteShanahan04}
Jim Hoste and Patrick~D. Shanahan, \emph{A formula for the {A}-polynomial of
  twist knots}, J. Knot Theory Ramifications \textbf{13} (2004), no.~2,
  193--209. \MR{2047468}

\bibitem{HMPT:WhiteheadSister}
Joshua~A. Howie, Daniel~V. Mathews, Jessica~S. Purcell, and Em~K. Thompson,
  \emph{A-polynomials of fillings of the {W}hitehead sister}, Internat. J.
  Math. \textbf{34} (2023), no.~13, Paper No. 2350085. \MR{4663827}

\bibitem{JacoRubinstein:LST}
William Jaco and Hyam Rubinstein, \emph{Layered triangulations of 3-manifolds},
  arXiv:math/0603601, 2006.

\bibitem{JacoRubinstein:0Eff}
William Jaco and J.~Hyam Rubinstein, \emph{{$0$}-efficient triangulations of
  3-manifolds}, J. Differential Geom. \textbf{65} (2003), no.~1, 61--168.
  \MR{2057531}

\bibitem{Mathews:A-poly_twist_knots_err}
Daniel~V. Mathews, \emph{Erratum: {A}n explicit formula for the {A}-polynomial
  of twist knots}, J. Knot Theory Ramifications \textbf{23} (2014), no.~11,
  1492001, 1. \MR{3293048}

\bibitem{Mathews:A-poly_twist_knots}
\bysame, \emph{An explicit formula for the {$A$}-polynomial of twist knots}, J.
  Knot Theory Ramifications \textbf{23} (2014), no.~9, 1450044, 5. \MR{3268980}

\bibitem{Neumann}
Walter~D. Neumann, \emph{Combinatorics of triangulations and the
  {C}hern-{S}imons invariant for hyperbolic {$3$}-manifolds}, Topology '90
  ({C}olumbus, {OH}, 1990), Ohio State Univ. Math. Res. Inst. Publ., vol.~1, de
  Gruyter, Berlin, 1992, pp.~243--271. \MR{1184415}

\bibitem{NeumannZagier}
Walter~D. Neumann and Don Zagier, \emph{Volumes of hyperbolic three-manifolds},
  Topology \textbf{24} (1985), no.~3, 307--332. \MR{815482}

\bibitem{Newman:IntegralMatrices}
Morris Newman, \emph{Integral matrices}, Academic Press, New York-London, 1972,
  Pure and Applied Mathematics, Vol. 45. \MR{0340283}

\bibitem{NiZhang}
Yi~Ni and Xingru Zhang, \emph{Detection of knots and a cabling formula for
  {$A$}-polynomials}, Algebr. Geom. Topol. \textbf{17} (2017), no.~1, 65--109.
  \MR{3604373}

\bibitem{Penner87}
R.~C. Penner, \emph{The decorated {T}eichm\"uller space of punctured surfaces},
  Comm. Math. Phys. \textbf{113} (1987), no.~2, 299--339. \MR{919235
  (89h:32044)}

\bibitem{Petersen:DoubleTwist}
Kathleen~L. Petersen, \emph{{$A$}-polynomials of a family of two-bridge knots},
  New York J. Math. \textbf{21} (2015), 847--881. \MR{3425625}

\bibitem{Segerman:Deformation}
Henry Segerman, \emph{A generalisation of the deformation variety}, Algebr.
  Geom. Topol. \textbf{12} (2012), no.~4, 2179--2244. \MR{3020204}

\bibitem{shalen:survey}
Peter~B. Shalen, \emph{Representations of 3-manifold groups}, Handbook of
  geometric topology, North-Holland, Amsterdam, 2002, pp.~955--1044.
  \MR{1886685}

\bibitem{TamuraYokota}
Naoko Tamura and Yoshiyuki Yokota, \emph{A formula for the {$A$}-polynomials of
  {$(-2,3,1+2n)$}-pretzel knots}, Tokyo J. Math. \textbf{27} (2004), no.~1,
  263--273. \MR{2060090}

\bibitem{thurston}
William~P. Thurston, \emph{The geometry and topology of three-manifolds},
  Princeton Univ. Math. Dept. Notes, 1979, Available at
  http://www.msri.org/communications/books/gt3m.

\bibitem{Thurston:3DGT}
\bysame, \emph{Three-dimensional geometry and topology. {V}ol. 1}, Princeton
  Mathematical Series, vol.~35, Princeton University Press, Princeton, NJ,
  1997, Edited by Silvio Levy. \MR{1435975}

\bibitem{Tran}
Anh~T. Tran, \emph{The {A}-polynomial 2-tuple of twisted {W}hitehead links},
  Internat. J. Math. \textbf{29} (2018), no.~2, 1850013, 14. \MR{3770936}

\bibitem{Weeks:Computation}
Jeff Weeks, \emph{Computation of hyperbolic structures in knot theory},
  Handbook of knot theory, Elsevier B. V., Amsterdam, 2005, pp.~461--480.
  \MR{2179268}

\bibitem{Williams:ClusterIntroduction}
Lauren~K. Williams, \emph{Cluster algebras: an introduction}, Bull. Amer. Math.
  Soc. (N.S.) \textbf{51} (2014), no.~1, 1--26. \MR{3119820}

\bibitem{Zickert:PtolemyDehnA-poly}
Christian~K. Zickert, \emph{Ptolemy coordinates, {D}ehn invariant and the
  {$A$}-polynomial}, Math. Z. \textbf{283} (2016), no.~1-2, 515--537.
  \MR{3489078}

\end{thebibliography}
\bibliographystyle{amsplain}

\end{document}